\documentclass[reqno]{amsart}

\usepackage{amsmath}
\usepackage{amssymb}
\usepackage{enumerate}
\usepackage[all]{xy}
\usepackage{amsthm}
\usepackage{amscd}
\usepackage{amsfonts}
\usepackage{latexsym}
\usepackage{amscd}
\usepackage{graphicx}
\usepackage{tikz}
\usepackage{tikz-cd}
\usepackage{setspace}
\usetikzlibrary{decorations.pathmorphing}

\definecolor{darkgreen}{HTML}{336633}
\definecolor{darkred}{HTML}{993333}
\makeatletter
\newcommand{\arxiv}[1]{{\tt arXiv:#1}}

\def\E{{\mathtt{E}}}
\def\D{{\mathtt{D}}}
\def\LC{{\mathcal L}}
\def\LL{{\mathbb L}}
\def\bd#1{\text{\boldmath${#1}$}}
\newcommand{\jontodo}{\todo[inline,color=green!20]}
\newcommand{\alextodo}{\todo[inline,color=red!20]}
\def\Sym{\operatorname{Sym}}
\def\SYM{\operatorname{SYM}}
\newtheorem{theorem}{Theorem}[section]
\newtheorem{lemma}[theorem]{Lemma}
\newtheorem{proposition}[theorem]{Proposition}
\newtheorem{corollary}[theorem]{Corollary} 
\theoremstyle{definition}  
\newtheorem{definition}[theorem]{Definition}
\newtheorem{example}[theorem]{Example}

\newtheorem{remark}[theorem]{Remark}

\def\bi{\text{\boldmath$i$}}
\def\bj{\text{\boldmath$j$}}

\newcommand{\sqbinom}[2]{\genfrac{[}{]}{0pt}{}{#1}{#2}}

\@addtoreset{equation}{section}
\def\theequation{\arabic{section}.\arabic{equation}}

\def\ii{\text{\i}}
\def\jj{\text{\j}}
\def\SCat{\mathcal{SC}at}
\def\SCAT{\mathfrak{SCat}}
\def\GSCat{\mathcal{GSC}at}
\def\GSCAT{\mathfrak{GSCat}}
\def\piSCat{\Pi\text{-}\mathcal{SC}at}
\def\piSCAT{\Pi\text{-}\mathfrak{SCat}}
\def\qpiGSCat{(Q,\Pi)\text{-}\mathcal{GSC}at}
\def\qpiGSCAT{(Q,\Pi)\text{-}\mathfrak{GSCat}}
\def\qpiCat{(Q,\Pi)\text{-}\mathcal{C}at}
\def\qpiCAT{(Q,\Pi)\text{-}\mathfrak{Cat}}
\def\piTSCat{\Pi\text{-2-}\mathcal{SC}at}
\def\piTSCAT{\Pi\text{-2-}\mathfrak{SCat}}
\def\Cat{\mathcal{C}at}
\def\CAT{\mathfrak{Cat}}
\def\piCat{\Pi\text{-}\mathcal{C}at}
\def\piCAT{\Pi\text{-}\mathfrak{Cat}}
\def\piTCat{\Pi\text{-2-}\mathcal{C}at}
\def\piTCAT{\Pi\text{-2-}\mathfrak{Cat}}
\def\SVec{\underline{\mathcal{SV}ec}}
\def\GSVec{\underline{\mathcal{GSV}ec}}
\newcommand{\GSMod}{\text{-}\mathcal{GS\!M}od}
\newcommand{\GSProj}{\text{-}\mathcal{GS\!P}roj}
\def\deg{\operatorname{deg}}
\def\ob{\operatorname{ob}}
\newcommand{\End}{\operatorname{End}}
\newcommand{\wt}{\operatorname{wt}}
\newcommand{\Aut}{\operatorname{End}} 
\newcommand{\CatHom}{\mathcal{H}om}
\def\H{\mathcal{H}}
\newcommand{\Hom}{\operatorname{Hom}} 
\newcommand{\HOM}{\operatorname{HOM}} 
\newcommand{\Mod}{\operatorname{-mod}}
\renewcommand{\mod}{\operatorname{-mod}_{l\!f}}
\newcommand{\Rep}{\operatorname{Rep}}
\newcommand{\id}{\text{id}}
\newcommand{\gl}{\mathfrak{gl}}
\newcommand{\g}{\mathfrak{g}}
\newcommand{\Z}{\mathbb{Z}}
\newcommand{\N}{\mathbb{N}}
\newcommand{\Q}{\mathbb{Q}}
\renewcommand{\k}{\Bbbk}
\newcommand{\eps}{\varepsilon}
\newcommand{\gr}{\operatorname{gr}}
\newcommand{\unit}{\mathds{1}}
\newcommand{\rev}{^{\operatorname{rev}}}
\newcommand{\rad}{{\operatorname{rad}}}
\newcommand{\SSeq}{\operatorname{Seq}}
\newcommand{\SSeqd}{\operatorname{Seqd}}
\def\la{\lambda}
\def\al{\alpha}
\def\be{\beta}
\def\ga{\gamma}
\def\Ga{\Gamma}
\def\De{\Delta}
\def\d{{\mathrm{d}}}
\def\e{{\mathrm{e}}}
\def\x{{\mathrm{x}}}
\def\y{{\mathrm{y}}}
\def\h{{\mathrm{h}}}
\def\U{{\mathcal U}}
\def\UU{{\mathfrak{U}}}
\def\A{\mathcal{A}}
\def\B{\mathcal{B}}
\def\C{\mathcal{C}}
\def\AA{\mathfrak{A}}
\def\BB{\mathfrak{B}}
\def\CC{\mathfrak{C}}
\def\T{\omega}
\def\boldU{\text{\boldmath$\mathcal U$}}
\def\wt{\operatorname{wt}}
\def\a{{\mathtt{a}}}
\def\b{{\mathtt{b}}}
\def\c{{\mathtt{c}}}
\def\d{{\mathtt{d}}}
\def\clubsuit{\diamondsuit}
\def\spadesuit{\diamondsuit}
\def\0{{\bar{0}}}
\def\1{{\bar{1}}}
\def\Id{1}

\allowdisplaybreaks

\begin{document}

\title[Super Kac-Moody 2-categories]{\boldmath Super Kac-Moody 2-categories}

\author[J. Brundan]{Jonathan Brundan}
\author[A. Ellis]{Alexander P. Ellis}
\address{Department of Mathematics,
University of Oregon, Eugene, OR 97403, USA}
\email{brundan@uoregon.edu, apellis@gmail.com}
\thanks{Research of J.B.
supported in part by NSF grant DMS-1161094.}

\begin{abstract}
We introduce generalizations of Kac-Moody 2-categories 
in which the quiver Hecke algebras 
of Khovanov, Lauda and Rouquier
are replaced by the
quiver Hecke superalgebras of
Kang, Kashiwara and Tsuchioka.
\end{abstract}

\maketitle  
\vspace{-4mm}
\section{Introduction}

\noindent{\bf Overview.} 
Kac-Moody 2-categories were introduced by Khovanov and Lauda \cite{KL3}
and Rouquier \cite{Rou}.
They have rapidly become accepted as fundamental objects
in representation theory, with intimate connections especially to quantum groups,
canonical bases and knot invariants.
Rouquier gave a seemingly different definition to Khovanov and Lauda:
\begin{itemize}
\item
Rouquier's presentation starts from generators and relations
for certain underlying quiver Hecke algebras, adjoins right duals of all the generating
1-morphisms, 
then imposes one more ``inversion relation'' at the level of
2-morphisms. 
\item
The Khovanov-Lauda presentation incorporates various
additional generating 2-morphisms, and extra relations including
biadjointness and cyclicity. These additional generators and
relations are useful for various applications, e.g.
they are needed in order to extract a candidate for a basis in each space of
2-morphisms. 
\end{itemize}
In \cite{B}, the first author has shown that the two
versions are actually equivalent.
The main purpose of this article is to extend the computations made in
\cite{B} to include {\em super Kac-Moody 2-categories}. 
We will define these shortly following Rouquier's approach,
starting from certain 
underlying {\em quiver Hecke superalgebras}
which were introduced already by Kang, Kashiwara and Tsuchioka \cite{KKT}.
For the quiver with one odd vertex, 
the quiver Hecke superalgebra is the {\em odd nilHecke algebra}
defined independently 
in \cite{EKL}; see also \cite[$\S$3.3]{Wang} which introduced the
closely related degenerate spin affine Hecke algebras. In
this case, a super analog of the Kac-Moody 2-category was defined and studied already
in \cite{EL}. 
We will work here in the setting of {\em 2-supercategories} following
\cite{BE}, since it leads to some conceptual simplifications 
compared to the approach of \cite{EL}.

\vspace{2mm}\noindent{\bf 2-Supercategories.}
We proceed to the definitions.
Fix once and for all a supercommutative ground ring $\k = \k_\0 \oplus
\k_1$. We are mainly interested in the situation that
$\k$ is a field concentrated in even parity.

\begin{definition}\label{defsuperspace}
A {\em superspace} is a $\Z/2$-graded $(\k,\k)$-bimodule
in which the left and right actions are related by $cv =
(-1)^{|c||v|}vc$; here and subsequently, $|x|$ denotes
the parity of a homogeneous vector in a superspace.
An {\em even linear map} between superspaces is a
parity-preserving $\k$-module homomorphism.
\end{definition}

Let $\SVec$ be the Abelian category of all (small) superspaces
and even linear maps.
It is a symmetric monoidal category with tensor functor
$$
-\otimes-:\SVec \times \SVec \rightarrow \SVec
$$ 
being
the usual tensor product over $\k$,
and symmetric
braiding 
defined on objects by 
$u \otimes v \mapsto (-1)^{|u||v|} v \otimes u$.
(Our notation here follows \cite{BE}: $\SVec$ is the underlying
category to the monoidal {\em super}category $\mathcal{SV}ec$ whose
morphisms are not necessarily homogeneous linear maps.)

\begin{definition}\label{defsupercat}
A {\em supercategory} means a $\SVec$-enriched category,
i.e. each morphism space is a superspace
and composition induces an even linear map.
A {\em superfunctor} 
between
supercategories is a $\SVec$-enriched functor, i.e. a functor 
$F:\A \rightarrow \B$ such that the
function $\Hom_{\A}(\lambda,\mu) \rightarrow
\Hom_{\B}(F\lambda, F\mu), f \mapsto Ff$
is an even linear map
for all $\lambda, \mu \in \ob\A$.
\end{definition}
 
Let $\SCat$ be the category of all (small) supercategories,
with morphisms being superfunctors.
 Given two supercategories $\A$ and $\B$, we
define $\A \boxtimes \B$ to be the supercategory
whose objects are ordered pairs $(\lambda,\mu)$ of objects of $\A$
and $\B$, respectively, and 
$$
\Hom_{\A \boxtimes \B}((\lambda,\mu), (\sigma,\tau)) :=
\Hom_{\A}(\lambda,\sigma) \otimes \Hom_{\B}(\mu,\tau).
$$
Composition in $\A \boxtimes \B$
is defined using the symmetric braiding in $\SVec$, so that
$(f \otimes g) \circ (h \otimes k) =
(-1)^{|g||h|}(f \circ h) \otimes (g \circ k)$.
Given superfunctors $F:\A \rightarrow \A'$
and $G:\B \rightarrow \B'$,
there is a superfunctor $F \boxtimes G:\A \boxtimes
\B \rightarrow \A' \boxtimes \B'$ sending $(\lambda,\mu) \mapsto (F
\lambda, G \mu)$ and $f \otimes g 
\mapsto Ff \otimes Gg$.
We have now defined a functor $$
- \boxtimes -:\SCat \times
\SCat \rightarrow \SCat
$$ which 
makes $\SCat$ into a monoidal category.

\begin{definition}\label{defsuper2cat}
A {\em 2-supercategory} is a category enriched in $\SCat$.
See also \cite[Definition 2.2]{BE}
for the definition of
a {\em 2-superfunctor}
between 2-supercategories.
\end{definition}

\begin{remark}\rm
In \cite[Definition 2.1]{BE}, the $2$-supercategories of
Definition~\ref{defsuper2cat} 
are called {\em strict} 2-supercategories. 
Since we will not encounter 
any 2-supercategories below that are not strict, we have
suppressed the adjective from the outset. On 
the other hand, we will occasionally meet 
2-superfunctors that are not strict.
\end{remark}

According to Definition~\ref{defsuper2cat}, 
for objects $\lambda,\mu$ in a 2-supercategory
$\AA$,
there is given a
supercategory $\mathcal{H}om_{\AA}(\lambda,\mu)$ of
morphisms from $\lambda$ to $\mu$.
Elements of $\Hom_{\AA}(\lambda,\mu) := \ob \mathcal{H}om_{\AA}(\lambda,\mu)$
are  {\em 1-morphisms} in $\AA$.
For 1-morphisms $F, G \in \Hom_{\AA}(\lambda,\mu)$, 
we also use the shorthand $\Hom_{\AA}(F,G)$ for the superspace
$\Hom_{\mathcal{H}om_{\AA}(\lambda,\mu)}(F,G)$.
Its elements are {\em 2-morphisms}.
We often represent $x \in \Hom_{\AA}(F,G)$
by the picture
\begin{equation}\label{Xpic}
\mathord{
\begin{tikzpicture}[baseline = 0]
	\draw[-,thick,darkred] (0.08,-.4) to (0.08,-.13);
	\draw[-,thick,darkred] (0.08,.4) to (0.08,.13);
      \draw[thick,darkred] (0.08,0) circle (4pt);
   \node at (0.08,0) {\color{darkred}$\scriptstyle{x}$};
   \node at (0.08,-.53) {$\scriptstyle{F}$};
   \node at (0.52,0) {$\scriptstyle{\lambda}.$};
   \node at (-0.32,0) {$\scriptstyle{\mu}$};
   \node at (0.08,.53) {$\scriptstyle{G}$};
\end{tikzpicture}
}
\end{equation}
The composition $y \circ x$ of $x$ with another 2-morphism
$y \in \Hom_{\AA}(G,H)$ is obtained by vertically stacking
pictures:
$$
\mathord{
\begin{tikzpicture}[baseline = 0]
	\draw[-,thick,darkred] (0.08,-.4) to (0.08,-.13);
	\draw[-,thick,darkred] (0.08,.4) to (0.08,.13);
	\draw[-,thick,darkred] (0.08,.77) to (0.08,.6);
	\draw[-,thick,darkred] (0.08,1.37) to (0.08,1.03);
      \draw[thick,darkred] (0.08,0) circle (4pt);
      \draw[thick,darkred] (0.08,.9) circle (4pt);
   \node at (0.08,0) {\color{darkred}$\scriptstyle{x}$};
   \node at (0.08,.91) {\color{darkred}$\scriptstyle{y}$};
   \node at (0.08,-.53) {$\scriptstyle{F}$};
   \node at (0.62,.5) {$\scriptstyle{\lambda}.$};
   \node at (-0.42,0.5) {$\scriptstyle{\mu}$};
   \node at (0.08,1.5) {$\scriptstyle{H}$};
   \node at (0.08,.5) {$\scriptstyle{G}$};
\end{tikzpicture}
}
$$
The composition law in $\AA$ gives a coherent family of superfunctors
$$
T_{\nu,\mu,\lambda}:\mathcal{H}om_{\AA}(\mu,\nu) \boxtimes
\mathcal{H}om_{\AA}(\lambda,\mu) \rightarrow \mathcal{H}om_{\AA}(\lambda,\nu)
$$ 
for $\lambda,\mu,\nu \in \ob\AA$. Given 2-morphisms 
$x:F \rightarrow H,
y:G \rightarrow K$ between 1-morphisms $F, H:\lambda \rightarrow\mu,
G,K:\mu \rightarrow\nu$,
we denote 
$T_{\nu,\mu,\lambda}(y\otimes x):T_{\nu,\mu,\lambda}(G, F) \rightarrow T_{\nu,\mu,\lambda}(K, H)$
simply by $yx:GF\rightarrow KH$,
and represent it by horizontally stacking pictures:
$$
\mathord{
\begin{tikzpicture}[baseline = 0]
	\draw[-,thick,darkred] (0.08,-.4) to (0.08,-.13);
	\draw[-,thick,darkred] (0.08,.4) to (0.08,.13);
      \draw[thick,darkred] (0.08,0) circle (4pt);
   \node at (0.08,0) {\color{darkred}$\scriptstyle{x}$};
   \node at (0.08,-.53) {$\scriptstyle{F}$};
   \node at (0.58,0) {$\scriptstyle{\lambda}.$};
   \node at (-0.37,0) {$\scriptstyle{\mu}$};
   \node at (0.08,.53) {$\scriptstyle{H}$};
	\draw[-,thick,darkred] (-.8,-.4) to (-.8,-.13);
	\draw[-,thick,darkred] (-.8,.4) to (-.8,.13);
      \draw[thick,darkred] (-.8,0) circle (4pt);
   \node at (-.8,0) {\color{darkred}$\scriptstyle{y}$};
   \node at (-.8,-.53) {$\scriptstyle{G}$};
   \node at (-1.22,0) {$\scriptstyle{\nu}$};
   \node at (-.8,.53) {$\scriptstyle{K}$};
\end{tikzpicture}
}
$$
When confusion seems unlikely, we will use the same notation for a $1$-morphism $F$
as for its identity $2$-morphism. 
With this convention, 
we have that $yH \circ Gx = yx = (-1)^{|x||y|} Kx \circ yF$, or in pictures:
$$
\mathord{
\begin{tikzpicture}[baseline = 0]
   \node at (0.08,-.53) {$\scriptstyle{F}$};
   \node at (0.58,0) {$\scriptstyle{\lambda}.$};
   \node at (-0.37,0) {$\scriptstyle{\mu}$};
   \node at (0.08,.53) {$\scriptstyle{H}$};
   \node at (-.8,-.53) {$\scriptstyle{G}$};
   \node at (-1.22,0) {$\scriptstyle{\nu}$};
   \node at (-.8,.53) {$\scriptstyle{K}$};
	\draw[-,thick,darkred] (0.08,-.4) to (0.08,-.23);
	\draw[-,thick,darkred] (0.08,.4) to (0.08,.03);
      \draw[thick,darkred] (0.08,-0.1) circle (4pt);
   \node at (0.08,-0.1) {\color{darkred}$\scriptstyle{x}$};
	\draw[-,thick,darkred] (-.8,-.4) to (-.8,-.03);
	\draw[-,thick,darkred] (-.8,.4) to (-.8,.23);
      \draw[thick,darkred] (-.8,0.1) circle (4pt);
   \node at (-.8,.1) {\color{darkred}$\scriptstyle{y}$};
\end{tikzpicture}
}
\quad=\quad
\mathord{
\begin{tikzpicture}[baseline = 0]
   \node at (0.08,-.53) {$\scriptstyle{F}$};
   \node at (0.58,0) {$\scriptstyle{\lambda}.$};
   \node at (-0.37,0) {$\scriptstyle{\mu}$};
   \node at (0.08,.53) {$\scriptstyle{H}$};
   \node at (-.8,-.53) {$\scriptstyle{G}$};
   \node at (-1.22,0) {$\scriptstyle{\nu}$};
   \node at (-.8,.53) {$\scriptstyle{K}$};
	\draw[-,thick,darkred] (0.08,-.4) to (0.08,-.13);
	\draw[-,thick,darkred] (0.08,.4) to (0.08,.13);
      \draw[thick,darkred] (0.08,0) circle (4pt);
   \node at (0.08,0) {\color{darkred}$\scriptstyle{x}$};
	\draw[-,thick,darkred] (-.8,-.4) to (-.8,-.13);
	\draw[-,thick,darkred] (-.8,.4) to (-.8,.13);
      \draw[thick,darkred] (-.8,0) circle (4pt);
   \node at (-.8,0) {\color{darkred}$\scriptstyle{y}$};
\end{tikzpicture}
}
\quad=\quad
(-1)^{|x||y|}\:
\mathord{
\begin{tikzpicture}[baseline = 0]
   \node at (0.08,-.53) {$\scriptstyle{F}$};
   \node at (0.58,0) {$\scriptstyle{\lambda}.$};
   \node at (-0.37,0) {$\scriptstyle{\mu}$};
   \node at (0.08,.53) {$\scriptstyle{H}$};
   \node at (-.8,-.53) {$\scriptstyle{G}$};
   \node at (-1.22,0) {$\scriptstyle{\nu}$};
   \node at (-.8,.53) {$\scriptstyle{K}$};
	\draw[-,thick,darkred] (0.08,-.4) to (0.08,-.03);
	\draw[-,thick,darkred] (0.08,.4) to (0.08,.23);
      \draw[thick,darkred] (0.08,0.1) circle (4pt);
   \node at (0.08,0.1) {\color{darkred}$\scriptstyle{x}$};
	\draw[-,thick,darkred] (-.8,-.4) to (-.8,-.23);
	\draw[-,thick,darkred] (-.8,.4) to (-.8,.03);
      \draw[thick,darkred] (-.8,-0.1) circle (4pt);
   \node at (-.8,-.1) {\color{darkred}$\scriptstyle{y}$};
\end{tikzpicture}
}.
$$
This identity is the {\em super interchange law}.
The presence of the sign 
here means that
a 2-supercategory is {\em not} a 2-category in the
usual sense. In particular, diagrams for
2-morphisms in 2-supercategories 
are only invariant under rectilinear 
isotopy modulo signs.
Consequently, care is needed with horizontal levels when working with
odd 2-morphisms diagrammatically: a more complicated diagram such as
$$
\quad{\scriptstyle\nu}\quad\!
\mathord{
\begin{tikzpicture}[baseline = 8]
   \node at (0.08,1.33) {$\scriptstyle{K}$};
   \node at (.08,-.53) {$\scriptstyle{G}$};
	\draw[-,thick,darkred] (0.08,-.4) to (0.08,-.13);
	\draw[-,thick,darkred] (0.08,.57) to (0.08,.13);
	\draw[-,thick,darkred] (0.08,1.17) to (0.08,.83);
      \draw[thick,darkred] (0.08,0) circle (4pt);
      \draw[thick,darkred] (0.08,.7) circle (4pt);
   \node at (0.08,0) {\color{darkred}$\scriptstyle{y}$};
   \node at (0.08,.71) {\color{darkred}$\scriptstyle{v}$};
\end{tikzpicture}
}
\!\quad{\scriptstyle\mu}\quad\!
\mathord{
\begin{tikzpicture}[baseline = 8]
   \node at (0.08,1.33) {$\scriptstyle{H}$};
   \node at (.08,-.53) {$\scriptstyle{F}$};
	\draw[-,thick,darkred] (0.08,-.4) to (0.08,-.13);
	\draw[-,thick,darkred] (0.08,.57) to (0.08,.13);
	\draw[-,thick,darkred] (0.08,1.17) to (0.08,.83);
      \draw[thick,darkred] (0.08,0) circle (4pt);
      \draw[thick,darkred] (0.08,.7) circle (4pt);
   \node at (0.08,0) {\color{darkred}$\scriptstyle{x}$};
   \node at (0.08,.71) {\color{darkred}$\scriptstyle{u}$};
\end{tikzpicture}
}
\!\quad{\scriptstyle\lambda}\quad
$$
should be interpreted by {\em first} composing horizontally {\em then}
composing vertically. The example just given represents $(v u)
\circ (yx)$ {\em not}  $(v \circ y) (u \circ x)$.

\vspace{2mm}\noindent{\bf Super Kac-Moody 2-categories.}
With these foundational definitions behind us, we are ready to
introduce the main object of study.
We need to fix some additional data:
\begin{itemize}
\item Let $I$ be a (possibly infinite) index set equipped with a parity function $I \rightarrow \Z / 2, i \mapsto |i|$; we will say
  that $i \in I$ is {\em even} or {\em odd} according to whether $|i| = \0$ or $\1$, respectively.
If $I$ has odd elements, we make the additional assumption that 2 is invertible in the ground ring $\k$.

\item Let $(-d_{ij})_{i,j \in I}$ be a generalized Cartan matrix,
so $d_{ii} = -2$, $d_{ij} \geq 0$ for $i \neq j$, and $d_{ij} = 0
\Leftrightarrow d_{ji} = 0$. 
We make the additional assumption that
\begin{equation}\label{a1}
\text{$|i| = \1\Rightarrow d_{ij}$ is even.}
\end{equation}
\item
Pick a complex vector space $\mathfrak{h}$
and
linearly independent subsets
$\{\alpha_i\:|\:i \in I\}$ and $\{h_i\:|\:i \in I\}$ of
$\mathfrak{h}^*$ and $\mathfrak{h}$, respectively,
such that 
$\langle h_i, \alpha_j\rangle = -d_{ij}$ for all $i,j \in I$.
Let $P := \{\lambda \in \mathfrak{h}^*\:|\:\langle h_i, \lambda\rangle
\in \Z\text{ for all }i \in I\}$ be the {\em weight
  lattice} and
$Q := \bigoplus_{i \in I} \Z \alpha_i$ be the {\em root lattice}.
\item
Let $\g$ be the {\em Kac-Moody algebra} associated to this data with
Chevalley generators $\{e_i, f_i, h_i\:|\:i \in I\}$ and Cartan
subalgebra $\mathfrak{h}$.
\item Finally
fix units $t_{ij} \in \k_\0^\times$ such that 
\begin{equation}\label{a2}
\text{$t_{ii} = 1,\quad
d_{ij}=0 \Rightarrow t_{ij} = t_{ji}$,}
\end{equation}
and
scalars $s_{ij}^{pq} \in \k_\0$ for $0 < p < d_{ij}$, $0 < q <
d_{ji}$ such that
\begin{equation}\label{a3}
\text{$s_{ij}^{pq} = s_{ji}^{qp}$, \quad $p |i| = \1 \Rightarrow s_{ij}^{pq} =0$.
}
\end{equation}
\end{itemize}
In case all elements of $I$ are even,
the following
is the same as the Rouquier's definition of 
Kac-Moody 2-category from \cite{Rou}
(viewing the latter as 
a 2-supercategory by declaring that all of its 2-morphisms are even).

\begin{definition}\label{def1}
The {\em Kac-Moody 2-supercategory}
is the 2-supercategory
 $\UU(\g)$
with objects $P$,
generating 
1-morphisms 
$E_i 1_\lambda:\lambda \rightarrow \lambda+\alpha_i$ and
$F_i 1_\lambda:\lambda \rightarrow \lambda-\alpha_i$ for each $i \in I$ and
$\lambda \in P$, and generating
2-morphisms
$x:E_i 1_\lambda \rightarrow E_i 1_\lambda$ of parity $|i|$,
$\tau:E_i E_j 1_\lambda \rightarrow E_j E_i 1_\lambda$
of parity $|i||j|$,
$\eta:1_\lambda \rightarrow F_i E_i 1_\lambda$ of parity $\0$ and
$\eps:E_i F_i 1_\lambda \rightarrow 1_\lambda$
of parity $\0$,
subject to certain relations.  
To record the relations among these generators, 
we switch to diagrams, representing the identity 2-morphisms
of $E_i 1_\lambda$ and $F_i 1_\lambda$ by
${\scriptstyle\substack{\lambda+\alpha_i\\\phantom{-}}}\substack{{\color{darkred}{\displaystyle\uparrow}} 
\\ {\scriptscriptstyle i}}{\scriptstyle\substack{\lambda\\\phantom{-}}}
$ and ${\scriptstyle\substack{\lambda-\alpha_i
\\\phantom{-}}}\substack{{\color{darkred}{\displaystyle\downarrow}} \\
  {\scriptscriptstyle
    i}}{\scriptstyle\substack{\lambda\\\phantom{-}}}$, respectively,
and the other generators by
\begin{align}\label{solid1}
&x 
= 
\mathord{
\begin{tikzpicture}[baseline = 0]
	\draw[->,thick,darkred] (0.08,-.3) to (0.08,.4);
      \node at (0.08,0.05) {$\color{darkred}\bullet$};
   \node at (0.08,-.4) {$\scriptstyle{i}$};
\end{tikzpicture}
}
{\scriptstyle\lambda}\:,
&&\tau
= 
\mathord{
\begin{tikzpicture}[baseline = 0]
	\draw[->,thick,darkred] (0.28,-.3) to (-0.28,.4);
	\draw[->,thick,darkred] (-0.28,-.3) to (0.28,.4);
   \node at (-0.28,-.4) {$\scriptstyle{i}$};
   \node at (0.28,-.4) {$\scriptstyle{j}$};
   \node at (.4,.05) {$\scriptstyle{\lambda}$};
\end{tikzpicture}
}\:,
&&
\eta
= 
\mathord{
\begin{tikzpicture}[baseline = 0]
	\draw[<-,thick,darkred] (0.4,0.3) to[out=-90, in=0] (0.1,-0.1);
	\draw[-,thick,darkred] (0.1,-0.1) to[out = 180, in = -90] (-0.2,0.3);
    \node at (-0.2,.4) {$\scriptstyle{i}$};
  \node at (0.3,-0.15) {$\scriptstyle{\lambda}$};
\end{tikzpicture}
}\:,
&&\eps
= 
\mathord{
\begin{tikzpicture}[baseline = 0]
	\draw[<-,thick,darkred] (0.4,-0.1) to[out=90, in=0] (0.1,0.3);
	\draw[-,thick,darkred] (0.1,0.3) to[out = 180, in = 90] (-0.2,-0.1);
    \node at (-0.2,-.2) {$\scriptstyle{i}$};
  \node at (0.3,0.4) {$\scriptstyle{\lambda}$};
\end{tikzpicture}
}.\\\nopagebreak
&\text{(parity $|i|$)}&
&\text{(parity $|i||j|$)}&
&\text{(parity $\0$)}&
&\text{(parity $\0$)}\notag
\end{align}
We denote the $n$th power of $x$ (under vertical composition) by 
\begin{align}\label{solid3}
&x^{\circ n}
= 
\mathord{
\begin{tikzpicture}[baseline = 0]
	\draw[->,thick,darkred] (0.08,-.3) to (0.08,.4);
      \node at (0.08,0.05) {$\color{darkred}\bullet$};
   \node at (0.08,-.4) {$\scriptstyle{i}$};
      \node at (-0.13,0.17) {$\color{darkred}\scriptstyle{n}$};
\end{tikzpicture}
}
{\scriptstyle\lambda}\:.\\
&\text{(parity $|i|n$)}\notag
\end{align}

First, we have the {\em quiver Hecke superalgebra relations} from \cite{KKT}:
\begin{align}
\mathord{
\begin{tikzpicture}[baseline = 0]
	\draw[->,thick,darkred] (0.28,.4) to[out=90,in=-90] (-0.28,1.1);
	\draw[->,thick,darkred] (-0.28,.4) to[out=90,in=-90] (0.28,1.1);
	\draw[-,thick,darkred] (0.28,-.3) to[out=90,in=-90] (-0.28,.4);
	\draw[-,thick,darkred] (-0.28,-.3) to[out=90,in=-90] (0.28,.4);
  \node at (-0.28,-.4) {$\scriptstyle{i}$};
  \node at (0.28,-.4) {$\scriptstyle{j}$};
   \node at (.43,.4) {$\scriptstyle{\lambda}$};
\end{tikzpicture}
}
=
\left\{
\begin{array}{ll}
0&\text{if $i=j$,}\\
t_{ij}\mathord{
\begin{tikzpicture}[baseline = 0]
	\draw[->,thick,darkred] (0.08,-.3) to (0.08,.4);
	\draw[->,thick,darkred] (-0.28,-.3) to (-0.28,.4);
   \node at (-0.28,-.4) {$\scriptstyle{i}$};
   \node at (0.08,-.4) {$\scriptstyle{j}$};
   \node at (.3,.05) {$\scriptstyle{\lambda}$};
\end{tikzpicture}
}&\text{if $d_{ij}=0$,}\\
 t_{ij}
\mathord{
\begin{tikzpicture}[baseline = 0]
	\draw[->,thick,darkred] (0.08,-.3) to (0.08,.4);
	\draw[->,thick,darkred] (-0.28,-.3) to (-0.28,.4);
   \node at (-0.28,-.4) {$\scriptstyle{i}$};
   \node at (0.08,-.4) {$\scriptstyle{j}$};
   \node at (.3,-.05) {$\scriptstyle{\lambda}$};
      \node at (-0.28,0.05) {$\color{darkred}\bullet$};
      \node at (-0.5,0.2) {$\color{darkred}\scriptstyle{d_{ij}}$};
\end{tikzpicture}
}
+
t_{ji}
\mathord{
\begin{tikzpicture}[baseline = 0]
	\draw[->,thick,darkred] (0.08,-.3) to (0.08,.4);
	\draw[->,thick,darkred] (-0.28,-.3) to (-0.28,.4);
   \node at (-0.28,-.4) {$\scriptstyle{i}$};
   \node at (0.08,-.4) {$\scriptstyle{j}$};
   \node at (.3,-.05) {$\scriptstyle{\lambda}$};
     \node at (0.08,0.05) {$\color{darkred}\bullet$};
     \node at (0.32,0.2) {$\color{darkred}\scriptstyle{d_{ji}}$};
\end{tikzpicture}
}
+\!\! \displaystyle\sum_{\substack{0 < p < d_{ij}\\0 < q <
    d_{ji}}} \!\!\!\!\!s_{ij}^{pq}
\mathord{
\begin{tikzpicture}[baseline = 0]
	\draw[->,thick,darkred] (0.08,-.3) to (0.08,.4);
	\draw[->,thick,darkred] (-0.28,-.3) to (-0.28,.4);
   \node at (-0.28,-.4) {$\scriptstyle{i}$};
   \node at (0.08,-.4) {$\scriptstyle{j}$};
   \node at (.3,-.05) {$\scriptstyle{\lambda}$};
      \node at (-0.28,0.05) {$\color{darkred}\bullet$};
      \node at (0.08,0.05) {$\color{darkred}\bullet$};
      \node at (-0.43,0.2) {$\color{darkred}\scriptstyle{p}$};
      \node at (0.22,0.2) {$\color{darkred}\scriptstyle{q}$};
\end{tikzpicture}
}
&\text{otherwise,}\\
\end{array}
\right. \label{now}
\end{align}
\begin{align}
\label{qha}
\mathord{
\begin{tikzpicture}[baseline = 0]
	\draw[<-,thick,darkred] (0.25,.6) to (-0.25,-.2);
	\draw[->,thick,darkred] (0.25,-.2) to (-0.25,.6);
  \node at (-0.25,-.26) {$\scriptstyle{i}$};
   \node at (0.25,-.26) {$\scriptstyle{j}$};
  \node at (.3,.25) {$\scriptstyle{\lambda}$};
      \node at (-0.13,-0.02) {$\color{darkred}\bullet$};
\end{tikzpicture}
}
- (-1)^{|i||j|}
\mathord{
\begin{tikzpicture}[baseline = 0]
	\draw[<-,thick,darkred] (0.25,.6) to (-0.25,-.2);
	\draw[->,thick,darkred] (0.25,-.2) to (-0.25,.6);
  \node at (-0.25,-.26) {$\scriptstyle{i}$};
   \node at (0.25,-.26) {$\scriptstyle{j}$};
  \node at (.3,.25) {$\scriptstyle{\lambda}$};
      \node at (0.13,0.42) {$\color{darkred}\bullet$};
\end{tikzpicture}
}
&=
\mathord{
\begin{tikzpicture}[baseline = 0]
 	\draw[<-,thick,darkred] (0.25,.6) to (-0.25,-.2);
	\draw[->,thick,darkred] (0.25,-.2) to (-0.25,.6);
  \node at (-0.25,-.26) {$\scriptstyle{i}$};
   \node at (0.25,-.26) {$\scriptstyle{j}$};
  \node at (.3,.25) {$\scriptstyle{\lambda}$};
      \node at (-0.13,0.42) {$\color{darkred}\bullet$};
\end{tikzpicture}
}
-(-1)^{|i||j|}
\mathord{
\begin{tikzpicture}[baseline = 0]
	\draw[<-,thick,darkred] (0.25,.6) to (-0.25,-.2);
	\draw[->,thick,darkred] (0.25,-.2) to (-0.25,.6);
  \node at (-0.25,-.26) {$\scriptstyle{i}$};
   \node at (0.25,-.26) {$\scriptstyle{j}$};
  \node at (.3,.25) {$\scriptstyle{\lambda}$};
      \node at (0.13,-0.02) {$\color{darkred}\bullet$};
\end{tikzpicture}
}
=
\delta_{i,j}
\mathord{
\begin{tikzpicture}[baseline = -5]
 	\draw[->,thick,darkred] (0.08,-.3) to (0.08,.4);
	\draw[->,thick,darkred] (-0.28,-.3) to (-0.28,.4);
   \node at (-0.28,-.4) {$\scriptstyle{i}$};
   \node at (0.08,-.4) {$\scriptstyle{j}$};
 \node at (.28,.06) {$\scriptstyle{\lambda}$};
\end{tikzpicture}
},
\end{align}
\begin{align}
\mathord{
\begin{tikzpicture}[baseline = 0]
	\draw[<-,thick,darkred] (0.45,.8) to (-0.45,-.4);
	\draw[->,thick,darkred] (0.45,-.4) to (-0.45,.8);
        \draw[-,thick,darkred] (0,-.4) to[out=90,in=-90] (-.45,0.2);
        \draw[->,thick,darkred] (-0.45,0.2) to[out=90,in=-90] (0,0.8);
   \node at (-0.45,-.45) {$\scriptstyle{i}$};
   \node at (0,-.45) {$\scriptstyle{j}$};
  \node at (0.45,-.45) {$\scriptstyle{k}$};
   \node at (.5,-.1) {$\scriptstyle{\lambda}$};
\end{tikzpicture}
}
\!-
\!\!\!
\mathord{
\begin{tikzpicture}[baseline = 0]
	\draw[<-,thick,darkred] (0.45,.8) to (-0.45,-.4);
	\draw[->,thick,darkred] (0.45,-.4) to (-0.45,.8);
        \draw[-,thick,darkred] (0,-.4) to[out=90,in=-90] (.45,0.2);
        \draw[->,thick,darkred] (0.45,0.2) to[out=90,in=-90] (0,0.8);
   \node at (-0.45,-.45) {$\scriptstyle{i}$};
   \node at (0,-.45) {$\scriptstyle{j}$};
  \node at (0.45,-.45) {$\scriptstyle{k}$};
   \node at (.5,-.1) {$\scriptstyle{\lambda}$};
\end{tikzpicture}
}
&=
\left\{
\begin{array}{ll}
\displaystyle
\sum_{\substack{r,s \geq 0 \\ r+s=d_{ij}-1}}
\!\!\!
(-1)^{|i|(|j|+s)}
t_{ij}
\!
\mathord{
\begin{tikzpicture}[baseline = 0]
	\draw[->,thick,darkred] (0.44,-.3) to (0.44,.4);
	\draw[->,thick,darkred] (0.08,-.3) to (0.08,.4);
	\draw[->,thick,darkred] (-0.28,-.3) to (-0.28,.4);
   \node at (-0.28,-.4) {$\scriptstyle{i}$};
   \node at (0.08,-.4) {$\scriptstyle{j}$};
   \node at (0.44,-.4) {$\scriptstyle{k}$};
  \node at (.6,-.1) {$\scriptstyle{\lambda}$};
     \node at (-0.28,0.05) {$\color{darkred}\bullet$};
     \node at (0.44,0.05) {$\color{darkred}\bullet$};
      \node at (-0.43,0.2) {$\color{darkred}\scriptstyle{r}$};
      \node at (0.55,0.2) {$\color{darkred}\scriptstyle{s}$};
\end{tikzpicture}
}\\
\quad\qquad+ \displaystyle
\!\!\sum_{\substack{0 < p < d_{ij}\\0 < q <
    d_{ji}\\r,s \geq 0\\r+s=p-1}}
\!\!\!\!
(-1)^{|i|(|j|+s)}
s_{ij}^{pq}
\mathord{
\begin{tikzpicture}[baseline = 0]
	\draw[->,thick,darkred] (0.44,-.3) to (0.44,.4);
	\draw[->,thick,darkred] (0.08,-.3) to (0.08,.4);
	\draw[->,thick,darkred] (-0.28,-.3) to (-0.28,.4);
   \node at (-0.28,-.4) {$\scriptstyle{i}$};
   \node at (0.08,-.4) {$\scriptstyle{j}$};
   \node at (0.44,-.4) {$\scriptstyle{k}$};
  \node at (.6,-.1) {$\scriptstyle{\lambda}$};
     \node at (-0.28,0.05) {$\color{darkred}\bullet$};
     \node at (0.44,0.05) {$\color{darkred}\bullet$};
      \node at (-0.43,0.2) {$\color{darkred}\scriptstyle{r}$};
     \node at (0.55,0.2) {$\color{darkred}\scriptstyle{s}$};
     \node at (0.08,0.05) {$\color{darkred}\bullet$};
      \node at (0.2,0.2) {$\color{darkred}\scriptstyle{q}$};
\end{tikzpicture}
}
&\text{if $i=k \neq j$,}\\
0&\text{otherwise.}
\end{array}\label{qhalast}
\right.\end{align}
In (\ref{now}), we have drawn multiple dots on the same horizontal
level, which is potentially ambiguous: our convention for this is that
it means the horizontal composition of $x^{\circ p}$ and $x^{\circ
  q}$,
so that
$$
\mathord{
\begin{tikzpicture}[baseline = 0]
	\draw[->,thick,darkred] (0.08,-.3) to (0.08,.4);
	\draw[->,thick,darkred] (-0.28,-.3) to (-0.28,.4);
   \node at (-0.28,-.4) {$\scriptstyle{i}$};
   \node at (0.08,-.4) {$\scriptstyle{j}$};
   \node at (.3,-.15) {$\scriptstyle{\lambda}$};
      \node at (-0.28,0.05) {$\color{darkred}\bullet$};
      \node at (0.08,0.05) {$\color{darkred}\bullet$};
      \node at (-0.43,0.2) {$\color{darkred}\scriptstyle{p}$};
      \node at (0.22,0.2) {$\color{darkred}\scriptstyle{q}$};
\end{tikzpicture}
}
:=
\mathord{
\begin{tikzpicture}[baseline = 0]
	\draw[->,thick,darkred] (0.08,-.3) to (0.08,.4);
	\draw[->,thick,darkred] (-0.28,-.3) to (-0.28,.4);
   \node at (-0.28,-.4) {$\scriptstyle{i}$};
   \node at (0.08,-.4) {$\scriptstyle{j}$};
   \node at (.3,-.25) {$\scriptstyle{\lambda}$};
      \node at (-0.28,0.15) {$\color{darkred}\bullet$};
      \node at (0.08,-0.05) {$\color{darkred}\bullet$};
      \node at (-0.43,0.3) {$\color{darkred}\scriptstyle{p}$};
      \node at (0.22,0.1) {$\color{darkred}\scriptstyle{q}$};
\end{tikzpicture}
}.
$$
Note further by the assumption (\ref{a3}) that
$$
s_{ij}^{pq} 
\mathord{
\begin{tikzpicture}[baseline = 0]
	\draw[->,thick,darkred] (0.08,-.3) to (0.08,.4);
	\draw[->,thick,darkred] (-0.28,-.3) to (-0.28,.4);
   \node at (-0.28,-.4) {$\scriptstyle{i}$};
   \node at (0.08,-.4) {$\scriptstyle{j}$};
   \node at (.3,-.25) {$\scriptstyle{\lambda}$};
      \node at (-0.28,0.15) {$\color{darkred}\bullet$};
      \node at (0.08,-0.05) {$\color{darkred}\bullet$};
      \node at (-0.43,0.3) {$\color{darkred}\scriptstyle{p}$};
      \node at (0.22,0.1) {$\color{darkred}\scriptstyle{q}$};
\end{tikzpicture}
}
=
s_{ij}^{pq} 
\mathord{
\begin{tikzpicture}[baseline = 0]
	\draw[->,thick,darkred] (0.08,-.3) to (0.08,.4);
	\draw[->,thick,darkred] (-0.28,-.3) to (-0.28,.4);
   \node at (-0.28,-.4) {$\scriptstyle{i}$};
   \node at (0.08,-.4) {$\scriptstyle{j}$};
   \node at (.3,-.05) {$\scriptstyle{\lambda}$};
      \node at (-0.28,-0.05) {$\color{darkred}\bullet$};
      \node at (0.08,0.15) {$\color{darkred}\bullet$};
      \node at (-0.43,0.1) {$\color{darkred}\scriptstyle{p}$};
      \node at (0.22,0.3) {$\color{darkred}\scriptstyle{q}$};
\end{tikzpicture}
}.
$$
Similar remarks apply to (\ref{qhalast}) and all other such situations
below.

Next we have the {\em right adjunction relations}:
\begin{align}\label{rightadj}
\mathord{
\begin{tikzpicture}[baseline = 0]
  \draw[->,thick,darkred] (0.3,0) to (0.3,.4);
	\draw[-,thick,darkred] (0.3,0) to[out=-90, in=0] (0.1,-0.4);
	\draw[-,thick,darkred] (0.1,-0.4) to[out = 180, in = -90] (-0.1,0);
	\draw[-,thick,darkred] (-0.1,0) to[out=90, in=0] (-0.3,0.4);
	\draw[-,thick,darkred] (-0.3,0.4) to[out = 180, in =90] (-0.5,0);
  \draw[-,thick,darkred] (-0.5,0) to (-0.5,-.4);
   \node at (-0.5,-.5) {$\scriptstyle{i}$};
   \node at (0.5,0) {$\scriptstyle{\lambda}$};
\end{tikzpicture}
}
&=
\mathord{\begin{tikzpicture}[baseline=0]
  \draw[->,thick,darkred] (0,-0.4) to (0,.4);
   \node at (0,-.5) {$\scriptstyle{i}$};
   \node at (0.2,0) {$\scriptstyle{\lambda}$};
\end{tikzpicture}
},\qquad
\mathord{
\begin{tikzpicture}[baseline = 0]
  \draw[->,thick,darkred] (0.3,0) to (0.3,-.4);
	\draw[-,thick,darkred] (0.3,0) to[out=90, in=0] (0.1,0.4);
	\draw[-,thick,darkred] (0.1,0.4) to[out = 180, in = 90] (-0.1,0);
	\draw[-,thick,darkred] (-0.1,0) to[out=-90, in=0] (-0.3,-0.4);
	\draw[-,thick,darkred] (-0.3,-0.4) to[out = 180, in =-90] (-0.5,0);
  \draw[-,thick,darkred] (-0.5,0) to (-0.5,.4);
   \node at (-0.5,.5) {$\scriptstyle{i}$};
   \node at (0.5,0) {$\scriptstyle{\lambda}$};
\end{tikzpicture}
}
=
\mathord{\begin{tikzpicture}[baseline=0]
  \draw[<-,thick,darkred] (0,-0.4) to (0,.4);
   \node at (0,.5) {$\scriptstyle{i}$};
   \node at (0.2,0) {$\scriptstyle{\lambda}$};
\end{tikzpicture}
}.
\end{align}
These imply that $F_i 1_{\lambda+\alpha_i}$ is
a right dual of $E_i 1_\lambda$.

Finally there are some {\em inversion relations}. To formulate these, we first
introduce a new 2-morphism
\begin{align}\label{sigrel}
&\mathord{
\begin{tikzpicture}[baseline = 0]
	\draw[<-,thick,darkred] (0.28,-.3) to (-0.28,.4);
	\draw[->,thick,darkred] (-0.28,-.3) to (0.28,.4);
   \node at (-0.28,-.4) {$\scriptstyle{j}$};
   \node at (-0.28,.5) {$\scriptstyle{i}$};
   \node at (.4,.05) {$\scriptstyle{\lambda}$};
\end{tikzpicture}
}
:=
\mathord{
\begin{tikzpicture}[baseline = 0]
	\draw[->,thick,darkred] (0.3,-.5) to (-0.3,.5);
	\draw[-,thick,darkred] (-0.2,-.2) to (0.2,.3);
        \draw[-,thick,darkred] (0.2,.3) to[out=50,in=180] (0.5,.5);
        \draw[->,thick,darkred] (0.5,.5) to[out=0,in=90] (0.8,-.5);
        \draw[-,thick,darkred] (-0.2,-.2) to[out=230,in=0] (-0.5,-.5);
        \draw[-,thick,darkred] (-0.5,-.5) to[out=180,in=-90] (-0.8,.5);
  \node at (-0.8,.6) {$\scriptstyle{i}$};
   \node at (0.28,-.6) {$\scriptstyle{j}$};
   \node at (1.05,.05) {$\scriptstyle{\lambda}$};
\end{tikzpicture}
}
:E_j F_i 1_\lambda \rightarrow F_i E_j 1_\lambda.\\
&\text{(parity $|i||j|$)}\notag
\end{align}
Then we require that the following (not necessarily homogeneous)
2-morphisms
are isomorphisms:
\begin{align}\label{inv1}
\mathord{
\begin{tikzpicture}[baseline = 0]
	\draw[<-,thick,darkred] (0.28,-.3) to (-0.28,.4);
	\draw[->,thick,darkred] (-0.28,-.3) to (0.28,.4);
   \node at (-0.28,-.4) {$\scriptstyle{j}$};
   \node at (-0.28,.5) {$\scriptstyle{i}$};
   \node at (.4,.05) {$\scriptstyle{\lambda}$};
\end{tikzpicture}
}
&:E_j F_i 1_\lambda \stackrel{\sim}{\rightarrow} F_i E_j 1_\lambda
&\text{if $i \neq j$,}\\
\label{inv2}
\mathord{
\begin{tikzpicture}[baseline = 0]
	\draw[<-,thick,darkred] (0.28,-.3) to (-0.28,.4);
	\draw[->,thick,darkred] (-0.28,-.3) to (0.28,.4);
   \node at (-0.28,-.4) {$\scriptstyle{i}$};
   \node at (-0.28,.5) {$\scriptstyle{i}$};
   \node at (.4,.05) {$\scriptstyle{\lambda}$};
\end{tikzpicture}
}
\oplus
\bigoplus_{n=0}^{\langle h_i,\lambda\rangle-1}
\mathord{
\begin{tikzpicture}[baseline = 0]
	\draw[<-,thick,darkred] (0.4,0) to[out=90, in=0] (0.1,0.4);
	\draw[-,thick,darkred] (0.1,0.4) to[out = 180, in = 90] (-0.2,0);
    \node at (-0.2,-.1) {$\scriptstyle{i}$};
  \node at (0.3,0.5) {$\scriptstyle{\lambda}$};
      \node at (-0.3,0.2) {$\color{darkred}\scriptstyle{n}$};
      \node at (-0.15,0.2) {$\color{darkred}\bullet$};
\end{tikzpicture}
}
&:
E_i F_i 1_\lambda\stackrel{\sim}{\rightarrow}
F_i E_i 1_\lambda \oplus 1_\lambda^{\oplus \langle h_i,\lambda\rangle}
&\text{if $\langle h_i,\lambda\rangle \geq
  0$},\\
\mathord{
\begin{tikzpicture}[baseline = 0]
	\draw[<-,thick,darkred] (0.28,-.3) to (-0.28,.4);
	\draw[->,thick,darkred] (-0.28,-.3) to (0.28,.4);
   \node at (-0.28,-.4) {$\scriptstyle{i}$};
   \node at (-0.28,.5) {$\scriptstyle{i}$};
   \node at (.4,.05) {$\scriptstyle{\lambda}$};
\end{tikzpicture}
}
\oplus
\bigoplus_{n=0}^{-\langle h_i,\lambda\rangle-1}
\mathord{
\begin{tikzpicture}[baseline = 0]
	\draw[<-,thick,darkred] (0.4,0.2) to[out=-90, in=0] (0.1,-.2);
	\draw[-,thick,darkred] (0.1,-.2) to[out = 180, in = -90] (-0.2,0.2);
    \node at (-0.2,.3) {$\scriptstyle{i}$};
  \node at (0.3,-0.25) {$\scriptstyle{\lambda}$};
      \node at (0.55,0) {$\color{darkred}\scriptstyle{n}$};
      \node at (0.38,0) {$\color{darkred}\bullet$};
\end{tikzpicture}
}
&
:E_i F_i 1_\lambda \oplus 
1_\lambda^{\oplus -\langle h_i,\lambda\rangle}
\stackrel{\sim}{\rightarrow}
 F_i E_i 1_\lambda&\text{if $\langle h_i,\lambda\rangle \leq
  0$}.\label{inv3}
\end{align}
Note that
(\ref{inv2})--(\ref{inv3}) are 2-morphisms in the additive envelope of
$\UU(\g)$.
Nevertheless this defines some genuine relations for
$\UU(\g)$ itself (rather than its additive envelope): we mean that there are some 
as yet unnamed
generating $2$-morphisms in $\UU(\g)$ which are
the matrix entries of
two-sided inverses to 
(\ref{inv2})--(\ref{inv3}).
\end{definition}

\vspace{2mm}
\noindent{\bf Second adjunction.}
Let
\begin{equation}
|i,\lambda| := |i| (\langle h_i,\lambda\rangle+1).
\end{equation}
Since $|i,\lambda| = |i,\lambda\pm\alpha_j|$ for
any $j \in I$, this only depends on the coset of $\lambda$ 
modulo $Q$.
In section~\ref{more}, we will define some additional 2-morphisms 
$\eta':1_\lambda \rightarrow E_i F_i 1_\lambda$ and $\eps':F_i E_i
1_\lambda \rightarrow 1_\lambda$
represented diagrammatically by leftward cups and caps:
\begin{align}\label{solid2}
&\eta' = 
\mathord{
\begin{tikzpicture}[baseline = 0]
	\draw[-,thick,darkred] (0.4,0.3) to[out=-90, in=0] (0.1,-0.1);
	\draw[->,thick,darkred] (0.1,-0.1) to[out = 180, in = -90] (-0.2,0.3);
    \node at (0.4,.4) {$\scriptstyle{i}$};
  \node at (0.3,-0.15) {$\scriptstyle{\lambda}$};
\end{tikzpicture}
}\:,
&&
\eps'=\mathord{
\begin{tikzpicture}[baseline = 0]
	\draw[-,thick,darkred] (0.4,-0.1) to[out=90, in=0] (0.1,0.3);
	\draw[->,thick,darkred] (0.1,0.3) to[out = 180, in = 90] (-0.2,-0.1);
    \node at (0.4,-.2) {$\scriptstyle{i}$};
  \node at (0.3,0.4) {$\scriptstyle{\lambda}$};
\end{tikzpicture}
}.\\
&\text{(parity }\text{$|i,\lambda|$)}&
&\text{(parity }\text{$|i,\lambda|$)}
\notag
\end{align}
Following the idea of \cite{BHLW}, we will normalize these
in a different way to \cite{CL, B},
in order to salvage some cyclicity.
Consequently, our definitions of $\eps'$ and $\eta'$ depend on the
additional choice of
units $c_{\lambda;i} \in \k_\0^\times$ for each $i \in I$ and
  $\lambda \in P$ such that 
\begin{equation}\label{a4}
c_{\lambda+\alpha_j;i} = t_{ij}
  c_{\lambda;i}.
\end{equation}
In section~\ref{sr4}, we will show that $\eta'$ and $\eps'$ satisfy the following {\em left adjunction relations}:
\begin{equation}
\mathord{
\begin{tikzpicture}[baseline = 0]
  \draw[-,thick,darkred] (0.3,0) to (0.3,-.4);
	\draw[-,thick,darkred] (0.3,0) to[out=90, in=0] (0.1,0.4);
	\draw[-,thick,darkred] (0.1,0.4) to[out = 180, in = 90] (-0.1,0);
	\draw[-,thick,darkred] (-0.1,0) to[out=-90, in=0] (-0.3,-0.4);
	\draw[-,thick,darkred] (-0.3,-0.4) to[out = 180, in =-90] (-0.5,0);
  \draw[->,thick,darkred] (-0.5,0) to (-0.5,.4);
   \node at (0.3,-.5) {$\scriptstyle{i}$};
   \node at (0.5,0) {$\scriptstyle{\lambda}$};
\end{tikzpicture}
}
=
(-1)^{|i,\lambda|}\mathord{\begin{tikzpicture}[baseline=0]
  \draw[->,thick,darkred] (0,-0.4) to (0,.4);
   \node at (0,-.5) {$\scriptstyle{i}$};
   \node at (0.2,0) {$\scriptstyle{\lambda}$};
\end{tikzpicture}
},
\qquad
\mathord{
\begin{tikzpicture}[baseline = 0]
  \draw[-,thick,darkred] (0.3,0) to (0.3,.4);
	\draw[-,thick,darkred] (0.3,0) to[out=-90, in=0] (0.1,-0.4);
	\draw[-,thick,darkred] (0.1,-0.4) to[out = 180, in = -90] (-0.1,0);
	\draw[-,thick,darkred] (-0.1,0) to[out=90, in=0] (-0.3,0.4);
	\draw[-,thick,darkred] (-0.3,0.4) to[out = 180, in =90] (-0.5,0);
  \draw[->,thick,darkred] (-0.5,0) to (-0.5,-.4);
   \node at (0.3,.5) {$\scriptstyle{i}$};
   \node at (0.5,0) {$\scriptstyle{\lambda}$};
\end{tikzpicture}
}
=
\mathord{\begin{tikzpicture}[baseline=0]
  \draw[<-,thick,darkred] (0,-0.4) to (0,.4);
   \node at (0,.5) {$\scriptstyle{i}$};
   \node at (0.2,0) {$\scriptstyle{\lambda}$};
\end{tikzpicture}
}.
\end{equation}
Consequently, 
$\Pi^{|i,\lambda|}F_i 1_{\lambda+\alpha_i}$ is a left dual of $E_i 1_\lambda$,
working now in the $\Pi$-envelope $\UU_\pi(\g)$ of $\UU(\g)$ from
 \cite[Definition 4.4]{BE}; cf.
Definition~\ref{ernie} below.

\vspace{2mm}
\noindent{\bf Further relations.}
In sections~\ref{sr1}--\ref{sr5}, we also derive various other relations
from the defining relations,
enough to see in particular that 
the inverses of the 2-morphisms
(\ref{inv1})--(\ref{inv3}) can be written as certain horizontal and
vertical compositions of
$x,\tau,\eps,\eta,\eps'$ and $\eta'$, i.e. 
the 2-morphisms named so far are enough to generate all other 
2-morphisms in $\UU(\g)$.
Some of our extra relations
are as follows.
\begin{itemize}
\item
The super analog of Lauda's {\em infinite Grassmannian relation}:
Let 
$\Sym$
be the algebra of symmetric functions over $\k$.
Recall $\Sym$ is generated both by the elementary symmetric functions $\e_r\:(r
\geq 0)$ and by the complete symmetric functions $\h_s\:(s \geq 0)$;
we view it as a superalgebra by declaring that all these generators are even.
By \cite[(I.2.6$'$)]{Mac},
elementary and complete symmetric functions are related by the equations 
$$
\e_0=\h_0 = 1,
\qquad
\sum_{r+s=n}(-1)^s \e_r \h_s = 0
\text{ for all $n > 0$}.
$$
Take $i \in I$, $\lambda \in P$ and set $h := \langle h_i,\lambda \rangle$.
If $i$ is even, Lauda \cite{L} observed already that there exists a
unique
homomorphism
\begin{equation}\label{beta1}
\beta_{\lambda;i}:\Sym \rightarrow
\operatorname{End}_{\UU(\g)}(1_\lambda)
\end{equation}
such that 
\begin{align*}
\qquad\quad \e_n
&\mapsto
c_{\lambda;i}^{-1}
\mathord{
\begin{tikzpicture}[baseline = 0]
  \draw[<-,thick,darkred] (0,0.4) to[out=180,in=90] (-.2,0.2);
  \draw[-,thick,darkred] (0.2,0.2) to[out=90,in=0] (0,.4);
 \draw[-,thick,darkred] (-.2,0.2) to[out=-90,in=180] (0,0);
  \draw[-,thick,darkred] (0,0) to[out=0,in=-90] (0.2,0.2);
 \node at (0,-.1) {$\scriptstyle{i}$};
   \node at (0.3,0.2) {$\scriptstyle{\lambda}$};
   \node at (-0.2,0.2) {$\color{darkred}\bullet$};
   \node at (-0.75,0.2) {$\color{darkred}\scriptstyle{n+h-1}$};
\end{tikzpicture}
}\text{ if $n> -h$,}
&\h_n
&\mapsto
(-1)^nc_{\lambda;i}\mathord{
\begin{tikzpicture}[baseline = 0]
  \draw[->,thick,darkred] (0.2,0.2) to[out=90,in=0] (0,.4);
  \draw[-,thick,darkred] (0,0.4) to[out=180,in=90] (-.2,0.2);
\draw[-,thick,darkred] (-.2,0.2) to[out=-90,in=180] (0,0);
  \draw[-,thick,darkred] (0,0) to[out=0,in=-90] (0.2,0.2);
 \node at (0,-.08) {$\scriptstyle{i}$};
   \node at (-0.3,0.2) {$\scriptstyle{\lambda}$};
   \node at (0.2,0.2) {$\color{darkred}\bullet$};
   \node at (0.75,0.2) {$\color{darkred}\scriptstyle{n-h-1}$};
\end{tikzpicture}
}\text{ if $n > h$,}
\end{align*}
bearing in mind the new normalization of bubbles.
The analog of this when $i$ is odd is as follows.
Let $\Sym[\d]$ be the supercommutative superalgebra obtained from
$\Sym$ 
by adjoining an odd generator $\d$ with $\d^2 = 0$.
Then there exists a unique 
homomorphism
\begin{equation}\label{beta2}
\beta_{\lambda;i}:\Sym[\d] \rightarrow
\operatorname{End}_{\UU(\g)}(1_\lambda)
\end{equation}
such that
\begin{align*}
\qquad\quad
\e_n
&\mapsto
c_{\lambda;i}^{-1}\!\mathord{
\begin{tikzpicture}[baseline = 0]
  \draw[<-,thick,darkred] (0,0.4) to[out=180,in=90] (-.2,0.2);
  \draw[-,thick,darkred] (0.2,0.2) to[out=90,in=0] (0,.4);
 \draw[-,thick,darkred] (-.2,0.2) to[out=-90,in=180] (0,0);
  \draw[-,thick,darkred] (0,0) to[out=0,in=-90] (0.2,0.2);
 \node at (0,-.1) {$\scriptstyle{i}$};
   \node at (0.3,0.2) {$\scriptstyle{\lambda}$};
   \node at (-0.2,0.2) {$\color{darkred}\bullet$};
   \node at (-0.8,0.2) {$\color{darkred}\scriptstyle{2n+h-1}$};
\end{tikzpicture}
}\text{ if $n > -\frac{h}{2}$,}\!&
\h_n
&\mapsto
(-1)^nc_{\lambda;i}\!\mathord{
\begin{tikzpicture}[baseline = 0]
  \draw[->,thick,darkred] (0.2,0.2) to[out=90,in=0] (0,.4);
  \draw[-,thick,darkred] (0,0.4) to[out=180,in=90] (-.2,0.2);
\draw[-,thick,darkred] (-.2,0.2) to[out=-90,in=180] (0,0);
  \draw[-,thick,darkred] (0,0) to[out=0,in=-90] (0.2,0.2);
 \node at (0,-.08) {$\scriptstyle{i}$};
   \node at (-0.3,0.2) {$\scriptstyle{\lambda}$};
   \node at (0.2,0.2) {$\color{darkred}\bullet$};
   \node at (0.8,0.2) {$\color{darkred}\scriptstyle{2n-h-1}$};
\end{tikzpicture}
}\text{ if $n > \frac{h}{2}$,}\\
\qquad\quad \d\e_n
&\mapsto
c_{\lambda;i}^{-1}\!\mathord{
\begin{tikzpicture}[baseline = 0]
  \draw[<-,thick,darkred] (0,0.4) to[out=180,in=90] (-.2,0.2);
  \draw[-,thick,darkred] (0.2,0.2) to[out=90,in=0] (0,.4);
 \draw[-,thick,darkred] (-.2,0.2) to[out=-90,in=180] (0,0);
  \draw[-,thick,darkred] (0,0) to[out=0,in=-90] (0.2,0.2);
 \node at (0,-.1) {$\scriptstyle{i}$};
   \node at (0.3,0.2) {$\scriptstyle{\lambda}$};
   \node at (-0.2,0.2) {$\color{darkred}\bullet$};
   \node at (-0.7,0.2) {$\color{darkred}\scriptstyle{2n+h}$};
\end{tikzpicture}
}\text{ if $n \geq -\frac{h}{2}$,}\!
&\d\h_n
&\mapsto
(-1)^nc_{\lambda;i}\!\mathord{
\begin{tikzpicture}[baseline = 0]
  \draw[->,thick,darkred] (0.2,0.2) to[out=90,in=0] (0,.4);
  \draw[-,thick,darkred] (0,0.4) to[out=180,in=90] (-.2,0.2);
\draw[-,thick,darkred] (-.2,0.2) to[out=-90,in=180] (0,0);
  \draw[-,thick,darkred] (0,0) to[out=0,in=-90] (0.2,0.2);
 \node at (0,-.08) {$\scriptstyle{i}$};
   \node at (-0.3,0.2) {$\scriptstyle{\lambda}$};
   \node at (0.2,0.2) {$\color{darkred}\bullet$};
   \node at (0.7,0.2) {$\color{darkred}\scriptstyle{2n-h}$};
\end{tikzpicture}
}\text{ if $n \geq \frac{h}{2}$}.
\end{align*}
Furthermore, letting
\begin{equation}\label{SYM}
\SYM :=
\bigotimes_{i\,\text{even}} \Sym \otimes
\bigotimes_{i\,\text{odd}}
\Sym[\d]
\end{equation}
where the tensor products are taken in some fixed order,
there is {\em surjective} homomorphism
\begin{equation}\label{beta}
\beta_\lambda : \SYM
\twoheadrightarrow \operatorname{End}_{\UU(\g)}(1_\lambda),
\end{equation}
defined by taking the product of the maps $\beta_{\lambda;i}$ applied to the
$i$th tensor factor of $\SYM$ for all $i \in I$.
\item 
{\em Centrality of odd bubbles}:
Assuming $i \in I$ is odd, we
introduce the odd 2-morphism
\begin{equation}
\mathord{\begin{tikzpicture}[baseline = 0]
  \draw[-,thick,darkred] (0,0.3) to[out=180,in=90] (-.2,0.1);
  \draw[-,thick,darkred] (0.2,0.1) to[out=90,in=0] (0,.3);
 \draw[-,thick,darkred] (-.2,0.1) to[out=-90,in=180] (0,-0.1);
  \draw[-,thick,darkred] (0,-0.1) to[out=0,in=-90] (0.2,0.1);
  \draw[-,thick,darkred] (0.14,-0.05) to (-0.14,0.23);
  \draw[-,thick,darkred] (0.14,0.23) to (-0.14,-0.05);
 \node at (0,-.18) {$\scriptstyle{i}$};
   \node at (-0.35,0.1) {$\scriptstyle{\lambda}$};
\end{tikzpicture}
}
:= \beta_{\lambda;i}(\d).
\end{equation} 
We call this the {\em odd bubble} of color $i$. 
By the super interchange law
it squares
to zero:
\begin{equation}
\label{obprops}
\left(\mathord{
\begin{tikzpicture}[baseline = 0]
  \draw[-,thick,darkred] (0,0.3) to[out=180,in=90] (-.2,0.1);
  \draw[-,thick,darkred] (0.2,0.1) to[out=90,in=0] (0,.3);
 \draw[-,thick,darkred] (-.2,0.1) to[out=-90,in=180] (0,-0.1);
  \draw[-,thick,darkred] (0,-0.1) to[out=0,in=-90] (0.2,0.1);
  \draw[-,thick,darkred] (0.14,-0.05) to (-0.14,0.23);
  \draw[-,thick,darkred] (0.14,0.23) to (-0.14,-0.05);
 \node at (-.28,.03) {$\scriptstyle{i}$};
   \node at (0.3,0.35) {$\scriptstyle{\lambda}$};
\end{tikzpicture}
}\right)^{2} = 0.
\end{equation}
We show moreover that odd bubbles
are central in $\UU(\g)$
in the sense that
\begin{equation}\label{centrality}
\mathord{
\begin{tikzpicture}[baseline = 0]
  \draw[-,thick,darkred] (0,0.2) to[out=180,in=90] (-.2,0);
  \draw[-,thick,darkred] (0.2,0) to[out=90,in=0] (0,.2);
 \draw[-,thick,darkred] (-.2,0) to[out=-90,in=180] (0,-0.2);
  \draw[-,thick,darkred] (0,-0.2) to[out=0,in=-90] (0.2,0);
  \draw[-,thick,darkred] (0.14,-0.15) to (-0.14,0.13);
  \draw[-,thick,darkred] (0.14,0.13) to (-0.14,-0.15);
 \node at (-.28,-.07) {$\scriptstyle{i}$};
   \node at (.5,-.5) {$\scriptstyle{j}$};
   \node at (.7,0) {$\scriptstyle{\lambda}$};
	\draw[->,thick,darkred] (.5,-.4) to (.5,.4);
\end{tikzpicture}
}=
\mathord{
\begin{tikzpicture}[baseline = 0]
  \draw[-,thick,darkred] (0,0.2) to[out=180,in=90] (-.2,0);
  \draw[-,thick,darkred] (0.2,0) to[out=90,in=0] (0,.2);
 \draw[-,thick,darkred] (-.2,0) to[out=-90,in=180] (0,-0.2);
  \draw[-,thick,darkred] (0,-0.2) to[out=0,in=-90] (0.2,0);
  \draw[-,thick,darkred] (0.14,-0.15) to (-0.14,0.13);
  \draw[-,thick,darkred] (0.14,0.13) to (-0.14,-0.15);
 \node at (-.28,-.07) {$\scriptstyle{i}$};
   \node at (-.5,-.5) {$\scriptstyle{j}$};
   \node at (.4,0) {$\scriptstyle{\lambda}$};
	\draw[->,thick,darkred] (-.5,-.4) to (-.5,.4);
\end{tikzpicture}
},
\qquad
\mathord{
\begin{tikzpicture}[baseline = 0]
  \draw[-,thick,darkred] (0,0.2) to[out=180,in=90] (-.2,0);
  \draw[-,thick,darkred] (0.2,0) to[out=90,in=0] (0,.2);
 \draw[-,thick,darkred] (-.2,0) to[out=-90,in=180] (0,-0.2);
  \draw[-,thick,darkred] (0,-0.2) to[out=0,in=-90] (0.2,0);
  \draw[-,thick,darkred] (0.14,-0.15) to (-0.14,0.13);
  \draw[-,thick,darkred] (0.14,0.13) to (-0.14,-0.15);
 \node at (-.28,-.07) {$\scriptstyle{i}$};
   \node at (.5,.5) {$\scriptstyle{j}$};
   \node at (.7,0) {$\scriptstyle{\lambda}$};
	\draw[<-,thick,darkred] (.5,-.4) to (.5,.4);
\end{tikzpicture}
}=
\mathord{
\begin{tikzpicture}[baseline = 0]
  \draw[-,thick,darkred] (0,0.2) to[out=180,in=90] (-.2,0);
  \draw[-,thick,darkred] (0.2,0) to[out=90,in=0] (0,.2);
 \draw[-,thick,darkred] (-.2,0) to[out=-90,in=180] (0,-0.2);
  \draw[-,thick,darkred] (0,-0.2) to[out=0,in=-90] (0.2,0);
  \draw[-,thick,darkred] (0.14,-0.15) to (-0.14,0.13);
  \draw[-,thick,darkred] (0.14,0.13) to (-0.14,-0.15);
 \node at (-.28,-.07) {$\scriptstyle{i}$};
   \node at (-.5,.5) {$\scriptstyle{j}$};
   \node at (.4,0) {$\scriptstyle{\lambda}$};
	\draw[<-,thick,darkred] (-.5,-.4) to (-.5,.4);
\end{tikzpicture}
}
\end{equation}
for all $j \in I$.
(This means that it would be reasonable to set odd bubbles to
zero, imposing
additional relations
$\mathord{\begin{tikzpicture}[baseline = 0]
  \draw[-,thick,darkred] (0,0.3) to[out=180,in=90] (-.2,0.1);
  \draw[-,thick,darkred] (0.2,0.1) to[out=90,in=0] (0,.3);
 \draw[-,thick,darkred] (-.2,0.1) to[out=-90,in=180] (0,-0.1);
  \draw[-,thick,darkred] (0,-0.1) to[out=0,in=-90] (0.2,0.1);
  \draw[-,thick,darkred] (0.14,-0.05) to (-0.14,0.23);
  \draw[-,thick,darkred] (0.14,0.23) to (-0.14,-0.05);
 \node at (0,-.18) {$\scriptstyle{i}$};
   \node at (-0.35,0.1) {$\scriptstyle{\lambda}$};
\end{tikzpicture}
}
=0$ 
for all odd $i \in I$ and $\lambda \in P$.)
\item
{\em Cyclicity properties}:
If $i$ is even
then
\begin{equation}
\mathord{
\begin{tikzpicture}[baseline = 0]
  \draw[->,thick,darkred] (0.3,0) to (0.3,-.4);
	\draw[-,thick,darkred] (0.3,0) to[out=90, in=0] (0.1,0.4);
	\draw[-,thick,darkred] (0.1,0.4) to[out = 180, in = 90] (-0.1,0);
	\draw[-,thick,darkred] (-0.1,0) to[out=-90, in=0] (-0.3,-0.4);
	\draw[-,thick,darkred] (-0.3,-0.4) to[out = 180, in =-90] (-0.5,0);
  \draw[-,thick,darkred] (-0.5,0) to (-0.5,.4);
   \node at (-0.1,0) {$\color{darkred}\bullet$};
  \node at (-0.5,.5) {$\scriptstyle{i}$};
  \node at (0.5,0) {$\scriptstyle{\lambda}$};
\end{tikzpicture}
}=\:
\mathord{
\begin{tikzpicture}[baseline = 0]
  \draw[-,thick,darkred] (0.3,0) to (0.3,.4);
	\draw[-,thick,darkred] (0.3,0) to[out=-90, in=0] (0.1,-0.4);
	\draw[-,thick,darkred] (0.1,-0.4) to[out = 180, in = -90] (-0.1,0);
	\draw[-,thick,darkred] (-0.1,0) to[out=90, in=0] (-0.3,0.4);
	\draw[-,thick,darkred] (-0.3,0.4) to[out = 180, in =90] (-0.5,0);
  \draw[->,thick,darkred] (-0.5,0) to (-0.5,-.4);
   \node at (0.3,.5) {$\scriptstyle{i}$};
   \node at (0.5,0) {$\scriptstyle{\lambda}$};
   \node at (-0.1,0) {$\color{darkred}\bullet$};
\end{tikzpicture}
}\,,
\end{equation}
i.e. even dots are cyclic.
However if $i$ is odd we have that
\begin{equation}
\mathord{
\begin{tikzpicture}[baseline = 0]
  \draw[->,thick,darkred] (0.3,0) to (0.3,-.4);
	\draw[-,thick,darkred] (0.3,0) to[out=90, in=0] (0.1,0.4);
	\draw[-,thick,darkred] (0.1,0.4) to[out = 180, in = 90] (-0.1,0);
	\draw[-,thick,darkred] (-0.1,0) to[out=-90, in=0] (-0.3,-0.4);
	\draw[-,thick,darkred] (-0.3,-0.4) to[out = 180, in =-90] (-0.5,0);
  \draw[-,thick,darkred] (-0.5,0) to (-0.5,.4);
   \node at (-0.1,0) {$\color{darkred}\bullet$};
  \node at (-0.5,.5) {$\scriptstyle{i}$};
  \node at (0.5,0) {$\scriptstyle{\lambda}$};
\end{tikzpicture}
}=
2\:
\mathord{\begin{tikzpicture}[baseline = 0]
  \draw[-,thick,darkred] (0,0.3) to[out=180,in=90] (-.2,0.1);
  \draw[-,thick,darkred] (0.2,0.1) to[out=90,in=0] (0,.3);
 \draw[-,thick,darkred] (-.2,0.1) to[out=-90,in=180] (0,-0.1);
  \draw[-,thick,darkred] (0,-0.1) to[out=0,in=-90] (0.2,0.1);
  \draw[-,thick,darkred] (0.14,-0.05) to (-0.14,0.23);
  \draw[-,thick,darkred] (0.14,0.23) to (-0.14,-0.05);
 \node at (0,-.18) {$\scriptstyle{i}$};
\end{tikzpicture}
}
\mathord{
\begin{tikzpicture}[baseline = 0]
	\draw[<-,thick,darkred] (0.08,-.4) to (0.08,.4);
     \node at (0.08,.5) {$\scriptstyle{i}$};
\end{tikzpicture}
}
{\scriptstyle\lambda}
-
\mathord{
\begin{tikzpicture}[baseline = 0]
  \draw[-,thick,darkred] (0.3,0) to (0.3,.4);
	\draw[-,thick,darkred] (0.3,0) to[out=-90, in=0] (0.1,-0.4);
	\draw[-,thick,darkred] (0.1,-0.4) to[out = 180, in = -90] (-0.1,0);
	\draw[-,thick,darkred] (-0.1,0) to[out=90, in=0] (-0.3,0.4);
	\draw[-,thick,darkred] (-0.3,0.4) to[out = 180, in =90] (-0.5,0);
  \draw[->,thick,darkred] (-0.5,0) to (-0.5,-.4);
   \node at (0.3,.5) {$\scriptstyle{i}$};
   \node at (0.5,0) {$\scriptstyle{\lambda}$};
   \node at (-0.1,0) {$\color{darkred}\bullet$};
\end{tikzpicture}
}\,.
\end{equation}
In all cases, crossings satisfy
\begin{equation}
\mathord{
\begin{tikzpicture}[baseline = 0]
\draw[->,thick,darkred] (1.3,.4) to (1.3,-1.2);
\draw[-,thick,darkred] (-1.3,-.4) to (-1.3,1.2);
\draw[-,thick,darkred] (.5,1.1) to [out=0,in=90] (1.3,.4);
\draw[-,thick,darkred] (-.35,.4) to [out=90,in=180] (.5,1.1);
\draw[-,thick,darkred] (-.5,-1.1) to [out=180,in=-90] (-1.3,-.4);
\draw[-,thick,darkred] (.35,-.4) to [out=-90,in=0] (-.5,-1.1);
\draw[-,thick,darkred] (.35,-.4) to [out=90,in=-90] (-.35,.4);
        \draw[-,thick,darkred] (-0.35,-.5) to[out=0,in=180] (0.35,.5);
        \draw[->,thick,darkred] (0.35,.5) to[out=0,in=90] (0.8,-1.2);
        \draw[-,thick,darkred] (-0.35,-.5) to[out=180,in=-90] (-0.8,1.2);
  \node at (-0.8,1.3) {$\scriptstyle{i}$};
   \node at (-1.3,1.3) {$\scriptstyle{j}$};
   \node at (1.55,0) {$\scriptstyle{\lambda}$};
\end{tikzpicture}
}=
\mathord{
\begin{tikzpicture}[baseline = 0]
\draw[->,thick,darkred] (-1.3,.4) to (-1.3,-1.2);
\draw[-,thick,darkred] (1.3,-.4) to (1.3,1.2);
\draw[-,thick,darkred] (-.5,1.1) to [out=180,in=90] (-1.3,.4);
\draw[-,thick,darkred] (.35,.4) to [out=90,in=0] (-.5,1.1);
\draw[-,thick,darkred] (.5,-1.1) to [out=0,in=-90] (1.3,-.4);
\draw[-,thick,darkred] (-.35,-.4) to [out=-90,in=180] (.5,-1.1);
\draw[-,thick,darkred] (-.35,-.4) to [out=90,in=-90] (.35,.4);
        \draw[-,thick,darkred] (0.35,-.5) to[out=180,in=0] (-0.35,.5);
        \draw[->,thick,darkred] (-0.35,.5) to[out=180,in=90] (-0.8,-1.2);
        \draw[-,thick,darkred] (0.35,-.5) to[out=0,in=-90] (0.8,1.2);
  \node at (0.8,1.3) {$\scriptstyle{j}$};
   \node at (1.3,1.3) {$\scriptstyle{i}$};
   \node at (1.55,0) {$\scriptstyle{\lambda}$};
\end{tikzpicture}
}.
\end{equation}
\end{itemize}

\vspace{2mm}
\noindent
{\bf Nondegeneracy Conjecture.}
Let $F, G:\lambda \rightarrow \mu$ be some 1-morphisms in $\UU(\g)$.
In section~\ref{s6}, we 
construct an explicit set $\big\{f(\sigma)
\:\big|\:\sigma \in \widehat{M}(F,G)\big\}$
of 2-morphisms which
generates
$\Hom_{\UU(\g)}(F,G)$
as a right $\SYM$-module;
here the action of $p \in \SYM$ is by horizontally composing on the
right with $\beta_\lambda(p)$.
This puts us in position to formulate the following conjecture, which is the appropriate generalization 
of the nondegeneracy condition
formulated by Khovanov and Lauda in \cite[$\S$3.2.3]{KL3};
for example, taking $F = G = 1_\lambda$, 
it implies that the homomorphism $\beta_\lambda$
from (\ref{beta}) is an isomorphism.

\vspace{2mm}

\noindent
\underline{Conjecture}:
{\em $\Hom_{\UU(\g)}(F, G)$
is a free $\SYM$-module with basis $\big\{f(\sigma)\:\big|\:\sigma \in \widehat{M}(F,G)\big\}$.}
\vspace{2mm}

\noindent
We cannot prove this at present.
We will discuss its signficance and some possible approaches to its proof
later on in the introduction.

\vspace{2mm}
\noindent
{\bf Gradings.}
By a {\em graded superspace}, we mean a superspace equipped 
with an additional $\Z$-grading $V = \bigoplus_{n \in \Z} V_n =
\bigoplus_{n \in \Z} V_{n,\0}\oplus V_{n,\1}$.
Let $\GSVec$ be the symmetric monoidal category of graded superspaces and
degree-preserving even linear maps.
Mimicking Definition~\ref{defsupercat}, 
a {\em graded supercategory} means a $\GSVec$-enriched category.
Let $\GSCat$ be the monoidal category of all (small) graded
supercategories.
Finally, mimicking Definition~\ref{defsuper2cat},
a {\em graded 2-supercategory} means a category enriched in
$\GSCat$.
Thus, it is a 2-supercategory
whose 2-morphism
spaces are graded
superspaces, and horizontal and vertical composition respect these gradings.
We will soon need the following universal
construction from \cite[Definition 6.10]{BE}:

\begin{definition}\label{ernie}
Suppose that $\AA$ is a graded 2-supercategory.
Its {\em $(Q,\Pi)$-envelope} $\AA_{q,\pi}$ is the graded 
2-supercategory with the same objects as $\AA$, 1-morphisms defined from
$$
\Hom_{\AA_{q,\pi}}(\lambda,\mu) := \left\{Q^m \Pi^a F\:\big|\:\text{for all $F
  \in \Hom_{\AA}(\lambda,\mu)$, $m \in \Z$ and $a \in \Z/2$}\right\}
$$
with the horizontal composition law
$(Q^n \Pi^b G)(Q^m \Pi^a F) := Q^{m+n} \Pi^{a+b} (GF)$,
and 2-morphisms defined from
$$
\Hom_{\AA_{q,\pi}}(Q^m \Pi^a F, Q^n \Pi^b G)
:= \left\{x^{n,b}_{m,a}\:\big|\:\text{for all }x \in \Hom_{\AA}(F,G)\right\}
$$
viewed
as a superspace with operations
$x^{n,b}_{m,a} + y^{n,b}_{m,a} := (x+y)^{n,b}_{m,a}$, $c
(x^{n,b}_{m,a}) := (cx)^{n,b}_{m,a}$ for $c \in \k$,
and grading
$\deg(x^{n,b}_{m,a}):=\deg(x)+n-m$,
$\big|x^{n,b}_{m,a}\big| := 
|x|+a+b$.
Representing
$x^{n,b}_{m,a}$
by the picture
$$
\mathord{
\begin{tikzpicture}[baseline = 0]
	\draw[-,thick,darkred] (0.08,-.4) to (0.08,-.13);
	\draw[-,thick,darkred] (0.08,.4) to (0.08,.13);
      \draw[thick,darkred] (0.08,0) circle (4pt);
   \node at (0.08,0) {\color{darkred}$\scriptstyle{x}$};
   \node at (0.45,0.03) {$\scriptstyle{\lambda}$};
   \node at (-0.32,0) {$\scriptstyle{\mu}$};
   \node at (0.05,0.58) {$\scriptstyle{G}$};
   \node at (0.05,-0.55) {$\scriptstyle{F}$};
\draw[-,thin,red](.4,-.4) to (-.24,-.4);
\draw[-,thin,red](.4,.4) to (-.24,.4);
\node at (.5,.4) {$\color{red}\scriptstyle b$};
\node at (.5,-.4) {$\color{red}\scriptstyle a$};
\node at (-.4,.4) {$\color{red}\scriptstyle n$};
\node at (-.4,-.4) {$\color{red}\scriptstyle m$};
\end{tikzpicture}
}
$$
for $x$ as in (\ref{Xpic}),
the vertical and horizontal composition laws for 2-morphisms 
in $\AA_{q,\pi}$ are defined in terms of the ones in $\AA$
as follows:
\begin{align}
\mathord{
\begin{tikzpicture}[baseline = -1.3]
	\draw[-,thick,darkred] (0.08,-.4) to (0.08,-.13);
	\draw[-,thick,darkred] (0.08,.4) to (0.08,.13);
      \draw[thick,darkred] (0.08,0) circle (4pt);
   \node at (0.08,0) {\color{darkred}$\scriptstyle{y}$};
\draw[-,thin,red](.4,-.4) to (-.24,-.4);
\draw[-,thin,red](.4,.4) to (-.24,.4);
\node at (.5,.4) {$\color{red}\scriptstyle c$};
\node at (.5,-.4) {$\color{red}\scriptstyle b$};
\node at (-.42,.4) {$\color{red}\scriptstyle n$};
\node at (-.42,-.4) {$\color{red}\scriptstyle m$};
\end{tikzpicture}
}
\:\circ
\mathord{
\begin{tikzpicture}[baseline =-1.3]
	\draw[-,thick,darkred] (0.08,-.4) to (0.08,-.13);
	\draw[-,thick,darkred] (0.08,.4) to (0.08,.13);
      \draw[thick,darkred] (0.08,0) circle (4pt);
   \node at (0.08,0) {\color{darkred}$\scriptstyle{x}$};
\draw[-,thin,red](.4,-.4) to (-.24,-.4);
\draw[-,thin,red](.4,.4) to (-.24,.4);
\node at (.5,.4) {$\color{red}\scriptstyle b$};
\node at (.5,-.4) {$\color{red}\scriptstyle a$};
\node at (-.42,.4) {$\color{red}\scriptstyle m$};
\node at (-.42,-.4) {$\color{red}\scriptstyle l$};
\end{tikzpicture}
}
&:=\!
\mathord{
\begin{tikzpicture}[baseline = 8]
	\draw[-,thick,darkred] (0.08,-.4) to (0.08,-.13);
	\draw[-,thick,darkred] (0.08,.57) to (0.08,.13);
	\draw[-,thick,darkred] (0.08,.83) to (0.08,1.1);
      \draw[thick,darkred] (0.08,0) circle (4pt);
      \draw[thick,darkred] (0.08,.7) circle (4pt);
   \node at (0.08,0) {\color{darkred}$\scriptstyle{x}$};
   \node at (0.08,.71) {\color{darkred}$\scriptstyle{y}$};
\draw[-,thin,red](.4,-.4) to (-.24,-.4);
\draw[-,thin,red](.4,1.1) to (-.24,1.1);
\node at (.5,1.1) {$\color{red}\scriptstyle c$};
\node at (.5,-.4) {$\color{red}\scriptstyle a$};
\node at (-.42,1.1) {$\color{red}\scriptstyle n$};
\node at (-.42,-.4) {$\color{red}\scriptstyle l$};
\end{tikzpicture}
},\\
\mathord{
\begin{tikzpicture}[baseline =-1.3]
	\draw[-,thick,darkred] (0.08,-.4) to (0.08,-.13);
	\draw[-,thick,darkred] (0.08,.4) to (0.08,.13);
      \draw[thick,darkred] (0.08,0) circle (4pt);
   \node at (0.08,0) {\color{darkred}$\scriptstyle{y}$};
\draw[-,thin,red](.4,-.4) to (-.24,-.4);
\draw[-,thin,red](.4,.4) to (-.24,.4);
\node at (.5,.4) {$\color{red}\scriptstyle d$};
\node at (.5,-.4) {$\color{red}\scriptstyle c$};
\node at (-.42,.4) {$\color{red}\scriptstyle n$};
\node at (-.42,-.4) {$\color{red}\scriptstyle m$};
\end{tikzpicture}
}
\,
\mathord{
\begin{tikzpicture}[baseline = -1.3]
	\draw[-,thick,darkred] (0.08,-.4) to (0.08,-.13);
	\draw[-,thick,darkred] (0.08,.4) to (0.08,.13);
      \draw[thick,darkred] (0.08,0) circle (4pt);
   \node at (0.08,0) {\color{darkred}$\scriptstyle{x}$};
\draw[-,thin,red](.4,-.4) to (-.24,-.4);
\draw[-,thin,red](.4,.4) to (-.24,.4);
\node at (.5,.4) {$\color{red}\scriptstyle b$};
\node at (.5,-.4) {$\color{red}\scriptstyle a$};
\node at (-.38,.4) {$\color{red}\scriptstyle l$};
\node at (-.38,-.4) {$\color{red}\scriptstyle k$};
\end{tikzpicture}
}
&:=
(-1)^{c|x|+b|y|+ac+bc}
\mathord{
\begin{tikzpicture}[baseline = -2]
	\draw[-,thick,darkred] (0.08,-.4) to (0.08,-.13);
	\draw[-,thick,darkred] (0.08,.4) to (0.08,.13);
      \draw[thick,darkred] (0.08,0) circle (4pt);
   \node at (0.08,0) {\color{darkred}$\scriptstyle{x}$};
	\draw[-,thick,darkred] (-.8,-.4) to (-.8,-.13);
	\draw[-,thick,darkred] (-.8,.4) to (-.8,.13);
      \draw[thick,darkred] (-.8,0) circle (4pt);
   \node at (-.8,0) {\color{darkred}$\scriptstyle{y}$};
\draw[-,thin,red](.45,-.4) to (-1.1,-.4);
\draw[-,thin,red](.45,.4) to (-1.1,.4);
\node at (.74,.4) {$\color{red}\scriptstyle b+d$};
\node at (.74,-.4) {$\color{red}\scriptstyle a+c$};
\node at (-1.45,.4) {$\color{red}\scriptstyle l+n$};
\node at (-1.45,-.4) {$\color{red}\scriptstyle k+m$};
\end{tikzpicture}
}.\label{web}
\end{align}
For each object $\lambda$,
there are distinguished 1-morphisms
$q_\lambda := Q^{1} \Pi^{\0} 1_\lambda$,
$q_\lambda^{-1} := Q^{-1} \Pi^\0 1_\lambda$ and $\pi_\lambda := Q^0
\Pi^{\1} 1_\lambda$
in $\End_{\AA_{q,\pi}}(\lambda)$.
Moreover, there are 2-isomorphisms
$\sigma_\lambda:q_\lambda \stackrel{\sim}{\rightarrow} 1_\lambda$,
$\bar\sigma_\lambda:q^{-1}_\lambda \stackrel{\sim}{\rightarrow}
1_\lambda$ and
$\zeta_\lambda:\pi_\lambda \stackrel{\sim}{\rightarrow} 1_\lambda$, all
induced by the identity 2-morphism $1_{1_\lambda}$.
These give the required structure maps to make $\AA_{q,\pi}$ into a {\em graded 
$(Q,\Pi)$-2-supercategory} in the sense of \cite[Definition 6.5]{BE}.
\end{definition}

Assume for the remainder of the introduction that 
the Cartan matrix $A$ is symmetrizable,
so that there exist positive integers $(d_i)_{i \in I}$ such that $d_i d_{ij} = d_j d_{ji}$ for all
$i,j \in I$.
Assume moreover that $\k$ is a field, and that the parameters chosen above satisfy the 
following {\em homogeneity condition}:
\begin{equation}\label{hc}
s_{ij}^{pq} \neq 0 
 \Rightarrow 
p d_{ji} + q d_{ij} = d_{ij}d_{ji}.
\end{equation}
Then we can put an additional $\Z$-grading on 
$\UU(\g)$ making it into a
graded 2-supercategory,
by declaring that the generators 
from (\ref{solid1}) and (\ref{solid2}) 
are of the degrees listed in the following table:
$$
\begin{array}{|c|c|c|c|c|c|}
\hline
x&\tau&\eta&\eps&\eta'&\eps' \\\hline
2d_i&d_i d_{ij}&d_i(1+\langle h_i,\lambda\rangle)
&d_i(1-\langle h_i,\lambda\rangle)
&d_i(1-\langle h_i,\lambda\rangle)
&d_i(1+\langle h_i,\lambda\rangle)\\\hline
\end{array}
$$
Let $\UU_{q,\pi}(\g)$ denote the $(Q,\Pi)$-envelope of $\UU(\g)$
in the sense of Definition~\ref{ernie}.
The {\em underlying 2-category}
$\underline{\UU}_{q,\pi}(\g)$
consists of the same objects and
1-morphisms as
$\UU_{q,\pi}(\g)$ but only its even 2-morphisms of degree zero.
Also let
$\dot{\underline{\UU}}_{q,\pi}(\g)$ be the idempotent completion of the
  additive envelope of $\underline{\UU}_{q,\pi}(\g)$.
Both of $\underline{\UU}_{q,\pi}(\g)$ 
and
$\dot{\underline{\UU}}_{q,\pi}(\g)$ are
{\em $(Q, \Pi)$-2-categories} in
the sense of \cite[Definition 6.14]{BE}.
In particular, they are equipped with distinguished
objects $q = (q_\lambda)$ and $\pi =
(\pi_\lambda)$ in their 
Drinfeld centers.

\vspace{2mm}
\noindent
{\bf \boldmath Relation to the Ellis-Lauda 2-category.}
Suppose that $\g$ is {\em odd $\mathfrak{sl}_2$}, i.e. $I$ is an odd
  singleton.
Then the 
2-category $\dot{\underline{\UU}}_{q,\pi}(\g)$
is 2-equivalent to the 2-category
introduced \cite{EL}.
We do not think that this is an
important result going forward, so we will only give a rough sketch of its proof in the next
paragraph.
Our new approach to the definition
seems to be both conceptually more satisfactory
and less prone to errors when working with the relations.
So our point of view really is that, henceforth, 
one should simply replace the object in \cite{EL} with the one here.

Briefly, the idea is simply to construct quasi-inverse 2-functors
between the Ellis-Lauda 2-category $\UU_{EL}$ and our
$\dot{\underline{\UU}}_{q,\pi}(\g)$
by verifying relations. 
Let us write simply $E, F$ and $h$ for $E_i, F_i$ and $h_i$
for the unique $i \in I$.
Also we take $d_i := 1$ and 
identify $P\leftrightarrow\Z$ so $\lambda \leftrightarrow \langle
h,\lambda \rangle$.
Then,
the appropriate 2-functor in the direction
$\UU_{EL} \rightarrow
\dot{\underline{\UU}}_{q,\pi}(\g)$ 
is the identity on the object set $P$.
It sends the generating 1-morphisms $\mathtt{E} 1_\lambda, \mathtt{F} 1_\lambda$ and
$\mathtt{\Pi} 1_\lambda$ from \cite[$\S$3.2.1]{EL} 
to our 1-morphisms $E 1_\lambda, \Pi^{\bar\lambda+\1} F
1_\lambda$ and $\pi_\lambda$, respectively. 
On the generating 2-morphisms from \cite[$\S$3.2.2]{EL}, it goes as follows:
\begin{align*}
\mathord{\begin{tikzpicture}[baseline = 0]
	\draw[-,thick,blue,dashed] (0,0.08) to [out=160,in=-90] (-0.3,.4);
	\draw[->,thick,blue] (0.08,-.3) to (0.08,.4);
      \node at (0.08,0.05) {$\color{blue}\bullet$};
   \node at (.35,.05) {$\scriptstyle{\lambda}$};
\end{tikzpicture}
}
&\mapsto\mathord{
\begin{tikzpicture}[baseline = 0]
	\draw[->,thick,darkred] (0.08,-.3) to (0.08,.4);
      \node at (0.08,0.05) {$\color{darkred}\bullet$};
   \node at (.35,.05) {$\scriptstyle{\lambda}$};
\draw[-,thin,red](.4,-.3) to (-.22,-.3);
\draw[-,thin,red](.4,.4) to (-.22,.4);
\node at (.55,-.3) {$\color{red}\scriptstyle \0$};
\node at (.55,.4) {$\color{red}\scriptstyle \1$};
\node at (-.37,.4) {$\color{red}\scriptstyle 0$};
\node at (-.37,-.3) {$\color{red}\scriptstyle 2$};
\end{tikzpicture}
}
&
\mathord{
\begin{tikzpicture}[baseline = 0]
	\draw[-,thick,blue,dashed] (0,0.04) to [out=160,in=-90] (-0.53,.4);
	\draw[->,thick,blue] (0.28,-.3) to (-0.28,.4);
	\draw[->,thick,blue] (-0.28,-.3) to (0.28,.4);
   \node at (.4,.05) {$\scriptstyle{\lambda}$};
\end{tikzpicture}
}
&\mapsto
\mathord{
\begin{tikzpicture}[baseline = 0]
	\draw[->,thick,darkred] (0.28,-.3) to [out=90,in=-90] (-0.28,.4);
	\draw[->,thick,darkred] (-0.28,-.3) to 
[out=90,in=-90] (0.28,.4);
   \node at (.4,.05) {$\scriptstyle{\lambda}$};
\draw[-,thin,red](.4,-.3) to (-.4,-.3);
\draw[-,thin,red](.4,.4) to (-.4,.4);
\node at (.55,-.3) {$\color{red}\scriptstyle \0$};
\node at (.55,.4) {$\color{red}\scriptstyle \1$};
\node at (-.55,.4) {$\color{red}\scriptstyle 0$};
\node at (-.65,-.3) {$\color{red}\scriptstyle -2$};
\end{tikzpicture}
},
&
\mathord{
\begin{tikzpicture}[baseline = 0]
	\draw[-,thick,dashed,blue] (0.28,-.3) to (-0.28,.4);
	\draw[-,dashed,thick,blue] (-0.28,-.3) to (0.28,.4);
   \node at (.4,.05) {$\scriptstyle{\lambda}$};
\end{tikzpicture}
}
&\mapsto
-\mathord{
\begin{tikzpicture}[baseline = 0]
 \node at (0.05,0.05) {$\scriptstyle{\lambda}$};
\draw[-,thin,red](.25,-.25) to (-.25,-.25);
\draw[-,thin,red](.25,.35) to (-.25,.35);
\node at (.4,-.25) {$\color{red}\scriptstyle \0$};
\node at (.4,.35) {$\color{red}\scriptstyle \0$};
\node at (-.4,.35) {$\color{red}\scriptstyle 0$};
\node at (-.4,-.25) {$\color{red}\scriptstyle 0$};
\end{tikzpicture}
},
\\
\mathord{\begin{tikzpicture}[baseline = 0]
	\draw[-,thick,blue,dashed] (0.1,0.05) to [out=20,in=-90] (0.4,.4);
	\draw[<-,thick,blue] (0.08,-.3) to (0.08,.4);
      \node at (0.08,0.05) {$\color{blue}\bullet$};
   \node at (.35,-.05) {$\scriptstyle{\lambda}$};
\end{tikzpicture}
}
&\mapsto
(-1)^{\lambda+1}\!\!\!\!\mathord{
\begin{tikzpicture}[baseline = 0]
	\draw[-,thick,darkred] (0.28,.05) to (0.28,.5);
	\draw[->,thick,darkred] (-0.28,-.05) to (-0.28,-.4);
	\draw[-,thick,darkred] (-0.28,-.05) to [out=90,in=180](-0.14,.3);
	\draw[-,thick,darkred] (0,.05) to [out=90,in=0](-0.14,.3);
	\draw[-,thick,darkred] (0,.05) to [out=-90,in=180](0.14,-.2);
	\draw[-,thick,darkred] (0.28,.05) to [out=-90,in=0](0.14,-.2);
      \node at (0,0.05) {$\color{darkred}\bullet$};
   \node at (.45,.05) {$\scriptstyle{\lambda}$};
\draw[-,thin,red](.4,-.4) to (-.43,-.4);
\draw[-,thin,red](.4,.5) to (-.43,.5);
\node at (.52,.5) {$\color{red}\scriptstyle \bar\lambda$};
\node at (.68,-.4) {$\color{red}\scriptstyle \bar\lambda+\1$};
\node at (-.52,.5) {$\color{red}\scriptstyle 0$};
\node at (-.52,-.4) {$\color{red}\scriptstyle 2$};
\end{tikzpicture}
}\!\!\!\!,\!\!\!\!
&
\mathord{
\begin{tikzpicture}[baseline = 0]
	\draw[-,thick,blue,dashed] (0,0.04) to [out=20,in=-90] (0.53,.4);
	\draw[<-,thick,blue] (0.28,-.3) to (-0.28,.4);
	\draw[<-,thick,blue] (-0.28,-.3) to (0.28,.4);
   \node at (.4,-.05) {$\scriptstyle{\lambda}$};
\end{tikzpicture}
}
&\mapsto
-\mathord{
\begin{tikzpicture}[baseline = 0]
	\draw[-,thick,darkred] (0.3,.5) to (0.3,0);
	\draw[-,thick,darkred] (0.5,.5) to (0.5,-.05);
\draw[-,thick,darkred] (-0.5,.15) to [out=90,in=180] (-0.25,.42);
\draw[-,thick,darkred] (0.1,.2) to [out=90,in=0] (-0.25,.42);
\draw[-,thick,darkred] (0.5,-.05) to [out=-90,in=0] (0.25,-.32);
\draw[-,thick,darkred] (-0.1,-.1) to [out=-90,in=180] (0.25,-.32);
\draw[-,thick,darkred] (0.3,0) to [out=-90,in=0] (0.2,-.1);
\draw[-,thick,darkred] (-0.3,.1) to [out=90,in=180] (-0.2,.2);
\draw[-,thick,darkred] (0,.05) to [out=120,in=0] (-0.2,.2);
\draw[-,thick,darkred] (0,.05) to [out=-60,in=180] (0.2,-.1);
	\draw[<-,thick,darkred] (-0.3,-.4) to (-0.3,.1);
	\draw[<-,thick,darkred] (-0.5,-.4) to (-0.5,.2);
	\draw[-,thick,darkred] (-0.1,-.1) to [out=90,in=-90] (0.1,.2);
   \node at (.68,.05) {$\scriptstyle{\lambda}$};
\draw[-,thin,red](.6,-.4) to (-.6,-.4);
\draw[-,thin,red](.6,.5) to (-.6,.5);
\node at (.7,-.4) {$\color{red}\scriptstyle \0$};
\node at (.7,.5) {$\color{red}\scriptstyle \1$};
\node at (-.7,.5) {$\color{red}\scriptstyle 0$};
\node at (-.82,-.4) {$\color{red}\scriptstyle -2$};
\end{tikzpicture}
},\!\!
&
\mathord{
\begin{tikzpicture}[baseline = 0]
	\draw[-,thick,dashed,blue] (0.28,-.3) to (-0.28,.4);
	\draw[->,thick,blue] (-0.28,-.3) to (0.28,.4);
   \node at (.4,.05) {$\scriptstyle{\lambda}$};
\end{tikzpicture}
}
&\mapsto
\mathord{
\begin{tikzpicture}[baseline = 0]
 \node at (0.25,0.05) {$\scriptstyle{\lambda}$};
	\draw[->,thick,darkred] (0,-.25) to (0,.35);
\draw[-,thin,red](.25,-.25) to (-.25,-.25);
\draw[-,thin,red](.25,.35) to (-.25,.35);
\node at (.4,-.25) {$\color{red}\scriptstyle \1$};
\node at (.4,.35) {$\color{red}\scriptstyle \1$};
\node at (-.4,.35) {$\color{red}\scriptstyle 0$};
\node at (-.4,-.25) {$\color{red}\scriptstyle 0$};
\end{tikzpicture}
},
\\
\mathord{
\begin{tikzpicture}[baseline = 0]
	\draw[-,thick,dashed,blue] (0.28,-.3) to (-0.28,.4);
	\draw[<-,thick,blue] (-0.28,-.3) to (0.28,.4);
   \node at (.4,.05) {$\scriptstyle{\lambda}$};
\end{tikzpicture}
}
&\mapsto
\mathord{
\begin{tikzpicture}[baseline = 0]
 \node at (0.25,0.05) {$\scriptstyle{\lambda}$};
	\draw[<-,thick,darkred] (0,-.25) to (0,.35);
\draw[-,thin,red](.25,-.25) to (-.25,-.25);
\draw[-,thin,red](.25,.35) to (-.25,.35);
\node at (.4,-.25) {$\color{red}\scriptstyle \bar\lambda$};
\node at (.4,.35) {$\color{red}\scriptstyle \bar\lambda$};
\node at (-.4,.35) {$\color{red}\scriptstyle 0$};
\node at (-.4,-.25) {$\color{red}\scriptstyle 0$};
\end{tikzpicture}
},
&
\mathord{
\begin{tikzpicture}[baseline = 0]
	\draw[-,thick,dashed,blue] (-0.28,-.3) to (0.28,.4);
	\draw[<-,thick,blue] (0.28,-.3) to (-0.28,.4);
   \node at (.4,.05) {$\scriptstyle{\lambda}$};
\end{tikzpicture}
}
&\mapsto
\mathord{
\begin{tikzpicture}[baseline = 0]
 \node at (0.25,0.05) {$\scriptstyle{\lambda}$};
	\draw[<-,thick,darkred] (0,-.25) to (0,.35);
\draw[-,thin,red](.25,-.25) to (-.25,-.25);
\draw[-,thin,red](.25,.35) to (-.25,.35);
\node at (.4,-.25) {$\color{red}\scriptstyle \bar\lambda$};
\node at (.4,.35) {$\color{red}\scriptstyle \bar\lambda$};
\node at (-.4,.35) {$\color{red}\scriptstyle 0$};
\node at (-.4,-.25) {$\color{red}\scriptstyle 0$};
\end{tikzpicture}
},
&
\mathord{
\begin{tikzpicture}[baseline = 0]
	\draw[-,thick,dashed,blue] (-0.28,-.3) to (0.28,.4);
	\draw[->,thick,blue] (0.28,-.3) to (-0.28,.4);
   \node at (.4,.05) {$\scriptstyle{\lambda}$};
\end{tikzpicture}
}
&\mapsto
\mathord{
\begin{tikzpicture}[baseline = 0]
 \node at (0.25,0.05) {$\scriptstyle{\lambda}$};
	\draw[->,thick,darkred] (0,-.25) to (0,.35);
\draw[-,thin,red](.25,-.25) to (-.25,-.25);
\draw[-,thin,red](.25,.35) to (-.25,.35);
\node at (.4,-.25) {$\color{red}\scriptstyle \1$};
\node at (.4,.35) {$\color{red}\scriptstyle \1$};
\node at (-.4,.35) {$\color{red}\scriptstyle 0$};
\node at (-.4,-.25) {$\color{red}\scriptstyle 0$};
\end{tikzpicture}
},
\\
\mathord{
\begin{tikzpicture}[baseline = 0]
	\draw[-,thick,blue] (0.4,-0.3) to[out=90, in=0] (0.1,0.1);
	\draw[->,thick,blue] (0.1,0.1) to[out = 180, in = 90] (-0.2,-0.3);
 \node at (0.55,0) {$\scriptstyle{\lambda}$};
\end{tikzpicture}
}
&\mapsto
\mathord{
\begin{tikzpicture}[baseline = 0]
	\draw[-,thick,darkred] (0.4,-0.3) to[out=90, in=0] (0.1,0.1);
	\draw[->,thick,darkred] (0.1,0.1) to[out = 180, in = 90] (-0.2,-0.3);
 \node at (0.5,0.05) {$\scriptstyle{\lambda}$};
\draw[-,thin,red](.55,-.3) to (-.4,-.3);
\draw[-,thin,red](.55,.3) to (-.4,.3);
\node at (.9,-.3) {$\color{red}\scriptstyle \bar\lambda+\1$};
\node at (.7,.3) {$\color{red}\scriptstyle \0$};
\node at (-.55,.3) {$\color{red}\scriptstyle 0$};
\node at (-.73,-.3) {$\color{red}\scriptstyle \lambda+1$};
\end{tikzpicture}
}
,
&
\mathord{
\begin{tikzpicture}[baseline = 0]
	\draw[-,thick,blue] (0.4,0.3) to[out=-90, in=0] (0.1,-0.1);
	\draw[->,thick,blue] (0.1,-0.1) to[out = 180, in = -90] (-0.2,0.3);
 \node at (0.55,0.1) {$\scriptstyle{\lambda}$};
\end{tikzpicture}
}
&\mapsto
\mathord{
\begin{tikzpicture}[baseline = 0]
	\draw[-,thick,darkred] (0.4,0.3) to[out=-90, in=0] (0.1,-0.1);
	\draw[->,thick,darkred] (0.1,-0.1) to[out = 180, in = -90] (-0.2,0.3);
 \node at (0.5,0) {$\scriptstyle{\lambda}$};
\draw[-,thin,red](.55,-.3) to (-.4,-.3);
\draw[-,thin,red](.55,.3) to (-.4,.3);
\node at (.9,.3) {$\color{red}\scriptstyle \bar\lambda+\1$};
\node at (.7,-.3) {$\color{red}\scriptstyle \0$};
\node at (-.55,.3) {$\color{red}\scriptstyle 0$};
\node at (-.73,-.3) {$\color{red}\scriptstyle 1-\lambda$};
\end{tikzpicture}
},
&
\mathord{
\begin{tikzpicture}[baseline = 0]
	\draw[-,dashed,thick,blue] (0.4,-0.3) to[out=90, in=0] (0.1,0.1);
	\draw[-,dashed,thick,blue] (-0.2,-0.3) to[in = 180, out = 90] (0.1,0.1);
 \node at (0.55,0) {$\scriptstyle{\lambda}$};
\end{tikzpicture}
}
&\mapsto
\mathord{
\begin{tikzpicture}[baseline = 0]
 \node at (0.05,0.05) {$\scriptstyle{\lambda}$};
\draw[-,thin,red](.25,-.3) to (-.25,-.3);
\draw[-,thin,red](.25,.3) to (-.25,.3);
\node at (.4,-.3) {$\color{red}\scriptstyle \0$};
\node at (.4,.3) {$\color{red}\scriptstyle \0$};
\node at (-.4,.3) {$\color{red}\scriptstyle 0$};
\node at (-.4,-.3) {$\color{red}\scriptstyle 0$};
\end{tikzpicture}
},\\
\mathord{
\begin{tikzpicture}[baseline = 0]
	\draw[<-,thick,blue] (0.4,-0.3) to[out=90, in=0] (0.1,0.1);
	\draw[-,thick,blue] (0.1,0.1) to[out = 180, in = 90] (-0.2,-0.3);
 \node at (0.55,0) {$\scriptstyle{\lambda}$};
        \draw[-,dashed,thick,blue] (.13,.4) to (.13,.1);
        \draw[-,dashed,thick,blue] (.07,.4) to (.07,.1);
 \node at (0.1,0.5) {$\color{blue}\scriptscriptstyle{\lambda-1}$};
\end{tikzpicture}
}
&\mapsto
\mathord{
\begin{tikzpicture}[baseline = 0]
	\draw[<-,thick,darkred] (0.4,-0.3) to[out=90, in=0] (0.1,0.1);
	\draw[-,thick,darkred] (0.1,0.1) to[out = 180, in = 90] (-0.2,-0.3);
 \node at (0.5,0.05) {$\scriptstyle{\lambda}$};
\draw[-,thin,red](.55,-.3) to (-.4,-.3);
\draw[-,thin,red](.55,.3) to (-.4,.3);
\node at (.9,-.3) {$\color{red}\scriptstyle \bar\lambda+\1$};
\node at (.9,.3) {$\color{red}\scriptstyle \bar\lambda+\1$};
\node at (-.55,.3) {$\color{red}\scriptstyle 0$};
\node at (-.73,-.3) {$\color{red}\scriptstyle 1-\lambda$};
\end{tikzpicture}
}
,
&
\mathord{
\begin{tikzpicture}[baseline = 0]
	\draw[<-,thick,blue] (0.4,0.3) to[out=-90, in=0] (0.1,-0.1);
	\draw[-,thick,blue] (0.1,-0.1) to[out = 180, in = -90]
        (-0.2,0.3);
        \draw[-,dashed,thick,blue] (.13,.3) to (.13,-.1);
        \draw[-,dashed,thick,blue] (.07,.3) to (.07,-.1);
 \node at (0.1,0.4) {$\color{blue}\scriptscriptstyle{\lambda+1}$};
 \node at (0.55,0.1) {$\scriptstyle{\lambda}$};
\end{tikzpicture}
}
&\mapsto
\mathord{
\begin{tikzpicture}[baseline = 0]
	\draw[<-,thick,darkred] (0.4,0.3) to[out=-90, in=0] (0.1,-0.1);
	\draw[-,thick,darkred] (0.1,-0.1) to[out = 180, in = -90] (-0.2,0.3);
 \node at (0.5,0) {$\scriptstyle{\lambda}$};
\draw[-,thin,red](.55,-.3) to (-.4,-.3);
\draw[-,thin,red](.55,.3) to (-.4,.3);
\node at (.7,.3) {$\color{red}\scriptstyle \0$};
\node at (.7,-.3) {$\color{red}\scriptstyle \0$};
\node at (-.55,.3) {$\color{red}\scriptstyle 0$};
\node at (-.73,-.3) {$\color{red}\scriptstyle \lambda+1$};
\end{tikzpicture}
},
&
\mathord{
\begin{tikzpicture}[baseline = 0]
	\draw[-,dashed,thick,blue] (0.4,0.3) to[out=-90, in=0] (0.1,-0.1);
	\draw[-,dashed,thick,blue] (-0.2,0.3) to[in = 180, out = -90] (0.1,-0.1);
 \node at (0.55,0) {$\scriptstyle{\lambda}$};
\end{tikzpicture}
}
&\mapsto
\mathord{
\begin{tikzpicture}[baseline = 0]
 \node at (0.05,0.05) {$\scriptstyle{\lambda}$};
\draw[-,thin,red](.25,-.3) to (-.25,-.3);
\draw[-,thin,red](.25,.3) to (-.25,.3);
\node at (.4,-.3) {$\color{red}\scriptstyle \0$};
\node at (.4,.3) {$\color{red}\scriptstyle \0$};
\node at (-.4,.3) {$\color{red}\scriptstyle 0$};
\node at (-.4,-.3) {$\color{red}\scriptstyle 0$};
\end{tikzpicture}
}.
\end{align*}
We leave it to the reader to compare the relations in \cite{EL} with
our relations, and to construct a quasi-inverse 2-functor in the other
direction. In fact, when doing this carefully, one uncovers some
inconsistencies in the relations of \cite{EL};
e.g. the relation \cite[(3.1)]{EL} is wrong in the case $\lambda = 0$
(due to an error in the last sentence of the proof of \cite[Lemma
5.1]{EL} related to the nilpotency of the odd bubble).

\vspace{2mm}
\noindent
{\bf Decategorification Conjecture.}
Recall finally that the {\em Grothendieck ring} of an additive
2-category
$\AA$
is 
\begin{equation}\label{thegg}
K_0(\AA) := \bigoplus_{\lambda,\mu \in \ob \AA}
K_0(\mathcal{H}om_{\AA}(\lambda,\mu))
\end{equation}
where the $K_0$ on the right hand side is the usual split Grothendieck
group of the additive category
$\mathcal{H}om_{\AA}(\lambda,\mu)$.
It is a locally unital ring with distinguished idempotents
$\{1_\lambda\:|\:\lambda \in \ob \AA\}$.
If $\AA$ is a $(Q,\Pi)$-2-category, then
$K_0(\AA)$
is also linear over
$\LC := \Z[q,q^{-1},\pi] / (\pi^2-1)$, with $q$ and $\pi$ acting by
multiplication by the classes
of the
distinguished objects $q$ and $\pi$
of the Drinfeld center.

This discussion applies in particular to the $(Q,\Pi)$-2-category
$\dot{\underline{\UU}}_{q,\pi}(\g)$,
so that 
$K_0(\dot{\underline{\UU}}_{q,\pi}(\g))$
is a locally unital $\LC$-algebra with idempotents
$\{1_\lambda\:|\:\lambda \in P\}$.
Also let $\dot{U}_{q,\pi}(\g)_\LC$ be the $\LC$-form of the 
idempotented version of the
covering quantized enveloping algebra 
associated to $\g$
introduced by Clark, Hill and Wang in \cite{CHW1}; see section~\ref{sqg}.
By similar arguments to those of \cite{KL3}, using also some results
from \cite{HW}, we will show in section~\ref{sonto} that there is a {\em surjective}
homomorphism of locally unital $\LC$-algebras
\begin{equation}\label{gamma}
\gamma:\dot{U}_{q,\pi}(\g)_\LC \twoheadrightarrow K_0(\dot{\underline{\UU}}_{q,\pi}(\g))
\end{equation}
sending $e_i 1_\lambda$ and $f_i
1_\lambda$ to $[E_i 1_\lambda]$ and $[F_i 1_\lambda]$, respectively.
Moreover, also just like in \cite{KL3}, we will show in section 12 that
the Nondegeneracy
Conjecture formulated above, 
together with an additional assumption of bar-consistency on the Cartan datum, 
implies the truth of the following:

\vspace{2mm}
\noindent
\underline{Conjecture}:
{\em $\gamma$ is an isomorphism.}

\vspace{2mm}

\noindent
{\bf Discussion.}
In the purely even case, i.e. when all $i \in I$ are even,
the Nondegeneracy Conjecture (hence, the Decategorification Conjecture)
was established 
by Khovanov and Lauda in \cite[$\S$6.4]{KL3} in case
$\g = \mathfrak{sl}_n$.
In \cite{erratum},
Webster has proposed a proof of the
Nondegeneracy Conjecture for all purely even types.
There is also a 
completely different proof of the
Decategorification Conjecture
based on results of \cite{KK},
which is valid in all finite types; see e.g. \cite[Corollary 4.21]{BD}.

Turning to the odd case, the Decategorification Conjecture for odd
$\mathfrak{sl}_2$
is proved in \cite[Theorem 8.4]{EL}.
The only additional finite type possibilities 
come from {\em odd $\mathfrak{b}_n$}, i.e. type 
$\mathfrak{b}_n$ with the element of $I$ corresponding to the short simple
root chosen to be odd. 
For these, the Decategorification Conjecture 
may be deduced from \cite{KKO1,KKO2}.
We hope that Webster's methods from \cite{erratum}
can be extended to the super case to prove the Nondegeneracy
Conjecture in general, but there is a great deal of work still to be
done in order to see this through. As a first step, we would like to
see the proof of the Nondegeneracy Conjecture
from \cite{KL3} extended in order to
include all odd $\mathfrak{b}_n$, and hope to address this in
subsequent work.

Assuming the Decategorification Conjecture,
one gets an interesting basis for the covering quantum group
$\dot{U}_{q,\pi}(g)_\LC$ coming from the isomorphism classes of the indecomposable objects of 
$\dot{\underline{\UU}}_{q,\pi}(\g)$.
In symmetric types, this 
should coincide (up to parity shift) with the canonical basis from \cite[Theorem
4.14]{Clark}.
For odd $\mathfrak{b}_1$, this assertion follows already from the
results of \cite{EL}.

In a different direction, it should now be possible to develop super analogs of many of
the foundational structural results proved by Chuang-Rouquier and
Rouquier in \cite{CR, Rou}. 
Various applications, e.g. to spin representations of symmetric
groups and to representations of the Lie superalgebra
$\mathfrak{q}(n)$, are expected.

\section{More generators}\label{more}

In sections 2--8, we 
assume that the ground ring 
$\k$ is as in Definition~\ref{defsuperspace}, and let $\UU(\g)$ be the
Kac-Moody 2-supercategory from Definition~\ref{def1}.
We begin by defining various additional 2-morphisms in
$\UU(\g)$. 
 
\begin{definition}
We have the {\em downward dots and crossings}, which are
the right mates of the upward dots and crossings:
\begin{align}\label{xp}
&
\mathord{

}.\\
\notag
&\text{(parity $|i|n$)}
\end{align}
The sign in (\ref{solid4}) is easily checked using the diagrammatics; see also
\cite[Proposition 7.14]{KKO2}.
Using (\ref{rightadj}) and (\ref{sigrel}), 
we deduce:
\begin{align}
\mathord{
%
}.\label{rightpitchfork2}
\end{align}
\end{definition}

\begin{definition}\label{def3}
We define the {\em leftward crossing}
and various {\em leftward cups and caps}.
First define
\begin{align}
&\mathord{
%
}
:F_i E_i 1_\lambda \rightarrow 1_\lambda,\\
&\text{(parity $|i||j|$)}&&
\text{(parity $|i|n$)}&&
\text{(parity $|i|n$)}\notag
\end{align}
by declaring that 
\begin{align}\label{rory}
\mathord{
%
}
\bigg)^{-1}
&&
\text{if $\langle h_i,\lambda \rangle \leq 0$,}
\end{align}
working in the additive envelope of $\UU(\g)$.
Then, remembering the scalars $c_{\lambda;i}$ chosen for (\ref{a4}), we set
\begin{align}
\eta'=\mathord{

}
&:=
\left\{
\begin{array}{rl}
c_{\lambda;i}
\mathord{
\begin{tikzpicture}[baseline = 0]
	\draw[-,thick,darkred] (0.4,0.4) to[out=-90, in=0] (0.1,0);
	\draw[->,thick,darkred] (0.1,0) to[out = 180, in = -90] (-0.2,0.4);
    \node at (0.4,.5) {$\scriptstyle{i}$};
  \node at (0.5,0.15) {$\scriptstyle{\lambda}$};
      \node at (0.2,-0.21) {$\color{darkred}\scriptstyle{\langle h_i,\lambda\rangle-1}$};
      \node at (0.12,0.01) {$\color{darkred}\scriptstyle{\spadesuit}$};
\end{tikzpicture}
}\hspace{7.5mm}
&\text{if $\langle h_i,\lambda \rangle > 0$,}\\
(-1)^{|i,\lambda|}
c_{\lambda;i}\,
\mathord{
\begin{tikzpicture}[baseline = 0]
	\draw[-,thick,darkred] (0.28,.6) to[out=240,in=90] (-0.28,-.1);
	\draw[<-,thick,darkred] (-0.28,.6) to[out=300,in=90] (0.28,-0.1);
   \node at (0.28,.7) {$\scriptstyle{i}$};
   \node at (.45,.2) {$\scriptstyle{\lambda}$};
	\draw[-,thick,darkred] (0.28,-0.1) to[out=-90, in=0] (0,-0.4);
	\draw[-,thick,darkred] (0,-0.4) to[out = 180, in = -90] (-0.28,-0.1);
      \node at (0.89,-0.1) {$\color{darkred}\scriptstyle{-\langle h_i,\lambda\rangle}$};
      \node at (0.27,-0.1) {$\color{darkred}\bullet$};
\end{tikzpicture}
}
&\text{if $\langle h_i,\lambda \rangle \leq 0$,}
\end{array}%
\right.
\end{align}
both of which are of parity $|i,\lambda|$.
The following are immediate from these definitions.
\begin{align}\label{posneg}
\mathord{
%
} \!\!= \delta_{n,-\langle h_i,\lambda\rangle-1} c_{\lambda;i}^{-1} 1_{1_\lambda}
\text{ all for $0 \leq n < -\langle h_i,\lambda\rangle$}.\label{startd}
\end{align}
\end{definition}

\begin{definition}\label{def4}
We give meaning to {\em negatively dotted bubbles} by making the
following definitions for $n < 0$:
\begin{align}
\mathord{

}
&:=
\left\{
\begin{array}{cl}
-
(-1)^{|i|(n+\langle h_i,\lambda\rangle+1)}
c_{\lambda;i}\mathord{
\begin{tikzpicture}[baseline = 0]
  \draw[-,thick,darkred] (0.2,0.1) to[out=90,in=0] (0,.3);
  \draw[-,thick,darkred] (0,0.3) to[out=180,in=90] (-.2,0.1);
\draw[->,thick,darkred] (-.2,0.1) to[out=-90,in=180] (0,-0.1);
  \draw[-,thick,darkred] (0,-0.1) to[out=0,in=-90] (0.2,0.1);
 \node at (-.3,.3) {$\scriptstyle{i}$};
   \node at (-.4,0) {$\scriptstyle{\lambda}$};
   \node at (0.2,0.1) {$\color{darkred}\bullet$};
   \node at (0,0.3) {$\color{darkred}\scriptstyle{\clubsuit}$};
   \node at (0.72,0.1) {$\color{darkred}\scriptstyle{-\langle h_i,\lambda\rangle}$};
   \node at (0,0.5) {$\color{darkred}\scriptstyle{-n-1}$};
\end{tikzpicture}
}&\text{if $n > \langle h_i,\lambda\rangle-1$,}\\
c_{\lambda;i} 1_{1_\lambda}&\text{if $n=\langle h_i, \lambda\rangle-1$},\\
0&\text{if $n< \langle h_i,\lambda\rangle-1$},
\end{array}%
\right.
\label{ig2}
\end{align}
Sometimes we will use the following 
convenient shorthand for dotted bubbles for any $n \in \Z$:
\begin{equation}\label{shorthand}
\mathord{
%
},
\end{equation}
both of which are of parity $|i|n$.
Also, assuming that $i \in I$ is {odd}, we introduce the {\em
  odd bubble}
\begin{equation}
\mathord{

}
:=
\left\{
\begin{array}{ll}
c_{\lambda;i}^{-1} \mathord{
\begin{tikzpicture}[baseline = 0]
  \draw[<-,thick,darkred] (0,0.3) to[out=180,in=90] (-.2,0.1);
  \draw[-,thick,darkred] (0.2,0.1) to[out=90,in=0] (0,.3);
 \draw[-,thick,darkred] (-.2,0.1) to[out=-90,in=180] (0,-0.1);
  \draw[-,thick,darkred] (0,-0.1) to[out=0,in=-90] (0.2,0.1);
 \node at (.28,.03) {$\scriptstyle{i}$};
  \node at (-0.2,0.1) {$\color{darkred}\bullet$};
   \node at (-0.7,0.1) {$\color{darkred}\scriptstyle{\langle h_i,\lambda\rangle}$};
   \node at (0.3,0.35) {$\scriptstyle{\lambda}$};
\end{tikzpicture}
}&\text{if $\langle h_i,\lambda\rangle \geq 0$,}\\
c_{\lambda;i}\mathord{
\begin{tikzpicture}[baseline = 0]
   \node at (0.3,0.35) {$\scriptstyle{\lambda}$};
  \draw[->,thick,darkred] (0.2,0.1) to[out=90,in=0] (0,.3);
  \draw[-,thick,darkred] (0,0.3) to[out=180,in=90] (-.2,0.1);
\draw[-,thick,darkred] (-.2,0.1) to[out=-90,in=180] (0,-0.1);
  \draw[-,thick,darkred] (0,-0.1) to[out=0,in=-90] (0.2,0.1);
   \node at (0.2,0.1) {$\color{darkred}\bullet$};
   \node at (0.8,0.1) {$\color{darkred}\scriptstyle{-\langle h_i,\lambda\rangle}$};
 \node at (-.28,.03) {$\scriptstyle{i}$};
\end{tikzpicture}
}&\text{if $\langle h_i,\lambda \rangle \leq 0$.}
\end{array}\right.
\label{oddbubble}\end{equation}
There is no ambiguity in this definition in the case $\langle
h_i,\lambda\rangle = 0$ thanks to the following calculation:
\begin{equation*}
c_{\lambda;i} \mathord{
%
}.
\end{equation*}
\end{definition}

\section{The Chevalley involution}\label{sr1}
The next task is to construct an important symmetry of
$\UU(\g)$. For this, we need some preliminary lemmas.

\begin{lemma}\label{firstlemma}
The following relations hold for all $n \geq 0$:
\begin{align}\label{upcross1}
\mathord{
%
}.\label{downcross2}
\end{align}
\end{lemma}

\begin{proof}
The first two relations follow inductively from (\ref{qha}).
The rest then follow
by rotating clockwise, i.e. attach rightward caps to the top right strands and rightward cups
to the bottom left strands then use (\ref{first})--(\ref{rightpitchfork2}).
\end{proof}

\begin{lemma}
The following relations hold:
\begin{align}
(-1)^{|i||j|}\!\!\mathord{

}\!
&=
\left\{
\begin{array}{ll}
0&\text{if $i=j$,}\\
t_{ij}\!\mathord{
\begin{tikzpicture}[baseline = 0]
	\draw[<-,thick,darkred] (0.08,-.3) to (0.08,.4);
	\draw[<-,thick,darkred] (-0.28,-.3) to (-0.28,.4);
   \node at (-0.28,.5) {$\scriptstyle{i}$};
   \node at (0.08,.5) {$\scriptstyle{j}$};
   \node at (.3,.05) {$\scriptstyle{\lambda}$};
\end{tikzpicture}
}&\text{if $d_{ij}=0$,}\\
(-1)^{|i|\lfloor \frac{d_{ij}}{2}\rfloor}\,t_{ij}
\mathord{
\begin{tikzpicture}[baseline = 0]
	\draw[<-,thick,darkred] (0.08,-.3) to (0.08,.4);
	\draw[<-,thick,darkred] (-0.28,-.3) to (-0.28,.4);
   \node at (-0.28,.5) {$\scriptstyle{i}$};
   \node at (0.08,.5) {$\scriptstyle{j}$};
   \node at (.3,-.05) {$\scriptstyle{\lambda}$};
      \node at (-0.28,0.05) {$\color{darkred}\bullet$};
      \node at (-0.5,0.2) {$\color{darkred}\scriptstyle{d_{ij}}$};
\end{tikzpicture}
}
+
(-1)^{|j|\lfloor \frac{d_{ji}}{2}\rfloor}\,t_{ji}
\mathord{
\begin{tikzpicture}[baseline = 0]
	\draw[<-,thick,darkred] (0.08,-.3) to (0.08,.4);
	\draw[<-,thick,darkred] (-0.28,-.3) to (-0.28,.4);
   \node at (-0.28,.5) {$\scriptstyle{i}$};
   \node at (0.08,.5) {$\scriptstyle{j}$};
   \node at (.3,-.05) {$\scriptstyle{\lambda}$};
     \node at (0.08,0.05) {$\color{darkred}\bullet$};
     \node at (0.32,0.2) {$\color{darkred}\scriptstyle{d_{ji}}$};
\end{tikzpicture}
}\\
\end{array}
\right.\label{noisy}
\end{align}
\end{lemma}

\begin{proof}
Rotate (\ref{now}) clockwise as explained 
in the proof of Lemma~\ref{firstlemma}.
\end{proof}

\begin{lemma}
The following relations hold:
\begin{align}
\label{lurking}
\mathord{

}
&=
\left\{
\begin{array}{ll}
\displaystyle\sum_{\substack{r,s \geq 0\\r+s=d_{ij}-1}}
\!\!\!
(-1)^{|i|s}
t_{ij}
\!
\mathord{
\begin{tikzpicture}[baseline = 0]
	\draw[<-,thick,darkred] (0.38,-.4) to [out=90,in=0] (0.165,-.05);
	\draw[-,thick,darkred] (0.165,-0.05) to [out=180,in=90] (-0.05,-.4);
	\draw[-,thick,darkred] (-0.38,.5) to [out=270,in=180] (-0.165,.15);
	\draw[->,thick,darkred] (-0.165,.15) to [out=0,in=270] (0.05,.5);
	\draw[->,thick,darkred] (-0.38,-.4) to (.38,.5);
   \node at (-0.05,-.5) {$\scriptstyle{k}$};
   \node at (-0.38,.6) {$\scriptstyle{i}$};
   \node at (-0.38,-.5) {$\scriptstyle{j}$};
   \node at (.4,.05) {$\scriptstyle{\lambda}$};
     \node at (0.02,0.3) {$\color{darkred}\bullet$};
     \node at (-0.04,-0.27) {$\color{darkred}\bullet$};
      \node at (-0.15,0.3) {$\color{darkred}\scriptstyle{r}$};
      \node at (0.1,-0.27) {$\color{darkred}\scriptstyle{s}$};
\end{tikzpicture}
}\\
\qquad\quad
+ \displaystyle
\!\!\sum_{\substack{0 < p < d_{ij}\\0 < q <
    d_{ji}\\r,s \geq 0\\r+s=p-1}}
\!\!\!\!(-1)^{|i|s}
s_{ij}^{pq}
\mathord{
\begin{tikzpicture}[baseline = 0]
	\draw[<-,thick,darkred] (0.38,-.4) to [out=90,in=0] (0.165,-.05);
	\draw[-,thick,darkred] (0.165,-0.05) to [out=180,in=90] (-0.05,-.4);
	\draw[-,thick,darkred] (-0.38,.5) to [out=270,in=180] (-0.165,.15);
	\draw[->,thick,darkred] (-0.165,.15) to [out=0,in=270] (0.05,.5);
	\draw[->,thick,darkred] (-0.38,-.4) to (.38,.5);
   \node at (-0.05,-.5) {$\scriptstyle{k}$};
   \node at (-0.38,.6) {$\scriptstyle{i}$};
   \node at (-0.38,-.5) {$\scriptstyle{j}$};
   \node at (.4,.05) {$\scriptstyle{\lambda}$};
     \node at (0.02,0.3) {$\color{darkred}\bullet$};
     \node at (0,0.05) {$\color{darkred}\bullet$};
     \node at (-0.04,-0.27) {$\color{darkred}\bullet$};
      \node at (-0.15,0.3) {$\color{darkred}\scriptstyle{r}$};
      \node at (0.1,-0.27) {$\color{darkred}\scriptstyle{s}$};
      \node at (0.15,0.06) {$\color{darkred}\scriptstyle{q}$};
\end{tikzpicture}
}
&\text{if $i =k\neq j$,}\\
\end{array}%
\right.\label{lastone}\end{align}
\end{lemma}

\begin{proof}
Rotate (\ref{qhalast}) clockwise.
\end{proof}

\begin{definition}\label{curry}
For a supercategory $\A$, we write
$\A^{\operatorname{sop}}$ for the supercategory with the same
objects, morphisms $$\Hom_{\A^{\operatorname{sop}}}(\lambda, \mu) := 
\Hom_{\A}(\mu,\lambda),$$ 
and
new composition law defined from $f^{\operatorname{sop}}
\circ g^{\operatorname{sop}} := (-1)^{|f||g|}(g \circ
f)^{\operatorname{sop}}$,
where we
denote 
a morphism $f:\lambda\rightarrow \mu$ in $\mathcal A$ viewed as a morphism
in $\mathcal A^{\operatorname{sop}}$ by
$f^{\operatorname{sop}}:\mu \rightarrow \lambda$.
For a 2-supercategory $\AA$, we write $\AA^{\operatorname{sop}}$
for the 2-supercategory with the same
objects as $\AA$,
and morphism categories defined from
$\mathcal{H}om_{\AA^{\operatorname{sop}}}(\lambda,\mu) :=
\mathcal{H}om_{\AA}(\lambda,\mu)^{\operatorname{sop}}$. 
Horizontal composition in $\AA^{\operatorname{sop}}$ is
the same as in $\AA$.
Here is the check of the super interchange law in
$\AA^{\operatorname{sop}}$:
\begin{align*}
(x^{\operatorname{sop}} y^{\operatorname{sop}}) \circ
(u^{\operatorname{sop}}v^{\operatorname{sop}})
&=
(xy)^{\operatorname{sop}} \circ (uv)^{\operatorname{sop}}
=
(-1)^{(|x|+|y|)(|u|+|v|)}
((uv) \circ (xy))^{\operatorname{sop}}\\
&=
(-1)^{|x||u|+|y||u|+|y||v|}
((u \circ x)(v \circ y))^{\operatorname{sop}}\\
&=
(-1)^{|x||u|+|y||u|+|y||v|}
(u \circ x)^{\operatorname{sop}} (v \circ y)^{\operatorname{sop}}\\
&= (-1)^{|y||u|} (x^{\operatorname{sop}} \circ u^{\operatorname{sop}})
(y^{\operatorname{sop}} \circ v^{\operatorname{sop}}).
\end{align*}
\end{definition}

We will often appeal to the following proposition to establish mirror images of
relations in a horizontal axis. (This formulation is more convenient than the 
version in \cite[Theorem 2.3]{B}, since $\T$ really is an involution of
$\UU(\g)$ rather than a map to another Kac-Moody 2-category.)

\begin{proposition}\label{opiso}
There is a
$2$-supercategory isomorphism 
$\T:\UU(\g) \stackrel{\sim}{\rightarrow}
\UU(\g)^{\operatorname{sop}}$ defined by the strict 2-superfunctor
$\T$ given
on objects by $\T(\lambda) := -\lambda$, on generating
1-morphisms
by $\T(E_i 1_\lambda) := F_i 1_{-\lambda}$ and
$\T(F_i 1_\lambda) := E_i 1_{-\lambda}$, and on 
generating
2-morphisms by
\begin{align*}
\mathord{

}\!\!.
\end{align*}
Moreover we have that $\T^2 = \operatorname{id}$, 
as follows from the following describing the effect of $\T$ on the
other named $2$-morphisms in $\UU(\g)$:
\begin{align*}
\mathord{
%
}.\!\!\!\!\!\!\!\!\!\!\!
\end{align*}
\end{proposition}

\begin{proof}
This is very similar to the proof of \cite[Theorem 2.3]{B} but the signs are
considerably more subtle, so we include a few remarks.
Note to start with that $\T$ should send
$x^{\circ n}$ (vertical composition computed in $\UU(\g)$) to
$\T(x)^{\circ n}$  (vertical composition 
computed in $\UU(\g)^{\operatorname{sop}}$), 
so that
\begin{equation}
\T\left(\:\mathord{
%
}.\label{tap}
\end{equation}
It is important that the sign here matches the signs in (\ref{first}).
The proof of the existence of $\T$ amounts to checking relations.
For example, to verify (\ref{qhalast}) in the case $i=k \neq j$, 
one
needs to show in view of (\ref{tap}) that
\begin{multline*}
(-1)^{|i||j|+|i|}
\left(\mathord{
%
}
\end{multline*}
in $\UU(\g)$. This follows from (\ref{lastone}).
The other relations follow similarly using (\ref{rightadj}),
(\ref{inv1})--(\ref{inv3}), (\ref{noisy}) and (\ref{downcross1})--(\ref{downcross2}).
The computation of the effect of $\T$ on the other 2-morphisms is
a mostly routine application of the definitions, 
but care is needed to distinguish multiplication (hence, multiplicative
inverses) in
$\UU(\g)$
from 
in  $\UU(\g)^{\operatorname{sop}}$. For example, when $i
\neq j$,
the inverse 
of
$\mathord{
%
}$.
\end{proof}

\section{Leftward dot slides}\label{sr2}
We proceed to prove analogs of the relations (\ref{first})
and (\ref{rtcross1})--(\ref{rtcross2}) for leftward cups, caps and crossings.

\begin{proposition}\label{leftdot}
The following relations hold for all $n \geq 0$:
\begin{align}
\mathord{

}
&=
\left\{
\begin{array}{ll}
\mathord{
\begin{tikzpicture}[baseline = 0]
      \node at (-0.18,0.2) {$\color{darkred}\bullet$};
      \node at (-0.38,0.2) {$\color{darkred}\scriptstyle{n}$};
	\draw[-,thick,darkred] (0.4,0) to[out=90, in=0] (0.1,0.4);
	\draw[->,thick,darkred] (0.1,0.4) to[out = 180, in = 90] (-0.2,0);
    \node at (0.4,-.1) {$\scriptstyle{i}$};
  \node at (0.5,0.25) {$\scriptstyle{\lambda}$};
\end{tikzpicture}
}
&\text{if $|i|n=\0$,}\\
(-1)^{\langle h_i, \lambda \rangle}\mathord{
\begin{tikzpicture}[baseline = 0]
      \node at (-0.18,0.2) {$\color{darkred}\bullet$};
      \node at (-0.38,0.2) {$\color{darkred}\scriptstyle{n}$};
	\draw[-,thick,darkred] (0.4,0) to[out=90, in=0] (0.1,0.4);
	\draw[->,thick,darkred] (0.1,0.4) to[out = 180, in = 90] (-0.2,0);
    \node at (0.4,-.1) {$\scriptstyle{i}$};
  \node at (0.5,0.25) {$\scriptstyle{\lambda}$};
\end{tikzpicture}
}
+
2\mathord{
\begin{tikzpicture}[baseline = 0]
	\draw[-,thick,darkred] (0.3,-0.3) to[out=90, in=0] (0,0.1);
	\draw[->,thick,darkred] (0,0.1) to[out = 180, in = 90] (-0.3,-0.3);
    \node at (0.3,-.4) {$\scriptstyle{i}$};
  \draw[-,thick,darkred] (0,0.2) to[out=180,in=-90] (-.2,0.4);
  \draw[-,thick,darkred] (0.2,0.4) to[out=-90,in=0] (0,0.2);
 \draw[-,thick,darkred] (-.2,0.4) to[out=90,in=180] (0,0.6);
  \draw[-,thick,darkred] (0,0.6) to[out=0,in=90] (0.2,0.4);
  \draw[-,thick,darkred] (0.14,0.55) to (-0.14,0.27);
  \draw[-,thick,darkred] (0.14,0.27) to (-0.14,0.55);
 \node at (-.28,.47) {$\scriptstyle{i}$};
   \node at (0.3,0.15) {$\scriptstyle{\lambda}$};
      \node at (-0.26,-0.1) {$\color{darkred}\bullet$};
      \node at (-0.65,-0.1) {$\color{darkred}\scriptstyle{n-1}$};
\end{tikzpicture}
}
&\text{if $|i|n=\1$,}
\end{array}%
\right.
\label{leftspade}
\end{align}
\end{proposition}

\begin{proof}
Let $h := \langle h_i,\lambda\rangle$. 
The relations (\ref{leftcross1})--(\ref{leftcross2}) follow easily by
induction starting from the case $n=1$, which asserts:
$$
\mathord{
%
}.
$$
It suffices to prove this under the assumption that $h \geq 0$;
the case $h < 0$ then follows by applying the Chevalley involution from
Proposition~\ref{opiso}.
Under this assumption, one vertically composes the $n=1$ case of
(\ref{rtcross1})--(\ref{rtcross2}) 
on top and bottom with a
leftward crossing, then simplifies
using (\ref{rory}) in case $i \neq j$ or
(\ref{first}) and
(\ref{rocksR})--(\ref{startd})
in case $i=j$.

For (\ref{leftclub})--(\ref{leftspade}),
we just need to prove the former, since the latter then
follows on applying $\T$.
When $i$ is even, (\ref{leftclub}) was already established in
\cite[Theorem 5.6]{B},
so let us assume for brevity that $i$ is odd (though the argument
here can easily be adapted to even $i$ too).
When $n=1$ we must prove:
$$
\mathord{

}.
$$
If $h < 0$
we vertically compose this on the bottom with the
isomorphism
$E_i F_i 1_\lambda \oplus 
1_\lambda^{\oplus -h}
\stackrel{\sim}{\rightarrow}
 F_i E_i 1_\lambda$
from (\ref{inv3}) to reduce to checking
\begin{align*}
\mathord{

},
&
\begin{array}{l}
\\\mathord{
\begin{tikzpicture} 
 \draw[-,thick,darkred] (0,0) to[out=180,in=90] (-.2,-0.2);
  \draw[->,thick,darkred] (0.2,-0.2) to[out=90,in=0] (0,0);
 \draw[-,thick,darkred] (-.2,-0.2) to[out=-90,in=180] (0,-0.4);
  \draw[-,thick,darkred] (0,-0.4) to[out=0,in=-90] (0.2,-.2);
 \node at (-.25,0) {$\scriptstyle{i}$};
   \node at (-0.45,-0.2) {$\scriptstyle{\lambda}$};
   \node at (0.175,-.12) {$\color{darkred}\bullet$};
   \node at (0.175,-.28) {$\color{darkred}\bullet$};
   \node at (0.37,-.29) {$\color{darkred}\scriptstyle{m}$};
\node at (0,-.7) {$\phantom{h}$};
\end{tikzpicture}}
\end{array}&=
(-1)^h
\begin{array}{l}\\
\mathord{
\begin{tikzpicture} 
 \draw[-,thick,darkred] (0,0) to[out=180,in=90] (-.2,-0.2);
  \draw[->,thick,darkred] (0.2,-0.2) to[out=90,in=0] (0,0);
 \draw[-,thick,darkred] (-.2,-0.2) to[out=-90,in=180] (0,-0.4);
  \draw[-,thick,darkred] (0,-0.4) to[out=0,in=-90] (0.2,-.2);
 \node at (.25,0) {$\scriptstyle{i}$};
   \node at (-0.45,-0.2) {$\scriptstyle{\lambda}$};
   \node at (-0.175,-.12) {$\color{darkred}\bullet$};
   \node at (0.175,-.28) {$\color{darkred}\bullet$};
   \node at (0.37,-.29) {$\color{darkred}\scriptstyle{m}$};
\node at (0,-.7) {$\phantom{h}$};
\end{tikzpicture}
}
\end{array}
+
2
\begin{array}{l}\\\mathord{
\begin{tikzpicture} 
 \draw[-,thick,darkred] (0,-0.1) to[out=180,in=90] (-.2,-0.3);
  \draw[->,thick,darkred] (0.2,-0.3) to[out=90,in=0] (0,-0.1);
 \draw[-,thick,darkred] (-.2,-0.3) to[out=-90,in=180] (0,-0.5);
  \draw[-,thick,darkred] (0,-0.5) to[out=0,in=-90] (0.2,-.3);
 \node at (-.25,-0.4) {$\scriptstyle{i}$};
   \node at (0.175,-.3) {$\color{darkred}\bullet$};
   \node at (0.37,-.3) {$\color{darkred}\scriptstyle{m}$};
  \draw[-,thick,darkred] (0,0) to[out=180,in=-90] (-.2,0.2);
  \draw[-,thick,darkred] (0.2,0.2) to[out=-90,in=0] (0,0);
 \draw[-,thick,darkred] (-.2,0.2) to[out=90,in=180] (0,0.4);
  \draw[-,thick,darkred] (0,0.4) to[out=0,in=90] (0.2,0.2);
  \draw[-,thick,darkred] (0.14,0.35) to (-0.14,0.07);
  \draw[-,thick,darkred] (0.14,0.07) to (-0.14,0.35);
 \node at (-.28,.27) {$\scriptstyle{i}$};
   \node at (0.3,-0.05) {$\scriptstyle{\lambda}$};
\end{tikzpicture}
}\end{array}
\end{align*}
for all $0 \leq m < -h$.
The first identity here is easily deduced from (\ref{rtcross1}) and
(\ref{startd}),
while the second follows using (\ref{startd}) and the definition (\ref{oddbubble}).
Now assume that $h \geq 0$. Then we have:
\begin{align*}
\mathord{
%
}.
\end{align*}
Thus we have proved the desired relation when $n=1$.
Applying it
twice and
using (\ref{obprops}), we deduce that
$\mathord{
%
},
$
which is the desired relation for $n=2$. The general case follows easily
from the two special cases established so far.
\end{proof}

\section{Infinite Grassmannian relations}\label{sr3}
Recall the shorthand for dotted bubbles
from (\ref{shorthand}),
and that the odd bubble 
$\mathord{
%
}$
squares to zero.
Our next proposition implies that the homomorphisms $\beta_{\lambda;i}$
from (\ref{beta1})--(\ref{beta2}) in the introduction
are well defined.
In terms of these maps, it shows moreover that
\begin{align}\label{goo1}
\mathord{

}
&= \left\{
\begin{array}{ll}
c_{\lambda;i}
\beta_{\lambda;i}\left(\e_n\right)
\hspace{18.8mm}&\text{if $|i|=\0$},\\
c_{\lambda;i}\beta_{\lambda;i}\left(\e_{\lfloor\!\frac{n}{2}\!\rfloor}\right)
&\text{if $|i|=\1$ and $n$ is even},\\
c_{\lambda;i}\beta_{\lambda;i}\left(\d\e_{\lfloor\!\frac{n}{2}\!\rfloor}\right)
&\text{if $|i|=\1$ and $n$ is odd,}
\end{array}%
\right.\label{goo2}
\end{align}
for all $n  \geq 0$.
This extends
the infinite
Grassmannian relation first introduced in \cite{L};
see also \cite[Proposition 3.5]{EL} for a related result in the odd
case.

\begin{proposition}\label{IG}
The following relations hold:
\begin{align}\label{bubble1}
\mathord{
%
}
&= c_{\lambda;i}^{-1}1_{1_\lambda}.\label{bubble2}
\end{align}
Also the following hold for all $t > 0$:
\begin{align}\label{bubble3}
\displaystyle\sum_{\substack{r,s \geq 0 \\ r+s=t}}
\begin{array}{l}
\hspace{-1mm}\mathord{
\begin{tikzpicture}[baseline = 0]
  \draw[<-,thick,darkred] (0,0.4) to[out=180,in=90] (-.2,0.2);
  \draw[-,thick,darkred] (0.2,0.2) to[out=90,in=0] (0,.4);
 \draw[-,thick,darkred] (-.2,0.2) to[out=-90,in=180] (0,0);
  \draw[-,thick,darkred] (0,0) to[out=0,in=-90] (0.2,0.2);
 \node at (.28,.13) {$\scriptstyle{i}$};
  \node at (-0.2,0.2) {$\color{darkred}\bullet$};
   \node at (-0.6,0.2) {$\color{darkred}\scriptstyle{r+*}$};
\end{tikzpicture}
}
\\
\hspace{4mm}
\mathord{
\begin{tikzpicture}[baseline = 0]
  \draw[->,thick,darkred] (0.2,0.2) to[out=90,in=0] (0,.4);
  \draw[-,thick,darkred] (0,0.4) to[out=180,in=90] (-.2,0.2);
\draw[-,thick,darkred] (-.2,0.2) to[out=-90,in=180] (0,0);
  \draw[-,thick,darkred] (0,0) to[out=0,in=-90] (0.2,0.2);
   \node at (0.2,0.2) {$\color{darkred}\bullet$};
   \node at (0.6,0.2) {$\color{darkred}\scriptstyle{s+*}$};
 \node at (-.28,.13) {$\scriptstyle{i}$};
   \node at (0.3,-0.05) {$\scriptstyle{\lambda}$};
\end{tikzpicture}
}\end{array}
&=0
\text{ if $i$ is even},\\
\displaystyle\sum_{\substack{r,s \geq 0 \\ r+s=t}}
\begin{array}{l}
\hspace{-1.5mm}\mathord{
\begin{tikzpicture}[baseline = 0]
  \draw[<-,thick,darkred] (0,0.4) to[out=180,in=90] (-.2,0.2);
  \draw[-,thick,darkred] (0.2,0.2) to[out=90,in=0] (0,.4);
 \draw[-,thick,darkred] (-.2,0.2) to[out=-90,in=180] (0,0);
  \draw[-,thick,darkred] (0,0) to[out=0,in=-90] (0.2,0.2);
 \node at (.28,.13) {$\scriptstyle{i}$};
  \node at (-0.2,0.2) {$\color{darkred}\bullet$};
   \node at (-0.66,0.2) {$\color{darkred}\scriptstyle{2r+*}$};
\end{tikzpicture}
}
\\
\hspace{5mm}
\mathord{
\begin{tikzpicture}[baseline = 0]
  \draw[->,thick,darkred] (0.2,0.2) to[out=90,in=0] (0,.4);
  \draw[-,thick,darkred] (0,0.4) to[out=180,in=90] (-.2,0.2);
\draw[-,thick,darkred] (-.2,0.2) to[out=-90,in=180] (0,0);
  \draw[-,thick,darkred] (0,0) to[out=0,in=-90] (0.2,0.2);
   \node at (0.2,0.2) {$\color{darkred}\bullet$};
   \node at (0.66,0.2) {$\color{darkred}\scriptstyle{2s+*}$};
 \node at (-.28,.13) {$\scriptstyle{i}$};
   \node at (0.3,-0.05) {$\scriptstyle{\lambda}$};
\end{tikzpicture}
}\end{array}
&=0
\text{ if $i$ is odd.}\label{bubble3b}
\end{align}
Finally if $i$ is odd, the following hold for all $n \in \Z$:
\begin{equation}\label{bubble4}
\mathord{

}
=
\begin{array}{l}
\hspace{-1mm}\mathord{
\begin{tikzpicture}[baseline = 5]
  \draw[-,thick,darkred] (0.2,0.2) to[out=90,in=0] (0,.4);
  \draw[<-,thick,darkred] (0,0.4) to[out=180,in=90] (-.2,0.2);
\draw[-,thick,darkred] (-.2,0.2) to[out=-90,in=180] (0,0);
  \draw[-,thick,darkred] (0,0) to[out=0,in=-90] (0.2,0.2);
   \node at (-0.19,0.2) {$\color{darkred}\bullet$};
   \node at (-0.66,0.2) {$\color{darkred}\scriptstyle{2n+*}$};
 \node at (.28,.13) {$\scriptstyle{i}$};
\end{tikzpicture}
}\\
\hspace{6mm}
\mathord{
\begin{tikzpicture}[baseline = -3]
   \node at (0.5,0.2) {$\scriptstyle{\lambda}$};
  \draw[-,thick,darkred] (0,0.3) to[out=180,in=90] (-.2,0.1);
  \draw[-,thick,darkred] (0.2,0.1) to[out=90,in=0] (0,.3);
 \draw[-,thick,darkred] (-.2,0.1) to[out=-90,in=180] (0,-0.1);
  \draw[-,thick,darkred] (0,-0.1) to[out=0,in=-90] (0.2,0.1);
  \draw[-,thick,darkred] (0.14,-0.05) to (-0.14,0.23);
  \draw[-,thick,darkred] (0.14,0.23) to (-0.14,-0.05);
 \node at (-.28,.03) {$\scriptstyle{i}$};
\end{tikzpicture}}
\end{array}%
.
\end{equation}
\end{proposition}

\begin{proof}
Let $h := \langle h_i,\lambda \rangle$.
The equations (\ref{bubble1})--(\ref{bubble2})
are implied by (\ref{startb})--(\ref{ig2}).
For the rest, we first assume that $h\geq 0$ and 
calculate:
\begin{align*}
\sum_{\substack{r,s \in \Z \\ r+s=t-2}}
&\!\!(-1)^{|i|s}\!\!\!\!\!%
}
\right) \stackrel{(\ref{posneg})}{=} 0.
\end{align*}
This establishes the first of the following identities,
and the second follows from that on supercommuting the bubbles then
applying the Chevalley involution from Proposition~\ref{opiso}:
for all $t > 0$ we have that
\begin{align}\label{ig3}
\sum_{\substack{r,s \in \Z \\ r+s=t-2}}
(-1)^{|i|s}\!\!\!\!
\begin{array}{l}
\:\,\mathord{
\begin{tikzpicture}[baseline = 0]
  \draw[->,thick,darkred] (0.2,0.2) to[out=90,in=0] (0,.4);
  \draw[-,thick,darkred] (0,0.4) to[out=180,in=90] (-.2,0.2);
\draw[-,thick,darkred] (-.2,0.2) to[out=-90,in=180] (0,0);
  \draw[-,thick,darkred] (0,0) to[out=0,in=-90] (0.2,0.2);
   \node at (0.2,0.2) {$\color{darkred}\bullet$};
   \node at (0.4,0.2) {$\color{darkred}\scriptstyle{s}$};
 \node at (-.28,.13) {$\scriptstyle{i}$};
\end{tikzpicture}
}\\
\mathord{
\begin{tikzpicture}[baseline = 0]
  \draw[<-,thick,darkred] (0,0.4) to[out=180,in=90] (-.2,0.2);
  \draw[-,thick,darkred] (0.2,0.2) to[out=90,in=0] (0,.4);
 \draw[-,thick,darkred] (-.2,0.2) to[out=-90,in=180] (0,0);
  \draw[-,thick,darkred] (0,0) to[out=0,in=-90] (0.2,0.2);
 \node at (.28,.13) {$\scriptstyle{i}$};
  \node at (-0.2,0.2) {$\color{darkred}\bullet$};
   \node at (-0.4,0.2) {$\color{darkred}\scriptstyle{r}$};
   \node at (-0.3,-0.05) {$\scriptstyle{\lambda}$};
\end{tikzpicture}
}
\end{array}
&=0
\text{ if $h \geq 0$,}
&\sum_{\substack{r,s \in \Z \\ r+s=t-2}}
(-1)^{|i|r}\!\!\!\!\begin{array}{l}
\mathord{
\begin{tikzpicture}[baseline = 0]
 \draw[<-,thick,darkred] (0,0.4) to[out=180,in=90] (-.2,0.2);
 \draw[-,thick,darkred] (0.2,0.2) to[out=90,in=0] (0,.4);
\draw[-,thick,darkred] (-.2,0.2) to[out=-90,in=180] (0,0);
 \draw[-,thick,darkred] (0,0) to[out=0,in=-90] (0.2,0.2);
\node at (.28,.13) {$\scriptstyle{i}$};
 \node at (-0.2,0.2) {$\color{darkred}\bullet$};
 \node at (-0.4,0.2) {$\color{darkred}\scriptstyle{r}$};
\end{tikzpicture}
}
\\
\:\,
\mathord{
\begin{tikzpicture}[baseline = 0] 
 \draw[->,thick,darkred] (0.2,0.2) to[out=90,in=0] (0,.4);
 \draw[-,thick,darkred] (0,0.4) to[out=180,in=90] (-.2,0.2);
\draw[-,thick,darkred] (-.2,0.2) to[out=-90,in=180] (0,0);
 \draw[-,thick,darkred] (0,0) to[out=0,in=-90] (0.2,0.2);
  \node at (0.2,0.2) {$\color{darkred}\bullet$};
  \node at (0.4,0.2) {$\color{darkred}\scriptstyle{s}$};
 \node at (-.28,.13) {$\scriptstyle{i}$};
   \node at (0.3,-0.05) {$\scriptstyle{\lambda}$};
\end{tikzpicture}
}\end{array}
=0
\text{ if $h \leq 0$.}
\end{align}
If $i$ is even, (\ref{ig3}) implies (\ref{bubble3}), and there is nothing more to be done.

For the remainder of the proof we assume that $i$ is odd.
Take $n > 0$ such that $n+h+1$ is odd.
We have that 
\begin{align*}
\mathord{

}
+2
\begin{array}{l}
\mathord{
\begin{tikzpicture}[baseline = 7]
  \draw[-,thick,darkred] (0,0.3) to[out=180,in=90] (-.2,0.1);
  \draw[-,thick,darkred] (0.2,0.1) to[out=90,in=0] (0,.3);
 \draw[-,thick,darkred] (-.2,0.1) to[out=-90,in=180] (0,-0.1);
  \draw[-,thick,darkred] (0,-0.1) to[out=0,in=-90] (0.2,0.1);
  \draw[-,thick,darkred] (0.14,-0.05) to (-0.14,0.23);
  \draw[-,thick,darkred] (0.14,0.23) to (-0.14,-0.05);
 \node at (-.28,.03) {$\scriptstyle{i}$};
   \node at (0.5,-0.1) {$\scriptstyle{\lambda}$};
\end{tikzpicture}
}\\
\mathord{
\begin{tikzpicture}[baseline = -5]
  \draw[->,thick,darkred] (0.2,0.2) to[out=90,in=0] (0,.4);
  \draw[-,thick,darkred] (0,0.4) to[out=180,in=90] (-.2,0.2);
\draw[-,thick,darkred] (-.2,0.2) to[out=-90,in=180] (0,0);
  \draw[-,thick,darkred] (0,0) to[out=0,in=-90] (0.2,0.2);
   \node at (0.19,0.2) {$\color{darkred}\bullet$};
   \node at (0.55,0.2) {$\color{darkred}\scriptstyle{n-1}$};
 \node at (-.28,.13) {$\scriptstyle{i}$};
\end{tikzpicture}
}
\end{array}%
.
\end{align*}
This shows that
\begin{equation}
\mathord{

}
=
\begin{array}{l}
\mathord{
\begin{tikzpicture}[baseline = 7]
  \draw[-,thick,darkred] (0,0.3) to[out=180,in=90] (-.2,0.1);
  \draw[-,thick,darkred] (0.2,0.1) to[out=90,in=0] (0,.3);
 \draw[-,thick,darkred] (-.2,0.1) to[out=-90,in=180] (0,-0.1);
  \draw[-,thick,darkred] (0,-0.1) to[out=0,in=-90] (0.2,0.1);
  \draw[-,thick,darkred] (0.14,-0.05) to (-0.14,0.23);
  \draw[-,thick,darkred] (0.14,0.23) to (-0.14,-0.05);
 \node at (-.28,.03) {$\scriptstyle{i}$};
   \node at (0.5,-0.1) {$\scriptstyle{\lambda}$};
\end{tikzpicture}}
\\
\mathord{
\begin{tikzpicture}[baseline = -5]
  \draw[->,thick,darkred] (0.2,0.2) to[out=90,in=0] (0,.4);
  \draw[-,thick,darkred] (0,0.4) to[out=180,in=90] (-.2,0.2);
\draw[-,thick,darkred] (-.2,0.2) to[out=-90,in=180] (0,0);
  \draw[-,thick,darkred] (0,0) to[out=0,in=-90] (0.2,0.2);
   \node at (0.19,0.2) {$\color{darkred}\bullet$};
   \node at (0.55,0.2) {$\color{darkred}\scriptstyle{n-1}$};
 \node at (-.28,.13) {$\scriptstyle{i}$};
\end{tikzpicture}
}
\end{array}\label{jordan}
\end{equation}
assuming $n > 0$ and $n+h+1$ is odd.
 A similar argument for clockwise bubbles
shows that
\begin{equation}
\mathord{

}
=
\begin{array}{l}
\mathord{
\begin{tikzpicture}[baseline = 5]
  \draw[-,thick,darkred] (0.2,0.2) to[out=90,in=0] (0,.4);
  \draw[<-,thick,darkred] (0,0.4) to[out=180,in=90] (-.2,0.2);
\draw[-,thick,darkred] (-.2,0.2) to[out=-90,in=180] (0,0);
  \draw[-,thick,darkred] (0,0) to[out=0,in=-90] (0.2,0.2);
   \node at (-0.19,0.2) {$\color{darkred}\bullet$};
   \node at (-0.55,0.2) {$\color{darkred}\scriptstyle{n-1}$};
 \node at (.28,.13) {$\scriptstyle{i}$};
\end{tikzpicture}
}\\
\:\:\:\:\:\:
\mathord{
\begin{tikzpicture}[baseline = -3]
   \node at (0.5,0.2) {$\scriptstyle{\lambda}$};
  \draw[-,thick,darkred] (0,0.3) to[out=180,in=90] (-.2,0.1);
  \draw[-,thick,darkred] (0.2,0.1) to[out=90,in=0] (0,.3);
 \draw[-,thick,darkred] (-.2,0.1) to[out=-90,in=180] (0,-0.1);
  \draw[-,thick,darkred] (0,-0.1) to[out=0,in=-90] (0.2,0.1);
  \draw[-,thick,darkred] (0.14,-0.05) to (-0.14,0.23);
  \draw[-,thick,darkred] (0.14,0.23) to (-0.14,-0.05);
 \node at (-.28,.03) {$\scriptstyle{i}$};
\end{tikzpicture}}
\end{array}
\label{jordan2}
\end{equation}
assuming that $n > 0$ and $n+h+1$ is odd. 
Now we proceed to show by ascending induction on $n$ that (\ref{jordan2}) also holds when $n \leq 0$ and $n+h+1$
is odd. 
This statement is vacuous if $n < h$, and it is also 
clear in case $n = h$ thanks to the definition (\ref{oddbubble}).
So assume that $h < n \leq 0$, $n+h+1$ is odd, and that
(\ref{jordan2})
has been proved for all smaller $n$ with $n+h+1$ odd.
From (\ref{ig3}), we get that
$$
\sum_{\substack{r,s \in \Z \\ r+s=n-h-1\\r+h+1\text{ odd}}}
\begin{array}{l}
\mathord{
\begin{tikzpicture}[baseline = 0]
  \draw[<-,thick,darkred] (0,0.4) to[out=180,in=90] (-.2,0.2);
  \draw[-,thick,darkred] (0.2,0.2) to[out=90,in=0] (0,.4);
 \draw[-,thick,darkred] (-.2,0.2) to[out=-90,in=180] (0,0);
  \draw[-,thick,darkred] (0,0) to[out=0,in=-90] (0.2,0.2);
 \node at (.28,.13) {$\scriptstyle{i}$};
  \node at (-0.2,0.2) {$\color{darkred}\bullet$};
   \node at (-0.4,0.2) {$\color{darkred}\scriptstyle{r}$};
\end{tikzpicture}
}
\\
\:\,
\mathord{
\begin{tikzpicture}[baseline = 0]
  \draw[->,thick,darkred] (0.2,0.2) to[out=90,in=0] (0,.4);
  \draw[-,thick,darkred] (0,0.4) to[out=180,in=90] (-.2,0.2);
\draw[-,thick,darkred] (-.2,0.2) to[out=-90,in=180] (0,0);
  \draw[-,thick,darkred] (0,0) to[out=0,in=-90] (0.2,0.2);
   \node at (0.2,0.2) {$\color{darkred}\bullet$};
   \node at (0.4,0.2) {$\color{darkred}\scriptstyle{s}$};
 \node at (-.28,.13) {$\scriptstyle{i}$};
   \node at (0.3,-0.05) {$\scriptstyle{\lambda}$};
\end{tikzpicture}
}\end{array}
-
\sum_{\substack{r,s \in \Z \\ r+s=n-h-1\\r+h+1\text{ even}}}
\begin{array}{l}
\mathord{
\begin{tikzpicture}[baseline = 0]
  \draw[<-,thick,darkred] (0,0.4) to[out=180,in=90] (-.2,0.2);
  \draw[-,thick,darkred] (0.2,0.2) to[out=90,in=0] (0,.4);
 \draw[-,thick,darkred] (-.2,0.2) to[out=-90,in=180] (0,0);
  \draw[-,thick,darkred] (0,0) to[out=0,in=-90] (0.2,0.2);
 \node at (.28,.13) {$\scriptstyle{i}$};
  \node at (-0.2,0.2) {$\color{darkred}\bullet$};
   \node at (-0.4,0.2) {$\color{darkred}\scriptstyle{r}$};
\end{tikzpicture}
}
\\
\:\,
\mathord{
\begin{tikzpicture}[baseline = 0]
  \draw[->,thick,darkred] (0.2,0.2) to[out=90,in=0] (0,.4);
  \draw[-,thick,darkred] (0,0.4) to[out=180,in=90] (-.2,0.2);
\draw[-,thick,darkred] (-.2,0.2) to[out=-90,in=180] (0,0);
  \draw[-,thick,darkred] (0,0) to[out=0,in=-90] (0.2,0.2);
   \node at (0.2,0.2) {$\color{darkred}\bullet$};
   \node at (0.4,0.2) {$\color{darkred}\scriptstyle{s}$};
 \node at (-.28,.13) {$\scriptstyle{i}$};
   \node at (0.3,-0.05) {$\scriptstyle{\lambda}$};
\end{tikzpicture}
}\end{array}
=0.
$$
The terms in the first summation here are zero unless $s \geq -h-1$,
hence, $r \leq n$.
In the second summation we always have that $s+h+1$ is odd, hence,
terms here are zero unless $s > -h-1 \geq 0$.
Applying (\ref{jordan}) to each non-zero term in the second summation, we
deduce that
$$
\sum_{\substack{r,s \in \Z \\ r+s=n-h-1\\r+h+1\text{ odd}}}
\begin{array}{l}
\mathord{
\begin{tikzpicture}[baseline = 0]
  \draw[<-,thick,darkred] (0,0.4) to[out=180,in=90] (-.2,0.2);
  \draw[-,thick,darkred] (0.2,0.2) to[out=90,in=0] (0,.4);
 \draw[-,thick,darkred] (-.2,0.2) to[out=-90,in=180] (0,0);
  \draw[-,thick,darkred] (0,0) to[out=0,in=-90] (0.2,0.2);
 \node at (.28,.13) {$\scriptstyle{i}$};
  \node at (-0.2,0.2) {$\color{darkred}\bullet$};
   \node at (-0.4,0.2) {$\color{darkred}\scriptstyle{r}$};
\end{tikzpicture}
}
\\
\:\,
\mathord{
\begin{tikzpicture}[baseline = 0]
  \draw[->,thick,darkred] (0.2,0.2) to[out=90,in=0] (0,.4);
  \draw[-,thick,darkred] (0,0.4) to[out=180,in=90] (-.2,0.2);
\draw[-,thick,darkred] (-.2,0.2) to[out=-90,in=180] (0,0);
  \draw[-,thick,darkred] (0,0) to[out=0,in=-90] (0.2,0.2);
   \node at (0.2,0.2) {$\color{darkred}\bullet$};
   \node at (0.4,0.2) {$\color{darkred}\scriptstyle{s}$};
 \node at (-.28,.13) {$\scriptstyle{i}$};
   \node at (0.3,-0.05) {$\scriptstyle{\lambda}$};
\end{tikzpicture}
}\end{array}
-
\sum_{\substack{r,s \in \Z \\ r+s=n-h-1\\r+h+1\text{ even}}}
\begin{array}{l}
\mathord{
\begin{tikzpicture}[baseline = 0]
  \draw[<-,thick,darkred] (0,0.4) to[out=180,in=90] (-.2,0.2);
  \draw[-,thick,darkred] (0.2,0.2) to[out=90,in=0] (0,.4);
 \draw[-,thick,darkred] (-.2,0.2) to[out=-90,in=180] (0,0);
  \draw[-,thick,darkred] (0,0) to[out=0,in=-90] (0.2,0.2);
 \node at (.28,.13) {$\scriptstyle{i}$};
  \node at (-0.2,0.2) {$\color{darkred}\bullet$};
   \node at (-0.4,0.2) {$\color{darkred}\scriptstyle{r}$};
\end{tikzpicture}
}
\\
\:\:\mathord{
\begin{tikzpicture}[baseline = -4]
  \draw[-,thick,darkred] (0,0.3) to[out=180,in=90] (-.2,0.1);
  \draw[-,thick,darkred] (0.2,0.1) to[out=90,in=0] (0,.3);
 \draw[-,thick,darkred] (-.2,0.1) to[out=-90,in=180] (0,-0.1);
  \draw[-,thick,darkred] (0,-0.1) to[out=0,in=-90] (0.2,0.1);
  \draw[-,thick,darkred] (0.14,-0.05) to (-0.14,0.23);
  \draw[-,thick,darkred] (0.14,0.23) to (-0.14,-0.05);
 \node at (-.28,.03) {$\scriptstyle{i}$};
\end{tikzpicture}}
\\
\:\:
\mathord{
\begin{tikzpicture}[baseline = 0]
  \draw[->,thick,darkred] (0.2,0.2) to[out=90,in=0] (0,.4);
  \draw[-,thick,darkred] (0,0.4) to[out=180,in=90] (-.2,0.2);
\draw[-,thick,darkred] (-.2,0.2) to[out=-90,in=180] (0,0);
  \draw[-,thick,darkred] (0,0) to[out=0,in=-90] (0.2,0.2);
   \node at (0.2,0.2) {$\color{darkred}\bullet$};
   \node at (0.53,0.2) {$\color{darkred}\scriptstyle{s-1}$};
 \node at (-.28,.13) {$\scriptstyle{i}$};
   \node at (0.3,-0.05) {$\scriptstyle{\lambda}$};
\end{tikzpicture}
}\end{array}
=0.
$$
Now we reindex the second summation, replacing $r$ by $r-1$ and $s$ by
$s+1$, to deduce that
$$
\sum_{\substack{r,s \in \Z \\ r+s=n-h-1\\r+h+1\text{ odd}}}
\left(
\begin{array}{l}
\mathord{
\begin{tikzpicture}[baseline = 0]
  \draw[<-,thick,darkred] (0,0.4) to[out=180,in=90] (-.2,0.2);
  \draw[-,thick,darkred] (0.2,0.2) to[out=90,in=0] (0,.4);
 \draw[-,thick,darkred] (-.2,0.2) to[out=-90,in=180] (0,0);
  \draw[-,thick,darkred] (0,0) to[out=0,in=-90] (0.2,0.2);
 \node at (.28,.13) {$\scriptstyle{i}$};
  \node at (-0.2,0.2) {$\color{darkred}\bullet$};
   \node at (-0.4,0.2) {$\color{darkred}\scriptstyle{r}$};
\end{tikzpicture}
}
\\
\:\,
\mathord{
\begin{tikzpicture}[baseline = 0]
  \draw[->,thick,darkred] (0.2,0.2) to[out=90,in=0] (0,.4);
  \draw[-,thick,darkred] (0,0.4) to[out=180,in=90] (-.2,0.2);
\draw[-,thick,darkred] (-.2,0.2) to[out=-90,in=180] (0,0);
  \draw[-,thick,darkred] (0,0) to[out=0,in=-90] (0.2,0.2);
   \node at (0.2,0.2) {$\color{darkred}\bullet$};
   \node at (0.4,0.2) {$\color{darkred}\scriptstyle{s}$};
 \node at (-.28,.13) {$\scriptstyle{i}$};
   \node at (0.3,-0.05) {$\scriptstyle{\lambda}$};
\end{tikzpicture}
}\end{array}
-
\begin{array}{l}
\!\!\!\!\!\mathord{
\begin{tikzpicture}[baseline = 0]
  \draw[<-,thick,darkred] (0,0.4) to[out=180,in=90] (-.2,0.2);
  \draw[-,thick,darkred] (0.2,0.2) to[out=90,in=0] (0,.4);
 \draw[-,thick,darkred] (-.2,0.2) to[out=-90,in=180] (0,0);
  \draw[-,thick,darkred] (0,0) to[out=0,in=-90] (0.2,0.2);
 \node at (.28,.13) {$\scriptstyle{i}$};
  \node at (-0.2,0.2) {$\color{darkred}\bullet$};
   \node at (-0.53,0.2) {$\color{darkred}\scriptstyle{r-1}$};
\end{tikzpicture}
}
\\
\:\:\mathord{
\begin{tikzpicture}[baseline = -4]
  \draw[-,thick,darkred] (0,0.3) to[out=180,in=90] (-.2,0.1);
  \draw[-,thick,darkred] (0.2,0.1) to[out=90,in=0] (0,.3);
 \draw[-,thick,darkred] (-.2,0.1) to[out=-90,in=180] (0,-0.1);
  \draw[-,thick,darkred] (0,-0.1) to[out=0,in=-90] (0.2,0.1);
  \draw[-,thick,darkred] (0.14,-0.05) to (-0.14,0.23);
  \draw[-,thick,darkred] (0.14,0.23) to (-0.14,-0.05);
 \node at (-.28,.03) {$\scriptstyle{i}$};
\end{tikzpicture}}
\\
\:\:
\mathord{
\begin{tikzpicture}[baseline = 0]
  \draw[->,thick,darkred] (0.2,0.2) to[out=90,in=0] (0,.4);
  \draw[-,thick,darkred] (0,0.4) to[out=180,in=90] (-.2,0.2);
\draw[-,thick,darkred] (-.2,0.2) to[out=-90,in=180] (0,0);
  \draw[-,thick,darkred] (0,0) to[out=0,in=-90] (0.2,0.2);
   \node at (0.2,0.2) {$\color{darkred}\bullet$};
   \node at (0.4,0.2) {$\color{darkred}\scriptstyle{s}$};
 \node at (-.28,.13) {$\scriptstyle{i}$};
   \node at (0.3,-0.05) {$\scriptstyle{\lambda}$};
\end{tikzpicture}
}\end{array}
\right)=0.
$$
In view of the induction hypothesis, all of the terms here in which $r
< n$ vanish.
This just leaves us with the term $r=n$, when $s=-h-1$ so $\mathord{
%
}=c_{\lambda;i}^{-1} 1_{1_\lambda}$, which we can cancel to  establish the desired instance of
(\ref{jordan2}).
This completes the induction.
Hence, we have established the first equation from (\ref{bubble4}); the second
follows from that using Proposition~\ref{opiso}.

It just remains to prove (\ref{bubble3b}).
We explain this assuming that $h \leq 0$; then one can apply the
Chevalley involution to get the other case.
From (\ref{ig3}) we get for any $t > 0$ that
$$
\sum_{\substack{r,s \geq 0 \\ r+s=2t}}
(-1)^r
\begin{array}{l}
\!\!\!\!\!\!\!\!\mathord{
\begin{tikzpicture}[baseline = 0]
  \draw[<-,thick,darkred] (0,0.4) to[out=180,in=90] (-.2,0.2);
  \draw[-,thick,darkred] (0.2,0.2) to[out=90,in=0] (0,.4);
 \draw[-,thick,darkred] (-.2,0.2) to[out=-90,in=180] (0,0);
  \draw[-,thick,darkred] (0,0) to[out=0,in=-90] (0.2,0.2);
 \node at (.28,.13) {$\scriptstyle{i}$};
  \node at (-0.2,0.2) {$\color{darkred}\bullet$};
   \node at (-0.8,0.2) {$\color{darkred}\scriptstyle{r+h-1}$};
\end{tikzpicture}
}
\\
\:\:\:\:\:\:\mathord{
\begin{tikzpicture}[baseline = 0]
  \draw[->,thick,darkred] (0.2,0.2) to[out=90,in=0] (0,.4);
  \draw[-,thick,darkred] (0,0.4) to[out=180,in=90] (-.2,0.2);
\draw[-,thick,darkred] (-.2,0.2) to[out=-90,in=180] (0,0);
  \draw[-,thick,darkred] (0,0) to[out=0,in=-90] (0.2,0.2);
   \node at (0.2,0.2) {$\color{darkred}\bullet$};
   \node at (0.8,0.2) {$\color{darkred}\scriptstyle{s-h-1}$};
 \node at (-.28,.13) {$\scriptstyle{i}$};
   \node at (0.3,-0.05) {$\scriptstyle{\lambda}$};
\end{tikzpicture}
}\end{array}
=0.
$$
In all the terms of this summation we have that $r \equiv s \pmod{2}$.
If both $r$ and $s$ are odd, we can apply (\ref{bubble4}) twice to
pull out two odd bubbles, hence, these terms vanish thanks to
(\ref{obprops}).
This leaves just the terms in which both $r$ and $s$ are even,
which is exactly what is needed to establish the identity (\ref{bubble3b}).
\end{proof}

Once we have proved the next two corollaries, we will not need to refer to
the decorated leftward cups and caps again.

\begin{corollary}\label{advances}
The following relations hold:
\begin{align}\label{invA}
\mathord{
%
}
&&\text{if $0 \leq n < -\langle h_i,\lambda\rangle$.}\label{invB}
\end{align}
\end{corollary}

\begin{proof}
We explain the proof of (\ref{invA}); the proof of (\ref{invB}) is
entirely similar or it may be deduced by applying $\T$ using also
(\ref{leftspade}).
Let $h := \langle h_i,\lambda\rangle > 0$.
Remembering the definition (\ref{cutting}), 
it suffices to show that the vertical composition consisting of (\ref{inv2})
on top of 
$$
\mathord{
-%
}
$$
is equal to the identity.
Using (\ref{posneg})--(\ref{startb}),
this reduces to checking that
\begin{align}
\sum_{r \geq 0}
(-1)^{|i|(h + n+r+1)}
\mathord{

}\!\!\!\!&=0&&\text{if $0 \leq n < h$,}\label{better1}\\
\displaystyle\sum_{r \geq 0}
(-1)^{|i|(h+ n+r+1)}
\begin{array}{l}
\:\,\mathord{
\begin{tikzpicture}[baseline = 0]
  \draw[<-,thick,darkred] (0,0.4) to[out=180,in=90] (-.2,0.2);
  \draw[-,thick,darkred] (0.2,0.2) to[out=90,in=0] (0,.4);
 \draw[-,thick,darkred] (-.2,0.2) to[out=-90,in=180] (0,0);
  \draw[-,thick,darkred] (0,0) to[out=0,in=-90] (0.2,0.2);
 \node at (.28,.13) {$\scriptstyle{i}$};
  \node at (-0.2,0.2) {$\color{darkred}\bullet$};
   \node at (-0.6,0.2) {$\color{darkred}\scriptstyle{m+r}$};
   \node at (-0.3,-0.05) {$\scriptstyle{\lambda}$};
\end{tikzpicture}
}
\\
\hspace{7mm}\mathord{
\begin{tikzpicture}[baseline = 0]
  \draw[->,thick,darkred] (0.2,0.2) to[out=90,in=0] (0,.4);
  \draw[-,thick,darkred] (0,0.4) to[out=180,in=90] (-.2,0.2);
\draw[-,thick,darkred] (-.2,0.2) to[out=-90,in=180] (0,0);
  \draw[-,thick,darkred] (0,0) to[out=0,in=-90] (0.2,0.2);
   \node at (0.2,0.2) {$\color{darkred}\bullet$};
   \node at (0.85,0.2) {$\color{darkred}\scriptstyle{-n-r-2}$};
 \node at (-.28,.13) {$\scriptstyle{i}$};
\end{tikzpicture}
}
\end{array}
&=\delta_{m,n} 1_{1_\lambda}
&&\text{if $0 \leq m,n < h$.}
\label{better2}\end{align}
For (\ref{better1}), each term in the summation is zero:
if $r \geq 
h$ the counterclockwise dotted bubble is zero
by (\ref{bubble2}); if $0 \leq r < h$ one
commutes the dots past the crossing using (\ref{rtcross1}) then
applies (\ref{startb}).
To prove (\ref{better2}), note 
by (\ref{bubble1})--(\ref{bubble2}) that
in order for 
$\mathord{
%
}$ 
to be non-zero we must have that $r \geq h-m-1$,
while for 
$\mathord{
%
}$ 
to be non-zero we must have
$r \leq h-n-1$.
Hence, we may assume that $m \geq n$, and are done for the same reasons
in case $m=n$.
If $m > n$ the left hand side of (\ref{better2}) is
equal to
$$
\sum_{\substack{r,s \geq 0 \\ r+s=m-n}}
(-1)^{|i|(m+n+r)}
\begin{array}{l}
\:\,\mathord{
\begin{tikzpicture}[baseline = 0]
  \draw[<-,thick,darkred] (0,0.4) to[out=180,in=90] (-.2,0.2);
  \draw[-,thick,darkred] (0.2,0.2) to[out=90,in=0] (0,.4);
 \draw[-,thick,darkred] (-.2,0.2) to[out=-90,in=180] (0,0);
  \draw[-,thick,darkred] (0,0) to[out=0,in=-90] (0.2,0.2);
 \node at (.28,.13) {$\scriptstyle{i}$};
  \node at (-0.2,0.2) {$\color{darkred}\bullet$};
   \node at (-0.8,0.2) {$\color{darkred}\scriptstyle{r+h-1}$};
   \node at (-0.3,-0.05) {$\scriptstyle{\lambda}$};
\end{tikzpicture}
}
\\
\hspace{11mm}\mathord{
\begin{tikzpicture}[baseline = 0]
  \draw[->,thick,darkred] (0.2,0.2) to[out=90,in=0] (0,.4);
  \draw[-,thick,darkred] (0,0.4) to[out=180,in=90] (-.2,0.2);
\draw[-,thick,darkred] (-.2,0.2) to[out=-90,in=180] (0,0);
  \draw[-,thick,darkred] (0,0) to[out=0,in=-90] (0.2,0.2);
   \node at (0.2,0.2) {$\color{darkred}\bullet$};
   \node at (0.8,0.2) {$\color{darkred}\scriptstyle{s-h-1}$};
 \node at (-.28,.13) {$\scriptstyle{i}$};
\end{tikzpicture}
}
\end{array}.
$$
Now one shows that this is zero using (\ref{bubble3})--(\ref{bubble4})
and (\ref{obprops});
when $i$ is odd it is convenient when checking this to treat the cases $m\equiv n\pmod{2}$ and $m
\not\equiv n\pmod{2}$ separately.
\end{proof}

\begin{corollary}
The following relations hold:
\begin{align}\label{pos}
\mathord{
%
}.
\end{align}
\end{corollary}

\begin{proof}
Substitute (\ref{invA})--(\ref{invB}) into (\ref{posneg}).
\end{proof}

\begin{corollary}
The following relations hold:
\begin{align}\label{everything}
\mathord{
%
}\:
{\scriptstyle\lambda}.
\label{dog2b}
\end{align}
\end{corollary}

\begin{proof}
We first prove (\ref{everything}).
By our usual argument with the Chevalley involution, it suffices to prove the left hand relation.
We are done already by (\ref{startd}) if $h := \langle
h_i,\lambda \rangle < 0$. 
If $h \geq 0$ then:
\begin{align*}
\mathord{
%
}.
\end{align*}
It remains to observe 
just like 
at the end of the proof of Corollary~\ref{advances}
that the
second summation on the right hand side vanishes.

Finally to deduce (\ref{dog1})--(\ref{dog2b}), use (\ref{upcross1})--(\ref{upcross2}) to commute the dots past the
upward crossing,
then convert the crossing to a rightward one
using (\ref{rightadj})--(\ref{sigrel}) and apply (\ref{everything}).
\end{proof}

\section{Left adjunction}\label{sr4}
The leftward cups and caps form the unit and counit of another adjunction.

\begin{lemma}\label{pitchforka}
The following relations hold:
\begin{align}\label{one}
\mathord{

}&&\text{if $\langle h_i,\lambda\rangle \geq -1$.}\label{two}
\end{align}
\end{lemma}

\begin{proof}
Let $h := 
\langle h_i,\lambda\rangle$ for short.
First we prove (\ref{one}), so $h \leq -1$.
We claim that
\begin{equation}
-\!\!\!
\mathord{
%
}.\label{quarry}
\end{equation}
To establish the claim, we vertically compose on the bottom with the
isomorphism
$\mathord{
%
}$
arising from (\ref{inv3})
to reduce to showing equivalently that
\begin{align}
-\mathord{
%
}
\text{ for $0 \leq n \leq -h-1$.}\label{claude2}
\end{align}
Here is the verification of (\ref{claude1}):
\begin{align*}
-\mathord{
%
}.
\end{align*}
For (\ref{claude2}), by (\ref{bubble2}) and (\ref{rightadj}),
the right hand side is $
c_{\lambda;i}^{-1}\substack{{\color{darkred}{\displaystyle\uparrow}} \\
  {\scriptscriptstyle i}}{\scriptstyle\substack{\lambda\\\phantom{-}}}
$ if $n=-h-1>0$, and it is zero otherwise.
On the other hand, the left hand side equals
\begin{align*}
-(-1)^{|i|n}\mathord{
%
}.
\end{align*}
This is obviously zero if $n = 0$.
Assuming $n > 0$, we apply (\ref{dog1}) to see that it is zero unless
$n = -h-1$,
when the term with $r=-h-2,s=0$ contributes
$c_{\lambda;i}^{-1}\substack{{\color{darkred}{\displaystyle\uparrow}} \\
  {\scriptscriptstyle i}}{\scriptstyle\substack{\lambda\\\phantom{-}}}
$.
This completes the proof of the claim.
Now we can establish (\ref{one}):
\begin{align*}
\mathord{
%
}.
\end{align*}
The proof of (\ref{two})
follows by a very similar argument;
one first checks that
$$
-\!\!\!\mathord{
%
}
$$
when $h \geq -1$.
\end{proof}

\begin{proposition}\label{adjdone}
The following relations hold:
\begin{align}\label{adjfinal}
\mathord{
%
}.
\end{align}
\end{proposition}

\begin{proof}
It suffices to prove the first equality;
the second one
then follows using Proposition~\ref{opiso}.
Let $h := \langle h_i,\lambda \rangle$ for short, 
and recall that $|i,\lambda| = |i|(h+1)$.
If $h \geq 0$ then
\begin{align*}
\mathord{
%
}.
\end{align*}
This completes the proof.
\end{proof}

\section{Final relations}\label{sr5}
There are just a few more important relations to be derived.

\begin{lemma}\label{pitchforkb}
The following hold for all $i \neq j$:
\begin{align}\label{easy1}
\mathord{
%
} &&\text{if $\langle h_i,\lambda\rangle \geq d_{ij}$.}\label{easy2}
\end{align}
\end{lemma}

\begin{proof}
Let $h := \langle h_i,\lambda\rangle$.
First we prove (\ref{easy1}) assuming that $0 < h \leq d_{ij}-1$.
Vertically composing on the bottom with the isomorphism
$\mathord{
%
}$,
we reduce to proving that 
\begin{equation}\label{critical}
\mathord{
%
}.
\end{equation}
Then to check this, we apply (\ref{neg}) to transform 
the left hand side into
\begin{align*}
-\!\mathord{
%
}.
\end{align*}
The first term on the right hand side here vanishes by (\ref{startd}).
Also the terms in the summations are zero unless $r \geq
d_{ij}-h-1$ and $s \geq h$ by (\ref{startb}) and (\ref{bubble2}), hence,
we are left just with the
$r=d_{ij}-h-1, s = h$ term, which equals
\begin{align*}
-(-1)^{|i|h}t_{ij}
c_{\lambda+\alpha_j;i}^{-1}
\mathord{
%
}.
\end{align*}
This is equal to the right hand  side of (\ref{critical}) thanks to (\ref{rory}).

To complete the proof of (\ref{easy1}), we need to show that it holds 
when $h \leq 0$.
By (\ref{inv1}) and (\ref{inv3}), the following 2-morphism is invertible:
$$
\mathord{
%
}.
$$
Vertically composing with this on the bottom,
we deduce that the
 relation we are trying to prove is equivalent to the following relations:
\begin{equation}\label{turn}
\mathord{
%
}
 \text{ for $0 \leq n < -h$.}
\end{equation}
To establish the first of these, we pull the $j$-string past the
$ii$-crossing:
\begin{align*}
\mathord{
%
}.
\end{align*}
If $h < 0$ then all the terms on the right hand side vanish thanks to
(\ref{startd}) and (\ref{bubble2}).
If $h = 0$ and $d_{ij} > 0$ everything except for the $r=d_{ij}-1$ term from
the first sum vanishes, and we get $t_{ij} c_{\lambda+\alpha_j;i}^{-1}
\mathord{
%
}$.
Finally if $h = d_{ij} = 0$, we only have the first
term on the right hand side, which
contributes
$t_{ij} c_{\lambda+\alpha_j;i}^{-1} \mathord{
%
}$
again thanks to (\ref{everything}), (\ref{bubble2}), (\ref{rightpitchfork}) and
(\ref{now}).
This is what we want because:
\begin{align*}
\mathord{
%
}.
\end{align*}
We are just left with the right hand relations from (\ref{turn})
involving bubbles:
\begin{align*}
&\mathord{
%
}
+
\left(
\text{a lin. comb. of}
\mathord{
%
}.
\end{align*}

The relation (\ref{easy2}) follows by very similar arguments to the
previous paragraph;
the first step is to vertically compose on the top with the isomorphism
$$
\displaystyle\mathord{
%
}.
$$
\end{proof}

\begin{proposition}\label{pitch}
The following relations hold for all $i, j$:
\begin{align}
\label{final1}
\mathord{
%
}.\label{final2}
\end{align}
\end{proposition}

\begin{proof}
We get (\ref{final1}) in half of the cases from
Lemmas~\ref{pitchforka} and \ref{pitchforkb}.
To deduce the other half of the cases, attach
leftward cups (resp.\ caps) to the two strands at the bottom 
(resp. the top) of the relations established in these two lemma, then simplify using
(\ref{adjfinal}).
Finally (\ref{final2}) follows from
(\ref{final1}) using Proposition~\ref{opiso} as usual.
\end{proof}

The final two propositions of the section extend
\cite[Propositions 3.3--3.5]{KL3}.

\begin{proposition}\label{bs}
The following hold for all $n \geq 0$ and $\lambda \in P$.
\begin{itemize}
\item[(i)]
If $i$ is even then
\begin{align}\label{bs1}
\mathord{
\begin{tikzpicture}[baseline = 0]
	\draw[->,thick,darkred] (0.08,-.4) to (0.08,.4);
     \node at (0.08,-.5) {$\scriptstyle{i}$};
\end{tikzpicture}
}
{\scriptstyle\lambda}
\,\:\mathord{\begin{tikzpicture}[baseline = 0]
  \draw[-,thick,darkred] (0,0.2) to[out=180,in=90] (-.2,0);
  \draw[->,thick,darkred] (0.2,0) to[out=90,in=0] (0,.2);
 \draw[-,thick,darkred] (-.2,0) to[out=-90,in=180] (0,-0.2);
  \draw[-,thick,darkred] (0,-0.2) to[out=0,in=-90] (0.2,0);
 \node at (0,-.28) {$\scriptstyle{i}$};
   \node at (0.6,0) {$\color{darkred}\scriptstyle{n+*}$};
      \node at (0.2,0) {$\color{darkred}\bullet$};
\end{tikzpicture}
}
&=
\sum_{r \geq 0}
(r+1)
\:\mathord{\begin{tikzpicture}[baseline = 0]
  \draw[-,thick,darkred] (0,0.2) to[out=180,in=90] (-.2,0);
  \draw[->,thick,darkred] (0.2,0) to[out=90,in=0] (0,.2);
 \draw[-,thick,darkred] (-.2,0) to[out=-90,in=180] (0,-0.2);
  \draw[-,thick,darkred] (0,-0.2) to[out=0,in=-90] (0.2,0);
 \node at (0,-.28) {$\scriptstyle{i}$};
   \node at (.8,0) {$\color{darkred}\scriptstyle{n-r+*}$};
      \node at (.2,0) {$\color{darkred}\bullet$};
\end{tikzpicture}
}
\!\!\mathord{
\begin{tikzpicture}[baseline = 0]
	\draw[->,thick,darkred] (0.08,-.4) to (0.08,.4);
   \node at (.3,0) {$\color{darkred}\scriptstyle{r}$};
      \node at (.08,0) {$\color{darkred}\bullet$};
     \node at (0.08,-.5) {$\scriptstyle{i}$};
\end{tikzpicture}
}
\!
{\scriptstyle\lambda}\,,
\\\label{bs2}
\mathord{\begin{tikzpicture}[baseline = 0]
  \draw[<-,thick,darkred] (0,0.2) to[out=180,in=90] (-.2,0);
  \draw[-,thick,darkred] (0.2,0) to[out=90,in=0] (0,.2);
 \draw[-,thick,darkred] (-.2,0) to[out=-90,in=180] (0,-0.2);
  \draw[-,thick,darkred] (0,-0.2) to[out=0,in=-90] (0.2,0);
 \node at (0,-.28) {$\scriptstyle{i}$};
   \node at (-0.6,0) {$\color{darkred}\scriptstyle{n+*}$};
      \node at (-0.2,0) {$\color{darkred}\bullet$};
\end{tikzpicture}
}
\:\mathord{
\begin{tikzpicture}[baseline = 0]
	\draw[->,thick,darkred] (0.08,-.4) to (0.08,.4);
     \node at (0.08,-.5) {$\scriptstyle{i}$};
\end{tikzpicture}
}
{\scriptstyle\lambda}
&=
\sum_{r \geq 0}
(r+1)
\mathord{
\begin{tikzpicture}[baseline = 0]
	\draw[->,thick,darkred] (0.08,-.4) to (0.08,.4);
   \node at (-.15,0) {$\color{darkred}\scriptstyle{r}$};
      \node at (.08,0) {$\color{darkred}\bullet$};
     \node at (0.08,-.5) {$\scriptstyle{i}$};
\end{tikzpicture}
}
{\scriptstyle\lambda}
\:\mathord{\begin{tikzpicture}[baseline = 0]
  \draw[<-,thick,darkred] (0,0.2) to[out=180,in=90] (-.2,0);
  \draw[-,thick,darkred] (0.2,0) to[out=90,in=0] (0,.2);
 \draw[-,thick,darkred] (-.2,0) to[out=-90,in=180] (0,-0.2);
  \draw[-,thick,darkred] (0,-0.2) to[out=0,in=-90] (0.2,0);
 \node at (0,-.28) {$\scriptstyle{i}$};
   \node at (-.8,0) {$\color{darkred}\scriptstyle{n-r+*}$};
      \node at (-.2,0) {$\color{darkred}\bullet$};
\end{tikzpicture}
}\,.
\end{align}
\item[(ii)]
If $i$ is odd then
\begin{align}\label{bs3}
\mathord{
\begin{tikzpicture}[baseline = 0]
	\draw[->,thick,darkred] (0.08,-.4) to (0.08,.4);
     \node at (0.08,-.5) {$\scriptstyle{i}$};
\end{tikzpicture}
}
{\scriptstyle\lambda}
\,\:\mathord{\begin{tikzpicture}[baseline = 0]
  \draw[-,thick,darkred] (0,0.2) to[out=180,in=90] (-.2,0);
  \draw[->,thick,darkred] (0.2,0) to[out=90,in=0] (0,.2);
 \draw[-,thick,darkred] (-.2,0) to[out=-90,in=180] (0,-0.2);
  \draw[-,thick,darkred] (0,-0.2) to[out=0,in=-90] (0.2,0);
 \node at (0,-.28) {$\scriptstyle{i}$};
   \node at (0.6,0) {$\color{darkred}\scriptstyle{n+*}$};
      \node at (0.2,0) {$\color{darkred}\bullet$};
\end{tikzpicture}
}
&=
\sum_{r \geq 0}
(2r+1)
\:\mathord{\begin{tikzpicture}[baseline = 0]
  \draw[-,thick,darkred] (0,0.2) to[out=180,in=90] (-.2,0);
  \draw[->,thick,darkred] (0.2,0) to[out=90,in=0] (0,.2);
 \draw[-,thick,darkred] (-.2,0) to[out=-90,in=180] (0,-0.2);
  \draw[-,thick,darkred] (0,-0.2) to[out=0,in=-90] (0.2,0);
 \node at (0,-.28) {$\scriptstyle{i}$};
   \node at (.85,0) {$\color{darkred}\scriptstyle{n-2r+*}$};
      \node at (.2,0) {$\color{darkred}\bullet$};
\end{tikzpicture}
}
\!\!\mathord{
\begin{tikzpicture}[baseline = 0]
	\draw[->,thick,darkred] (0.08,-.4) to (0.08,.4);
   \node at (.35,0) {$\color{darkred}\scriptstyle{2r}$};
      \node at (.08,0) {$\color{darkred}\bullet$};
     \node at (0.08,-.5) {$\scriptstyle{i}$};
\end{tikzpicture}
}
\!
{\scriptstyle\lambda}\,,
\\\label{bs4}
\mathord{\begin{tikzpicture}[baseline = 0]
  \draw[<-,thick,darkred] (0,0.2) to[out=180,in=90] (-.2,0);
  \draw[-,thick,darkred] (0.2,0) to[out=90,in=0] (0,.2);
 \draw[-,thick,darkred] (-.2,0) to[out=-90,in=180] (0,-0.2);
  \draw[-,thick,darkred] (0,-0.2) to[out=0,in=-90] (0.2,0);
 \node at (0,-.28) {$\scriptstyle{i}$};
   \node at (-0.6,0) {$\color{darkred}\scriptstyle{n+*}$};
      \node at (-0.2,0) {$\color{darkred}\bullet$};
\end{tikzpicture}
}
\:\mathord{
\begin{tikzpicture}[baseline = 0]
	\draw[->,thick,darkred] (0.08,-.4) to (0.08,.4);
     \node at (0.08,-.5) {$\scriptstyle{i}$};
\end{tikzpicture}
}
{\scriptstyle\lambda}
&=
\sum_{r \geq 0}
(2r+1)
\mathord{
\begin{tikzpicture}[baseline = 0]
	\draw[->,thick,darkred] (0.08,-.4) to (0.08,.4);
   \node at (-.2,0) {$\color{darkred}\scriptstyle{2r}$};
      \node at (.08,0) {$\color{darkred}\bullet$};
     \node at (0.08,-.5) {$\scriptstyle{i}$};
\end{tikzpicture}
}
{\scriptstyle\lambda}
\:\mathord{\begin{tikzpicture}[baseline = 0]
  \draw[<-,thick,darkred] (0,0.2) to[out=180,in=90] (-.2,0);
  \draw[-,thick,darkred] (0.2,0) to[out=90,in=0] (0,.2);
 \draw[-,thick,darkred] (-.2,0) to[out=-90,in=180] (0,-0.2);
  \draw[-,thick,darkred] (0,-0.2) to[out=0,in=-90] (0.2,0);
 \node at (0,-.28) {$\scriptstyle{i}$};
   \node at (-.85,0) {$\color{darkred}\scriptstyle{n-2r+*}$};
      \node at (-.2,0) {$\color{darkred}\bullet$};
\end{tikzpicture}
}\,.
\end{align}
\item[(iii)]
For $i \neq j$ with $d_{ij} > 0$ we have that
\begin{align}\label{bs5}
\mathord{
\begin{tikzpicture}[baseline = 0]
	\draw[->,thick,darkred] (0.08,-.4) to (0.08,.4);
     \node at (0.08,-.5) {$\scriptstyle{j}$};
\end{tikzpicture}
}
\!{\scriptstyle\lambda}
\:\mathord{\begin{tikzpicture}[baseline = 0]
  \draw[-,thick,darkred] (0,0.2) to[out=180,in=90] (-.2,0);
  \draw[->,thick,darkred] (0.2,0) to[out=90,in=0] (0,.2);
 \draw[-,thick,darkred] (-.2,0) to[out=-90,in=180] (0,-0.2);
  \draw[-,thick,darkred] (0,-0.2) to[out=0,in=-90] (0.2,0);
 \node at (0,-.28) {$\scriptstyle{i}$};
   \node at (0.6,0) {$\color{darkred}\scriptstyle{n+*}$};
      \node at (0.2,0) {$\color{darkred}\bullet$};
\end{tikzpicture}
}
&=
t_{ij}
\,\mathord{\begin{tikzpicture}[baseline = 0]
  \draw[-,thick,darkred] (0,0.2) to[out=180,in=90] (-.2,0);
  \draw[->,thick,darkred] (0.2,0) to[out=90,in=0] (0,.2);
 \draw[-,thick,darkred] (-.2,0) to[out=-90,in=180] (0,-0.2);
  \draw[-,thick,darkred] (0,-0.2) to[out=0,in=-90] (0.2,0);
 \node at (0,-.28) {$\scriptstyle{i}$};
   \node at (.6,0) {$\color{darkred}\scriptstyle{n+*}$};
      \node at (.2,0) {$\color{darkred}\bullet$};
\end{tikzpicture}
}
\!\!\mathord{
\begin{tikzpicture}[baseline = 0]
	\draw[->,thick,darkred] (0.08,-.4) to (0.08,.4);
     \node at (0.08,-.5) {$\scriptstyle{j}$};
\end{tikzpicture}
}
\!
{\scriptstyle\lambda}
+
t_{ji}\,
\mathord{\begin{tikzpicture}[baseline = 0]
  \draw[-,thick,darkred] (0,0.2) to[out=180,in=90] (-.2,0);
  \draw[->,thick,darkred] (0.2,0) to[out=90,in=0] (0,.2);
 \draw[-,thick,darkred] (-.2,0) to[out=-90,in=180] (0,-0.2);
  \draw[-,thick,darkred] (0,-0.2) to[out=0,in=-90] (0.2,0);
 \node at (0,-.28) {$\scriptstyle{i}$};
   \node at (.85,0) {$\color{darkred}\scriptstyle{n-d_{ij}+*}$};
      \node at (.2,0) {$\color{darkred}\bullet$};
\end{tikzpicture}
}
\!\!\mathord{
\begin{tikzpicture}[baseline = 0]
	\draw[->,thick,darkred] (0.08,-.4) to (0.08,.4);
   \node at (.4,0) {$\color{darkred}\scriptstyle{d_{ji}}$};
      \node at (.08,0) {$\color{darkred}\bullet$};
     \node at (0.08,-.5) {$\scriptstyle{j}$};
\end{tikzpicture}
}
\!
{\scriptstyle\lambda}
+
\!\!\displaystyle
\sum_{\substack{0 < p < d_{ij}\\0 < q < d_{ji}}}
\!\!
s_{ij}^{pq}
\:\mathord{\begin{tikzpicture}[baseline = 0]
  \draw[-,thick,darkred] (0,0.2) to[out=180,in=90] (-.2,0);
  \draw[->,thick,darkred] (0.2,0) to[out=90,in=0] (0,.2);
 \draw[-,thick,darkred] (-.2,0) to[out=-90,in=180] (0,-0.2);
  \draw[-,thick,darkred] (0,-0.2) to[out=0,in=-90] (0.2,0);
 \node at (0,-.28) {$\scriptstyle{i}$};
   \node at (1.05,0) {$\color{darkred}\scriptstyle{n+p-d_{ij}+*}$};
      \node at (.2,0) {$\color{darkred}\bullet$};
\end{tikzpicture}
}
\mathord{
\begin{tikzpicture}[baseline = 0]
	\draw[->,thick,darkred] (0.08,-.4) to (0.08,.4);
   \node at (.3,0) {$\color{darkred}\scriptstyle{q}$};
      \node at (.08,0) {$\color{darkred}\bullet$};
     \node at (0.08,-.5) {$\scriptstyle{j}$};
\end{tikzpicture}
}
\!
{\scriptstyle\lambda},\\\label{bs6}
\mathord{\begin{tikzpicture}[baseline = 0]
  \draw[<-,thick,darkred] (0,0.2) to[out=180,in=90] (-.2,0);
  \draw[-,thick,darkred] (0.2,0) to[out=90,in=0] (0,.2);
 \draw[-,thick,darkred] (-.2,0) to[out=-90,in=180] (0,-0.2);
  \draw[-,thick,darkred] (0,-0.2) to[out=0,in=-90] (0.2,0);
 \node at (0,-.28) {$\scriptstyle{i}$};
   \node at (-0.6,0) {$\color{darkred}\scriptstyle{n+*}$};
      \node at (-0.2,0) {$\color{darkred}\bullet$};
\end{tikzpicture}
}
\mathord{
\begin{tikzpicture}[baseline = 0]
	\draw[->,thick,darkred] (0.08,-.4) to (0.08,.4);
     \node at (0.08,-.5) {$\scriptstyle{j}$};
\end{tikzpicture}
}
{\scriptstyle\lambda}
&=
t_{ij}
\mathord{
\begin{tikzpicture}[baseline = 0]
	\draw[->,thick,darkred] (0.08,-.4) to (0.08,.4);
     \node at (0.08,-.5) {$\scriptstyle{j}$};
\end{tikzpicture}
}
\!\mathord{\begin{tikzpicture}[baseline = 0]
  \draw[<-,thick,darkred] (0,0.2) to[out=180,in=90] (-.2,0);
  \draw[-,thick,darkred] (0.2,0) to[out=90,in=0] (0,.2);
 \draw[-,thick,darkred] (-.2,0) to[out=-90,in=180] (0,-0.2);
  \draw[-,thick,darkred] (0,-0.2) to[out=0,in=-90] (0.2,0);
 \node at (0,-.28) {$\scriptstyle{i}$};
   \node at (-.6,0) {$\color{darkred}\scriptstyle{n+*}$};
      \node at (-.2,0) {$\color{darkred}\bullet$};
\end{tikzpicture}
}\,
{\scriptstyle\lambda}
+
t_{ji}\!
\mathord{
\begin{tikzpicture}[baseline = 0]
	\draw[->,thick,darkred] (0.08,-.4) to (0.08,.4);
   \node at (-.2,0) {$\color{darkred}\scriptstyle{d_{ji}}$};
      \node at (.08,0) {$\color{darkred}\bullet$};
     \node at (0.08,-.5) {$\scriptstyle{j}$};
\end{tikzpicture}
}\!\!
\mathord{\begin{tikzpicture}[baseline = 0]
  \draw[<-,thick,darkred] (0,0.2) to[out=180,in=90] (-.2,0);
  \draw[-,thick,darkred] (0.2,0) to[out=90,in=0] (0,.2);
 \draw[-,thick,darkred] (-.2,0) to[out=-90,in=180] (0,-0.2);
  \draw[-,thick,darkred] (0,-0.2) to[out=0,in=-90] (0.2,0);
 \node at (0,-.28) {$\scriptstyle{i}$};
   \node at (-.85,0) {$\color{darkred}\scriptstyle{n-d_{ij}+*}$};
      \node at (-.2,0) {$\color{darkred}\bullet$};
\end{tikzpicture}
}\,
{\scriptstyle\lambda}
+
\!\!\displaystyle
\sum_{\substack{0 < p < d_{ij}\\0 < q < d_{ji}}}
\!\!
s_{ij}^{pq}
\:
\mathord{
\begin{tikzpicture}[baseline = 0]
	\draw[->,thick,darkred] (0.08,-.4) to (0.08,.4);
   \node at (-.15,0) {$\color{darkred}\scriptstyle{q}$};
      \node at (.08,0) {$\color{darkred}\bullet$};
     \node at (0.08,-.5) {$\scriptstyle{j}$};
\end{tikzpicture}
}\!
\mathord{\begin{tikzpicture}[baseline = 0]
  \draw[<-,thick,darkred] (0,0.2) to[out=180,in=90] (-.2,0);
  \draw[-,thick,darkred] (0.2,0) to[out=90,in=0] (0,.2);
 \draw[-,thick,darkred] (-.2,0) to[out=-90,in=180] (0,-0.2);
  \draw[-,thick,darkred] (0,-0.2) to[out=0,in=-90] (0.2,0);
 \node at (0,-.28) {$\scriptstyle{i}$};
   \node at (-1.05,0) {$\color{darkred}\scriptstyle{n+p-d_{ij}+*}$};
      \node at (-.2,0) {$\color{darkred}\bullet$};
\end{tikzpicture}
}\,
{\scriptstyle\lambda}.
\end{align}
\item[(iv)]
For $i \neq j$ with $d_{ij} = 0$ we have that
\begin{align}\label{bs7}
\mathord{
\begin{tikzpicture}[baseline = 0]
	\draw[->,thick,darkred] (0.08,-.4) to (0.08,.4);
     \node at (0.08,-.5) {$\scriptstyle{j}$};
\end{tikzpicture}
}
\!{\scriptstyle\lambda}
\:\mathord{\begin{tikzpicture}[baseline = 0]
  \draw[-,thick,darkred] (0,0.2) to[out=180,in=90] (-.2,0);
  \draw[->,thick,darkred] (0.2,0) to[out=90,in=0] (0,.2);
 \draw[-,thick,darkred] (-.2,0) to[out=-90,in=180] (0,-0.2);
  \draw[-,thick,darkred] (0,-0.2) to[out=0,in=-90] (0.2,0);
 \node at (0,-.28) {$\scriptstyle{i}$};
   \node at (0.6,0) {$\color{darkred}\scriptstyle{n+*}$};
      \node at (0.2,0) {$\color{darkred}\bullet$};
\end{tikzpicture}
}
&=
\,\mathord{\begin{tikzpicture}[baseline = 0]
  \draw[-,thick,darkred] (0,0.2) to[out=180,in=90] (-.2,0);
  \draw[->,thick,darkred] (0.2,0) to[out=90,in=0] (0,.2);
 \draw[-,thick,darkred] (-.2,0) to[out=-90,in=180] (0,-0.2);
  \draw[-,thick,darkred] (0,-0.2) to[out=0,in=-90] (0.2,0);
 \node at (0,-.28) {$\scriptstyle{i}$};
   \node at (.6,0) {$\color{darkred}\scriptstyle{n+*}$};
      \node at (.2,0) {$\color{darkred}\bullet$};
\end{tikzpicture}
}
\!\!\mathord{
\begin{tikzpicture}[baseline = 0]
	\draw[->,thick,darkred] (0.08,-.4) to (0.08,.4);
     \node at (0.08,-.5) {$\scriptstyle{j}$};
\end{tikzpicture}
}
\!
{\scriptstyle\lambda},\\
\mathord{\begin{tikzpicture}[baseline = 0]
  \draw[<-,thick,darkred] (0,0.2) to[out=180,in=90] (-.2,0);
  \draw[-,thick,darkred] (0.2,0) to[out=90,in=0] (0,.2);
 \draw[-,thick,darkred] (-.2,0) to[out=-90,in=180] (0,-0.2);
  \draw[-,thick,darkred] (0,-0.2) to[out=0,in=-90] (0.2,0);
 \node at (0,-.28) {$\scriptstyle{i}$};
   \node at (-0.6,0) {$\color{darkred}\scriptstyle{n+*}$};
      \node at (-0.2,0) {$\color{darkred}\bullet$};
\end{tikzpicture}
}\:
\mathord{
\begin{tikzpicture}[baseline = 0]
	\draw[->,thick,darkred] (0.08,-.4) to (0.08,.4);
     \node at (0.08,-.5) {$\scriptstyle{j}$};
\end{tikzpicture}
}
{\scriptstyle\lambda}
&=
\mathord{
\begin{tikzpicture}[baseline = 0]
	\draw[->,thick,darkred] (0.08,-.4) to (0.08,.4);
     \node at (0.08,-.5) {$\scriptstyle{j}$};
\end{tikzpicture}
}
\!\mathord{\begin{tikzpicture}[baseline = 0]
  \draw[<-,thick,darkred] (0,0.2) to[out=180,in=90] (-.2,0);
  \draw[-,thick,darkred] (0.2,0) to[out=90,in=0] (0,.2);
 \draw[-,thick,darkred] (-.2,0) to[out=-90,in=180] (0,-0.2);
  \draw[-,thick,darkred] (0,-0.2) to[out=0,in=-90] (0.2,0);
 \node at (0,-.28) {$\scriptstyle{i}$};
   \node at (-.6,0) {$\color{darkred}\scriptstyle{n+*}$};
      \node at (-.2,0) {$\color{darkred}\bullet$};
\end{tikzpicture}
}\,
{\scriptstyle\lambda}.\label{bs8}
\end{align}
\end{itemize}
\end{proposition}

\begin{proof}
Let $h := \langle h_i,\lambda \rangle$ throughout the proof.

(i)--(ii)
When $i$ is even, this was already established in \cite{L}.
So we just need to prove (ii), assuming $i$ is odd.
We observe to start with that the identities (\ref{bs3})  and
(\ref{bs4}) (for fixed $\lambda$ and all $n \geq 0$) are
equivalent. To see this, let us rephrase them in terms of power
series.
We make $\End_{\UU(\g)}(E_i 1_\lambda)$
into a $\k[x]$-module
so that $x$ acts as
by vertically composing on top with a dot.
Let $t$ be an indeterminate and $\e(t) := \sum_{n \geq 0} \e_n t^n$,
$\h(t) := \sum_{n \geq 0} \h_n t^n$, which are power series in
$\Sym[[t]]$.
Recalling (\ref{goo1})--(\ref{goo2}),
the identities (\ref{bs3}) 
and (\ref{bs4}) for all 
$n \geq 0$ 
are equivalent to the generating function identities
\begin{align*}
\mathord{
\begin{tikzpicture}[baseline = 0]
	\draw[->,thick,darkred] (0.08,-.4) to (0.08,.4);
     \node at (0.08,-.5) {$\scriptstyle{i}$};
\end{tikzpicture}
}
\!{\scriptstyle\lambda}
\:\:
\beta_{\lambda;i}((1-\d t)\h(-t^2))
&= \left(\sum_{r \geq 0} (2r+1) x^{2r} t^{2r}\right)
\beta_{\lambda+\alpha_i;i}((1-\d t)\h(-t^2)) \mathord{
\begin{tikzpicture}[baseline = 0]
	\draw[->,thick,darkred] (0.08,-.4) to (0.08,.4);
     \node at (0.08,-.5) {$\scriptstyle{i}$};
\end{tikzpicture}
}
{\scriptstyle\lambda}
\,,\\
\beta_{\lambda+\alpha_i;i}((1+\d t)\e(t^2)) \mathord{
\begin{tikzpicture}[baseline = 0]
	\draw[->,thick,darkred] (0.08,-.4) to (0.08,.4);
     \node at (0.08,-.5) {$\scriptstyle{i}$};
\end{tikzpicture}
}
{\scriptstyle\lambda}
&= \left(\sum_{r \geq 0} (2r+1) x^{2r} t^{2r} \right)
\mathord{
\begin{tikzpicture}[baseline = 0]
	\draw[->,thick,darkred] (0.08,-.4) to (0.08,.4);
     \node at (0.08,-.5) {$\scriptstyle{i}$};
\end{tikzpicture}
}
{\scriptstyle\lambda}
\:\:
\beta_{\lambda;i}((1+\d t)\e(t^2)),
\end{align*}
respectively, as
follows by equating coefficients of $t$.
Since $\e(t) \h(-t) = 1$ in $\Sym$
and $\d^2 = 0$, we deduce that
$(1+\d t)\e(t^2)$ and $(1-\d t) \h(-t^2)$ are two-sided inverses.
Using this, it is easy to see that the two generating function
identities are indeed equivalent, e.g.
multiplying the first 
on the right by $\beta_{\lambda;i}((1+\d t)\e(t^2))$ and on the left by
$\beta_{\lambda+\alpha_i;i}((1- \d t)\e(t^2))$
transforms it into the second.

To complete the proof of (ii), 
we need to show that one of (\ref{bs3}) or (\ref{bs4}) holds for
each fixed $h$.
We proceed to verify (\ref{bs3}) in case $h \leq -1$;
a similar argument establishes (\ref{bs4}) in case $h \geq -1$.
So assume that $h \leq -1$.
The identity to be proved is trivial in case $n= 0$ so suppose moreover
that $n > 0$, so that $n-h-1 \geq 1$.
Then we have that
\begin{align*}
\mathord{

}
{\scriptstyle\lambda}.
\end{align*}
This is what we wanted.

(iii)--(iv)
By an argument with generating functions similar to the one explained
in the proof of (ii) above, the identities (\ref{bs5}) and
(\ref{bs6}) are equivalent, as are (\ref{bs7}) and (\ref{bs8}).
Therefore it suffices just to prove one of them for each fixed $h$ and
all $n \geq 0$.
For any $n \geq 0$, we have that
\begin{align*}
&
\mathord{
}\\
&\stackrel{(\ref{now})}{=}
\left\{
\begin{array}{ll}
t_{ij}\mathord{\begin{tikzpicture}[baseline=0]
  \draw[->,thick,darkred] (0.6,-0.5) to (0.6,.6);
   \node at (0.6,-.6) {$\scriptstyle{j}$};
   \node at (0.9,.-0.33) {$\scriptstyle{\lambda}$};
  \draw[->,thick,darkred] (0.3,0) to[out=90,in=0] (0,0.3);
  \draw[-,thick,darkred] (0,0.3) to[out=180,in=90] (-.3,0);
\draw[-,thick,darkred] (-.3,0) to[out=-90,in=180] (0,-0.3);
  \draw[-,thick,darkred] (0,-0.3) to[out=0,in=-90] (0.3,0);
 \node at (-0.4,0) {$\scriptstyle{i}$};
   \node at (.8,0) {$\color{darkred}\scriptstyle{n+d_{ij}}$};
      \node at (.3,0) {$\color{darkred}\bullet$};
\end{tikzpicture}
}+t_{ji}
\mathord{\begin{tikzpicture}[baseline=0]
  \draw[->,thick,darkred] (0.6,-0.5) to (0.6,.6);
   \node at (0.6,-.6) {$\scriptstyle{j}$};
   \node at (0.9,.-0.33) {$\scriptstyle{\lambda}$};
  \draw[->,thick,darkred] (0.3,0) to[out=90,in=0] (0,0.3);
  \draw[-,thick,darkred] (0,0.3) to[out=180,in=90] (-.3,0);
\draw[-,thick,darkred] (-.3,0) to[out=-90,in=180] (0,-0.3);
  \draw[-,thick,darkred] (0,-0.3) to[out=0,in=-90] (0.3,0);
 \node at (-0.4,0) {$\scriptstyle{i}$};
   \node at (.1,0) {$\color{darkred}\scriptstyle{n}$};
      \node at (.3,0) {$\color{darkred}\bullet$};
   \node at (.9,0) {$\color{darkred}\scriptstyle{d_{ji}}$};
      \node at (.6,0) {$\color{darkred}\bullet$};
\end{tikzpicture}
}+\displaystyle
\sum_{\substack{0 < p < d_{ij}\\0 < q < d_{ji}}}
s_{ij}^{pq}
\mathord{\begin{tikzpicture}[baseline=0]
  \draw[->,thick,darkred] (0.6,-0.5) to (0.6,.6);
   \node at (0.6,-.6) {$\scriptstyle{j}$};
   \node at (0.9,.-0.33) {$\scriptstyle{\lambda}$};
  \draw[->,thick,darkred] (0.3,0) to[out=90,in=0] (0,0.3);
  \draw[-,thick,darkred] (0,0.3) to[out=180,in=90] (-.3,0);
\draw[-,thick,darkred] (-.3,0) to[out=-90,in=180] (0,-0.3);
  \draw[-,thick,darkred] (0,-0.3) to[out=0,in=-90] (0.3,0);
 \node at (-0.4,0) {$\scriptstyle{i}$};
   \node at (0.1,0) {$\color{darkred}\scriptstyle{p}$};
      \node at (.3,0) {$\color{darkred}\bullet$};
   \node at (0.3,-.36) {$\color{darkred}\scriptstyle{n}$};
      \node at (.22,-.2) {$\color{darkred}\bullet$};
   \node at (.84,0) {$\color{darkred}\scriptstyle{q}$};
      \node at (.6,0) {$\color{darkred}\bullet$};
\end{tikzpicture}
}&\text{if $d_{ij} \neq 0$,}\\
\end{array}\right.
\end{align*}
This proves both (\ref{bs5}) and (\ref{bs7}) for $n \geq h+1$.
Also, the case $n=0$ follows from (\ref{a4}), 
hence, we are completely done if $h \leq 0$.
A similar argument establishes (\ref{bs6}) and (\ref{bs8}) for $n \geq
d_{ij}-h+1$,
hence, we are completely done if $h \geq d_{ij}$.

We are left with proving (\ref{bs5})--(\ref{bs6}) when $1 \leq h \leq
d_{ij}-1$.
We claim that (\ref{bs5}) holds for all $n \leq d_{ij}-h$.
The claim implies that (\ref{bs6}) holds for all $n \leq d_{ij}-h$ too, and we have
already established (\ref{bs6}) for $n \geq d_{ij}-h+1$, so the
claim is
enough to finish the proof.
For the claim,
we proceed by induction on $n = 0,1,\dots,d_{ij}-h$. The base case $n=0$
is trivial.
For the induction step, 
take $1 \leq n \leq
d_{ij}-h$.
By (\ref{lurking}), we have that
\begin{multline*}
\mathord{
\begin{tikzpicture}[baseline = 0]
	\draw[-,thick,darkred] (0.45,-.6) to [out=-90,in=0] (0.225,-.85);
	\draw[-,thick,darkred] (0.225,-0.85) to [out=180,in=-90] (0,-.6);
	\draw[-,thick,darkred] (-0.45,.6) to [out=-270,in=180] (-.225,.85);
	\draw[-,thick,darkred] (-0.225,.85) to [out=0,in=-270] (0,.6);
	\draw[<-,thick,darkred] (0.55,.85) to (-0.55,-.85);
	\draw[<-,thick,darkred] (0.45,-.6) to [in=-90,out=90] (-0.45,.6);
   \node at (-0.55,-.95) {$\scriptstyle{j}$};
  \node at (-0.55,.7) {$\scriptstyle{i}$};
   \node at (.5,0.4) {$\scriptstyle{\lambda}$};
     \node at (-0.02,0.7) {$\color{darkred}\bullet$};
      \node at (0.17,0.88) {$\color{darkred}\scriptstyle{n-1}$};
        \draw[-,thick,darkred] (0,-.6) to[out=90,in=-90] (.45,0);
        \draw[-,thick,darkred] (0.45,0) to[out=90,in=-90] (0,0.6);
\end{tikzpicture}
}-
\mathord{
\begin{tikzpicture}[baseline = 0]
	\draw[-,thick,darkred] (0.45,-.6) to [out=-90,in=0] (0.225,-.85);
	\draw[-,thick,darkred] (0.225,-0.85) to [out=180,in=-90] (0,-.6);
	\draw[-,thick,darkred] (-0.45,.6) to [out=-270,in=180] (-.225,.85);
	\draw[-,thick,darkred] (-0.225,.85) to [out=0,in=-270] (0,.6);
	\draw[<-,thick,darkred] (0.55,.85) to (-0.55,-.85);
	\draw[<-,thick,darkred] (0.45,-.6) to [in=-90,out=90] (-0.45,.6);
        \draw[-,thick,darkred] (0,-.6) to[out=90,in=-90] (-.45,0);
        \draw[-,thick,darkred] (-0.45,0) to[out=90,in=-90] (0,0.6);
   \node at (-0.55,-.95) {$\scriptstyle{j}$};
  \node at (-0.55,.7) {$\scriptstyle{i}$};
   \node at (.4,0) {$\scriptstyle{\lambda}$};
     \node at (-0.02,0.7) {$\color{darkred}\bullet$};
      \node at (0.17,0.88) {$\color{darkred}\scriptstyle{n-1}$};
\end{tikzpicture}
}
=
\displaystyle\sum_{\substack{r,s \geq 0\\r+s=d_{ij}-1}}
\!\!\!
(-1)^{|i|s}
t_{ij}
\!
\mathord{
\begin{tikzpicture}[baseline = 0]
	\draw[-,thick,darkred] (0.38,-.3) to [out=-90,in=0] (0.165,-.55);
	\draw[-,thick,darkred] (0.165,-0.55) to [out=180,in=-90] (-0.05,-.3);
	\draw[<-,thick,darkred] (-0.38,.4) to [out=-270,in=180] (-0.165,.65);
	\draw[-,thick,darkred] (-0.165,.65) to [out=0,in=-270] (0.05,.4);
	\draw[<-,thick,darkred] (0.38,-.4) to [out=90,in=0] (0.165,-.05);
	\draw[-,thick,darkred] (0.165,-0.05) to [out=180,in=90] (-0.05,-.4);
	\draw[-,thick,darkred] (-0.38,.5) to [out=270,in=180] (-0.165,.15);
	\draw[-,thick,darkred] (-0.165,.15) to [out=0,in=270] (0.05,.5);
	\draw[->,thick,darkred] (-0.38,-.55) to (.38,.65);
   \node at (0.25,-.6) {$\scriptstyle{i}$};
   \node at (-0.38,.66) {$\scriptstyle{i}$};
   \node at (-0.38,-.65) {$\scriptstyle{j}$};
   \node at (.4,.15) {$\scriptstyle{\lambda}$};
     \node at (0.02,0.3) {$\color{darkred}\bullet$};
     \node at (-0.04,-0.27) {$\color{darkred}\bullet$};
      \node at (-0.15,0.3) {$\color{darkred}\scriptstyle{r}$};
      \node at (0.1,-0.27) {$\color{darkred}\scriptstyle{s}$};
     \node at (0,0.55) {$\color{darkred}\bullet$};
      \node at (0.18,0.74) {$\color{darkred}\scriptstyle{n-1}$};
\end{tikzpicture}
}
+ \displaystyle
\!\!\sum_{\substack{0 < p < d_{ij}\\0 < q <
    d_{ji}\\r,s \geq 0\\r+s=p-1}}
\!\!\!\!(-1)^{|i|s}
s_{ij}^{pq}
\mathord{
\begin{tikzpicture}[baseline = 0]
	\draw[-,thick,darkred] (0.38,-.3) to [out=-90,in=0] (0.165,-.55);
	\draw[-,thick,darkred] (0.165,-0.55) to [out=180,in=-90] (-0.05,-.3);
	\draw[<-,thick,darkred] (-0.38,.4) to [out=-270,in=180] (-0.165,.65);
	\draw[-,thick,darkred] (-0.165,.65) to [out=0,in=-270] (0.05,.4);
	\draw[<-,thick,darkred] (0.38,-.3) to [out=90,in=0] (0.165,-.05);
	\draw[-,thick,darkred] (0.165,-0.05) to [out=180,in=90] (-0.05,-.3);
	\draw[-,thick,darkred] (-0.38,.4) to [out=270,in=180] (-0.165,.15);
	\draw[-,thick,darkred] (-0.165,.15) to [out=0,in=270] (0.05,.4);
	\draw[->,thick,darkred] (-0.38,-.55) to (.38,.65);
   \node at (0.45,-.5) {$\scriptstyle{i}$};
   \node at (-0.4,.67) {$\scriptstyle{i}$};
   \node at (-0.38,-.65) {$\scriptstyle{j}$};
   \node at (.4,.15) {$\scriptstyle{\lambda}$};
     \node at (0,0.05) {$\color{darkred}\bullet$};
     \node at (-0.04,-0.27) {$\color{darkred}\bullet$};
      \node at (0.1,-0.27) {$\color{darkred}\scriptstyle{s}$};
      \node at (0.15,0.07) {$\color{darkred}\scriptstyle{q}$};
     \node at (0.02,0.3) {$\color{darkred}\bullet$};
      \node at (-0.15,0.3) {$\color{darkred}\scriptstyle{r}$};
     \node at (0,0.55) {$\color{darkred}\bullet$};
      \node at (0.18,0.74) {$\color{darkred}\scriptstyle{n-1}$};
\end{tikzpicture}
}.
\end{multline*}
Both terms on the left hand side are zero: for the first this follows
immediately from (\ref{startb}); for the second one this follows from
(\ref{startd}) and (\ref{bubble2}) on applying
(\ref{rtcross1}) to pull the dots past the crossing.
Replacing $s$ by $s+h-1$, we have proved that
$$
\sum_{s=0}^n (-1)^{|i|s}
\left(
t_{ij}
\mathord{\begin{tikzpicture}[baseline = -10]
  \draw[-,thick,darkred] (0,0.2) to[out=180,in=90] (-.2,0);
  \draw[->,thick,darkred] (0.2,0) to[out=90,in=0] (0,.2);
 \draw[-,thick,darkred] (-.2,0) to[out=-90,in=180] (0,-0.2);
  \draw[-,thick,darkred] (0,-0.2) to[out=0,in=-90] (0.2,0);
 \node at (0,-.28) {$\scriptstyle{i}$};
   \node at (.75,0) {$\color{darkred}\scriptstyle{n-s+*}$};
      \node at (.2,0) {$\color{darkred}\bullet$};
\end{tikzpicture}
}\!\!\!
\mathord{
\begin{tikzpicture}[baseline = 0]
	\draw[->,thick,darkred] (0.08,-.8) to (0.08,.8);
     \node at (0.08,-.9) {$\scriptstyle{j}$};
\end{tikzpicture}
}\!\!\!
\mathord{\begin{tikzpicture}[baseline = 10]
  \draw[<-,thick,darkred] (0,0.2) to[out=180,in=90] (-.2,0);
  \draw[-,thick,darkred] (0.2,0) to[out=90,in=0] (0,.2);
 \draw[-,thick,darkred] (-.2,0) to[out=-90,in=180] (0,-0.2);
  \draw[-,thick,darkred] (0,-0.2) to[out=0,in=-90] (0.2,0);
 \node at (0,-.28) {$\scriptstyle{i}$};
   \node at (-.55,0) {$\color{darkred}\scriptstyle{s+*}$};
      \node at (-.2,0) {$\color{darkred}\bullet$};
\end{tikzpicture}
}\!\!\!
{\scriptstyle\lambda}
+ 
\displaystyle
\!\!\sum_{\substack{0 < p < d_{ij}\\0 < q <
    d_{ji}}}
s_{ij}^{pq}
\mathord{\begin{tikzpicture}[baseline = -10]
  \draw[-,thick,darkred] (0,0.2) to[out=180,in=90] (-.2,0);
  \draw[->,thick,darkred] (0.2,0) to[out=90,in=0] (0,.2);
 \draw[-,thick,darkred] (-.2,0) to[out=-90,in=180] (0,-0.2);
  \draw[-,thick,darkred] (0,-0.2) to[out=0,in=-90] (0.2,0);
 \node at (0,-.28) {$\scriptstyle{i}$};
   \node at (1.2,0) {$\color{darkred}\scriptstyle{n-s+p-d_{ij}+*}$};
      \node at (.2,0) {$\color{darkred}\bullet$};
\end{tikzpicture}
}\!\!\!
\mathord{
\begin{tikzpicture}[baseline = 0]
	\draw[->,thick,darkred] (0.08,-.8) to (0.08,.8);
   \node at (0.25,0) {$\color{darkred}\scriptstyle{q}$};
      \node at (0.08,0) {$\color{darkred}\bullet$};
     \node at (0.08,-.9) {$\scriptstyle{j}$};
\end{tikzpicture}
}\!\!\!
\mathord{\begin{tikzpicture}[baseline = 10]
  \draw[<-,thick,darkred] (0,0.2) to[out=180,in=90] (-.2,0);
  \draw[-,thick,darkred] (0.2,0) to[out=90,in=0] (0,.2);
 \draw[-,thick,darkred] (-.2,0) to[out=-90,in=180] (0,-0.2);
  \draw[-,thick,darkred] (0,-0.2) to[out=0,in=-90] (0.2,0);
 \node at (0,-.28) {$\scriptstyle{i}$};
   \node at (-.55,0) {$\color{darkred}\scriptstyle{s+*}$};
      \node at (-.2,0) {$\color{darkred}\bullet$};
\end{tikzpicture}
}\!\!\!
{\scriptstyle\lambda}
\:\right)
= 0.
$$
Also because $n < d_{ij}$ we have that
$$
\sum_{s=0}^n (-1)^{|i|s} \left(t_{ji}
\mathord{\begin{tikzpicture}[baseline = -10]
  \draw[-,thick,darkred] (0,0.2) to[out=180,in=90] (-.2,0);
  \draw[->,thick,darkred] (0.2,0) to[out=90,in=0] (0,.2);
 \draw[-,thick,darkred] (-.2,0) to[out=-90,in=180] (0,-0.2);
  \draw[-,thick,darkred] (0,-0.2) to[out=0,in=-90] (0.2,0);
 \node at (0,-.28) {$\scriptstyle{i}$};
   \node at (1.1,0) {$\color{darkred}\scriptstyle{n-s-d_{ij}+*}$};
      \node at (.2,0) {$\color{darkred}\bullet$};
\end{tikzpicture}
}\!\!\!
\mathord{
\begin{tikzpicture}[baseline = 0]
	\draw[->,thick,darkred] (0.08,-.8) to (0.08,.8);
     \node at (0.08,-.9) {$\scriptstyle{j}$};
\end{tikzpicture}
}\!\!\!
\mathord{\begin{tikzpicture}[baseline = 10]
  \draw[<-,thick,darkred] (0,0.2) to[out=180,in=90] (-.2,0);
  \draw[-,thick,darkred] (0.2,0) to[out=90,in=0] (0,.2);
 \draw[-,thick,darkred] (-.2,0) to[out=-90,in=180] (0,-0.2);
  \draw[-,thick,darkred] (0,-0.2) to[out=0,in=-90] (0.2,0);
 \node at (0,-.28) {$\scriptstyle{i}$};
   \node at (-.55,0) {$\color{darkred}\scriptstyle{s+*}$};
      \node at (-.2,0) {$\color{darkred}\bullet$};
\end{tikzpicture}
}\!\!\!
{\scriptstyle\lambda}\right)
=0.
$$
Now consider the identity obtained by adding these two expressions together.
We use the induction hypothesis
(\ref{bs5}) to simplify
all of the terms with $s \geq 1$, keeping the $s=0$ terms on the left
hand side,
to obtain
\begin{multline*}
t_{ij}
\,\mathord{\begin{tikzpicture}[baseline = 0]
  \draw[-,thick,darkred] (0,0.2) to[out=180,in=90] (-.2,0);
  \draw[->,thick,darkred] (0.2,0) to[out=90,in=0] (0,.2);
 \draw[-,thick,darkred] (-.2,0) to[out=-90,in=180] (0,-0.2);
  \draw[-,thick,darkred] (0,-0.2) to[out=0,in=-90] (0.2,0);
 \node at (0,-.28) {$\scriptstyle{i}$};
   \node at (.6,0) {$\color{darkred}\scriptstyle{n+*}$};
      \node at (.2,0) {$\color{darkred}\bullet$};
\end{tikzpicture}
}
\!\!\mathord{
\begin{tikzpicture}[baseline = 0]
	\draw[->,thick,darkred] (0.08,-.4) to (0.08,.4);
     \node at (0.08,-.5) {$\scriptstyle{j}$};
\end{tikzpicture}
}
\!
{\scriptstyle\lambda}
+
t_{ji}\,
\mathord{\begin{tikzpicture}[baseline = 0]
  \draw[-,thick,darkred] (0,0.2) to[out=180,in=90] (-.2,0);
  \draw[->,thick,darkred] (0.2,0) to[out=90,in=0] (0,.2);
 \draw[-,thick,darkred] (-.2,0) to[out=-90,in=180] (0,-0.2);
  \draw[-,thick,darkred] (0,-0.2) to[out=0,in=-90] (0.2,0);
 \node at (0,-.28) {$\scriptstyle{i}$};
   \node at (.85,0) {$\color{darkred}\scriptstyle{n-d_{ij}+*}$};
      \node at (.2,0) {$\color{darkred}\bullet$};
\end{tikzpicture}
}
\!\!\mathord{
\begin{tikzpicture}[baseline = 0]
	\draw[->,thick,darkred] (0.08,-.4) to (0.08,.4);
   \node at (.4,0) {$\color{darkred}\scriptstyle{d_{ji}}$};
      \node at (.08,0) {$\color{darkred}\bullet$};
     \node at (0.08,-.5) {$\scriptstyle{j}$};
\end{tikzpicture}
}
\!
{\scriptstyle\lambda}
+
\!\!\displaystyle
\sum_{\substack{0 < p < d_{ij}\\0 < q < d_{ji}}}
\!\!
s_{ij}^{pq}
\:\mathord{\begin{tikzpicture}[baseline = 0]
  \draw[-,thick,darkred] (0,0.2) to[out=180,in=90] (-.2,0);
  \draw[->,thick,darkred] (0.2,0) to[out=90,in=0] (0,.2);
 \draw[-,thick,darkred] (-.2,0) to[out=-90,in=180] (0,-0.2);
  \draw[-,thick,darkred] (0,-0.2) to[out=0,in=-90] (0.2,0);
 \node at (0,-.28) {$\scriptstyle{i}$};
   \node at (1.05,0) {$\color{darkred}\scriptstyle{n+p-d_{ij}+*}$};
      \node at (.2,0) {$\color{darkred}\bullet$};
\end{tikzpicture}
}
\mathord{
\begin{tikzpicture}[baseline = 0]
	\draw[->,thick,darkred] (0.08,-.4) to (0.08,.4);
   \node at (.3,0) {$\color{darkred}\scriptstyle{q}$};
      \node at (.08,0) {$\color{darkred}\bullet$};
     \node at (0.08,-.5) {$\scriptstyle{j}$};
\end{tikzpicture}
}
\!
{\scriptstyle\lambda}\\
=
-\sum_{s=1}^n (-1)^{|i|s}
\mathord{
\begin{tikzpicture}[baseline = 0]
	\draw[->,thick,darkred] (0.08,-.5) to (0.08,.8);
     \node at (0.08,-.6) {$\scriptstyle{j}$};
\end{tikzpicture}
}
\begin{array}{l}
\hspace{1mm}\mathord{
\begin{tikzpicture}[baseline = 0]
  \draw[-,thick,darkred] (0,0.4) to[out=180,in=90] (-.2,0.2);
  \draw[->,thick,darkred] (0.2,0.2) to[out=90,in=0] (0,.4);
 \draw[-,thick,darkred] (-.2,0.2) to[out=-90,in=180] (0,0);
  \draw[-,thick,darkred] (0,0) to[out=0,in=-90] (0.2,0.2);
 \node at (-.28,.13) {$\scriptstyle{i}$};
  \node at (0.2,0.2) {$\color{darkred}\bullet$};
   \node at (0.75,0.2) {$\color{darkred}\scriptstyle{n-s+*}$};
\end{tikzpicture}
}
\\
\hspace{-4mm}
\mathord{
\begin{tikzpicture}[baseline = 0]
  \draw[-,thick,darkred] (0.2,0.2) to[out=90,in=0] (0,.4);
  \draw[<-,thick,darkred] (0,0.4) to[out=180,in=90] (-.2,0.2);
\draw[-,thick,darkred] (-.2,0.2) to[out=-90,in=180] (0,0);
  \draw[-,thick,darkred] (0,0) to[out=0,in=-90] (0.2,0.2);
   \node at (-0.2,0.2) {$\color{darkred}\bullet$};
   \node at (-0.6,0.2) {$\color{darkred}\scriptstyle{s+*}$};
 \node at (.28,.13) {$\scriptstyle{i}$};
   \node at (0.4,0.4) {$\scriptstyle{\lambda}$};
\end{tikzpicture}
}
\end{array}
\stackrel{(\ref{ig3})}{=}
\mathord{
\begin{tikzpicture}[baseline = 0]
	\draw[->,thick,darkred] (0.08,-.4) to (0.08,.4);
     \node at (0.08,-.5) {$\scriptstyle{j}$};
\end{tikzpicture}
}
\!{\scriptstyle\lambda}
\:\mathord{\begin{tikzpicture}[baseline = 0]
  \draw[-,thick,darkred] (0,0.2) to[out=180,in=90] (-.2,0);
  \draw[->,thick,darkred] (0.2,0) to[out=90,in=0] (0,.2);
 \draw[-,thick,darkred] (-.2,0) to[out=-90,in=180] (0,-0.2);
  \draw[-,thick,darkred] (0,-0.2) to[out=0,in=-90] (0.2,0);
 \node at (0,-.28) {$\scriptstyle{i}$};
   \node at (0.6,0) {$\color{darkred}\scriptstyle{n+*}$};
      \node at (0.2,0) {$\color{darkred}\bullet$};
\end{tikzpicture}
}.
\end{multline*}
This completes the proof of the claim.
\end{proof}

\begin{corollary}
For $i,j \in I$ with $i$ odd, we have that
\begin{equation}\label{centrality2}
\mathord{
\begin{tikzpicture}[baseline = 0]
  \draw[-,thick,darkred] (0,0.2) to[out=180,in=90] (-.2,0);
  \draw[-,thick,darkred] (0.2,0) to[out=90,in=0] (0,.2);
 \draw[-,thick,darkred] (-.2,0) to[out=-90,in=180] (0,-0.2);
  \draw[-,thick,darkred] (0,-0.2) to[out=0,in=-90] (0.2,0);
  \draw[-,thick,darkred] (0.14,-0.15) to (-0.14,0.13);
  \draw[-,thick,darkred] (0.14,0.13) to (-0.14,-0.15);
 \node at (-.28,-.07) {$\scriptstyle{i}$};
   \node at (.4,-.5) {$\scriptstyle{j}$};
   \node at (.6,0) {$\scriptstyle{\lambda}$};
	\draw[->,thick,darkred] (.4,-.4) to (.4,.4);
\end{tikzpicture}
}=
\mathord{
\begin{tikzpicture}[baseline = 0]
  \draw[-,thick,darkred] (0,0.2) to[out=180,in=90] (-.2,0);
  \draw[-,thick,darkred] (0.2,0) to[out=90,in=0] (0,.2);
 \draw[-,thick,darkred] (-.2,0) to[out=-90,in=180] (0,-0.2);
  \draw[-,thick,darkred] (0,-0.2) to[out=0,in=-90] (0.2,0);
  \draw[-,thick,darkred] (0.14,-0.15) to (-0.14,0.13);
  \draw[-,thick,darkred] (0.14,0.13) to (-0.14,-0.15);
 \node at (-.28,-.07) {$\scriptstyle{i}$};
   \node at (-.4,-.5) {$\scriptstyle{j}$};
   \node at (.4,0) {$\scriptstyle{\lambda}$};
	\draw[->,thick,darkred] (-.4,-.4) to (-.4,.4);
\end{tikzpicture}
},
\qquad
\mathord{
\begin{tikzpicture}[baseline = 0]
  \draw[-,thick,darkred] (0,0.2) to[out=180,in=90] (-.2,0);
  \draw[-,thick,darkred] (0.2,0) to[out=90,in=0] (0,.2);
 \draw[-,thick,darkred] (-.2,0) to[out=-90,in=180] (0,-0.2);
  \draw[-,thick,darkred] (0,-0.2) to[out=0,in=-90] (0.2,0);
  \draw[-,thick,darkred] (0.14,-0.15) to (-0.14,0.13);
  \draw[-,thick,darkred] (0.14,0.13) to (-0.14,-0.15);
 \node at (-.28,-.07) {$\scriptstyle{i}$};
   \node at (.4,.5) {$\scriptstyle{j}$};
   \node at (.6,0) {$\scriptstyle{\lambda}$};
	\draw[<-,thick,darkred] (.4,-.4) to (.4,.4);
\end{tikzpicture}
}=
\mathord{
\begin{tikzpicture}[baseline = 0]
  \draw[-,thick,darkred] (0,0.2) to[out=180,in=90] (-.2,0);
  \draw[-,thick,darkred] (0.2,0) to[out=90,in=0] (0,.2);
 \draw[-,thick,darkred] (-.2,0) to[out=-90,in=180] (0,-0.2);
  \draw[-,thick,darkred] (0,-0.2) to[out=0,in=-90] (0.2,0);
  \draw[-,thick,darkred] (0.14,-0.15) to (-0.14,0.13);
  \draw[-,thick,darkred] (0.14,0.13) to (-0.14,-0.15);
 \node at (-.28,-.07) {$\scriptstyle{i}$};
   \node at (-.4,.5) {$\scriptstyle{j}$};
   \node at (.4,0) {$\scriptstyle{\lambda}$};
	\draw[<-,thick,darkred] (-.4,-.4) to (-.4,.4);
\end{tikzpicture}
}.
\end{equation}
\end{corollary}

\begin{proof}
Remembering the definition (\ref{oddbubble}),
the first relation follows from the $n=1$ cases of (\ref{bs3}), 
(\ref{bs5}) and (\ref{bs7});
to see that the lower terms in (\ref{bs5}) vanish,
recall that 
$d_{ij}$ is even. Hence, it satisfies $d_{ij} \geq 2$,
and $s_{ij}^{pq} = 0$ if $p = d_{ij}-1$.
The second relation follows from the first by applying $\T$.
\end{proof}

\begin{remark}
One can invert the formulae in Proposition~\ref{bs} to obtain also
various bubble slides in the other direction. For example, 
inverting (\ref{bs1})--(\ref{bs4}) produces the following, for $i$ even, $i$
even, $i$ odd and $i$ odd, respectively:
\begin{align}
\mathord{

}
\!
{\scriptstyle\lambda}.
\end{align}
\end{remark}

\begin{proposition}\label{altbraidprop}
The following relation holds:
\begin{align}
&\mathord{

}\notag\\
&\quad=
\left\{
\begin{array}{ll}
\displaystyle\sum_{r,s,t \geq 0}
(-1)^{|i|(\langle h_i,\lambda\rangle+r+s+1)}\mathord{
\begin{tikzpicture}[baseline = 0]
	\draw[-,thick,darkred] (0.3,0.7) to[out=-90, in=0] (0,0.3);
	\draw[->,thick,darkred] (0,0.3) to[out = 180, in = -90] (-0.3,0.7);
    \node at (0.3,.8) {$\scriptstyle{i}$};
    \node at (0.98,-.8) {$\scriptstyle{i}$};
    \node at (1.4,-0.32) {$\scriptstyle{\lambda}$};
  \draw[->,thick,darkred] (0.2,0) to[out=90,in=0] (0,0.2);
  \draw[-,thick,darkred] (0,0.2) to[out=180,in=90] (-.2,0);
\draw[-,thick,darkred] (-.2,0) to[out=-90,in=180] (0,-0.2);
  \draw[-,thick,darkred] (0,-0.2) to[out=0,in=-90] (0.2,0);
 \node at (-0.3,0) {$\scriptstyle{i}$};
   \node at (0.2,0) {$\color{darkred}\bullet$};
   \node at (.96,0) {$\color{darkred}\scriptstyle{-r-s-t-3}$};
   \node at (-0.23,0.43) {$\color{darkred}\bullet$};
   \node at (-0.43,0.43) {$\color{darkred}\scriptstyle{r}$};
	\draw[<-,thick,darkred] (0.3,-.7) to[out=90, in=0] (0,-0.3);
	\draw[-,thick,darkred] (0,-0.3) to[out = 180, in = 90] (-0.3,-.7);
    \node at (-0.3,-.8) {$\scriptstyle{i}$};
   \node at (-0.25,-0.5) {$\color{darkred}\bullet$};
   \node at (-.4,-.5) {$\color{darkred}\scriptstyle{s}$};
	\draw[->,thick,darkred] (.98,-0.7) to (.98,0.7);
   \node at (0.98,0.53) {$\color{darkred}\bullet$};
   \node at (1.15,0.53) {$\color{darkred}\scriptstyle{t}$};
\end{tikzpicture}
}&\text{if $i=j=k$,}\\
\end{array}
\right.
\label{altbraid}
\end{align}
\end{proposition}

\begin{proof}
  Assuming either $i=j=k$ or $i \neq k$,
we attach crossings to the top left and bottom right 
pairs of strands of (\ref{lurking}) to deduce that
\begin{equation}\label{birdy}
\mathord{
%
}.
\end{equation}
When $i \neq k$, the lemma follows easily from this on
simplifying 
using (\ref{rory}).
A similar argument treats the case $i \neq j$, attaching
crossings to the top right and bottom left pairs of strands in
the
relation
$$
\mathord{
%
}
$$
which may be deduced by attaching a leftward cap to the top left and a
leftward cup to the bottom right of (\ref{qhalast}) and using
(\ref{adjfinal}) and (\ref{final1}).
We are just left with the case that $i=j=k$.
For this we use (\ref{birdy}) again
to reduce to checking:
\begin{align*}
\mathord{
%
}.
\end{align*}
These two identities are proved in similar ways. One first uses
(\ref{pos})--(\ref{neg})
to reduce the double crossings, then
(\ref{upcross1})--(\ref{upcross2}) to pull the dots to the boundary,
remembering also (\ref{startb})--(\ref{startd}),
(\ref{rightpitchfork}) 
and (\ref{now}). By now we can safely leave the
details to the reader!
\end{proof}

\section{The nondegeneracy conjecture}\label{s6}

The main result of this section is
a generalization of \cite[Proposition 3.11]{KL3}.
We need some further notation which is adapted
from \cite{KL3}. 
Let
$\SSeq$ be the set of all 
words 
in the alphabet $\{{\color{darkred}{\uparrow}}_i, {\color{darkred}{\downarrow}}_i\:|\:i \in I\}$; our words
correspond to the {\em signed sequences} of \cite{KL3}.
For $\a = \a_m\cdots\a_1 \in \SSeq$,
let
\begin{equation}
\wt(\a) := \sum_{i \in I} 
\Big(\#\{n=1,\dots,m\:|\:\a_n = {\color{darkred}{\uparrow}}_i\}-
\#\{n=1,\dots,m\:|\:\a_n = {\color{darkred}{\downarrow}}_i\}\Big)
\alpha_i
\in Q.
\end{equation}
To $\lambda \in P$ and 
$\a = \a_m\cdots\a_1 \in \SSeq$, 
we associate 
the 1-morphism
\begin{equation}\label{safetypin}
E_\a 1_\lambda := E_{\a_m} \cdots E_{\a_1} 1_\lambda:\lambda
\rightarrow \lambda + \wt(\a)
\end{equation}
in $\UU(\g)$,
with the convention that $E_{{\color{darkred}{\uparrow}}_i}=E_i$
and $E_{{\color{darkred}{\downarrow}}_i}=F_i$.
As $\lambda$ and $\a$ vary, these give all of the 
$1$-morphisms in $\UU(\g)$.

Suppose that we are given 
$\a = \a_m\cdots \a_1$ and $\b = \b_n\cdots \b_1 \in \SSeq$.
An {\em $\a\b$-matching}
is a planar diagram with
\begin{itemize}
\item
$m$ distinct vertices on a horizontal axis at the bottom
labeled from right to left by the letters $\a_1,\dots,\a_m$;
\item 
$n$ distinct vertices on a horizontal axis at the top 
labeled from right to left by the letters $\b_1,\dots,\b_n$;
\item $(m+n)/2$ smoothly
immersed directed
$I$-colored strands
drawn between the horizontal axes whose endpoints are 
the given $(m+n)$ vertices.
\end{itemize}
We require moreover that:
\begin{itemize}
\item
the strands have only finitely many intersections and critical
points ($=$ points of slope zero);
\item 
there are no intersections at critical points, no triple
intersections, and no tangencies;
\item
the colors and directions of the strands are
consistent
with the letters at their endpoints.
\end{itemize}
Note at least one $\a\b$-matching exists if and only if
$\wt(\a)=\wt(\b)$.
Here is an example with $\a = {\color{darkred}\uparrow}_j\, {\color{darkred}\downarrow}_j\, {\color{darkred}\uparrow}_i\,
{\color{darkred}\downarrow}_k$ and $\b = {\color{darkred}\uparrow}_i\, {\color{darkred}\downarrow}_k\, {\color{darkred}\uparrow}_i\,{\color{darkred}\downarrow}_i$:
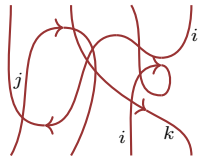
$$
\mathord{\begin{tikzpicture}[baseline = 0]
\draw[-,thick,darkred] (0,1) to [out=-90,in=180] (.45,-0.6);
\draw[<-,thick,darkred] (.45,-.6) to [out=0,in=180] (1.4,0.6);
\draw[-,thick,darkred] (1.4,.6) to [out=0,in=180] (2,0.3);
\draw[-,thick,darkred] (2,.3) to [out=0,in=-90] (2.4,1);
\draw[->,thick,darkred] (0,-1) to [out=60,in=180] (.7,.7);
\draw[-,thick,darkred] (.7,.7) to [out=0,in=60] (0.8,-1);
\draw[->,thick,darkred] (0.8,1) to [out=-90,in=150] (1.8,-0.4);
\draw[-,thick,darkred] (1.8,-0.4) to [out=-30,in=90] (2.4,-1);
\draw[->,thick,darkred] (1.6,-1) to [out=90,in=-180] (2,0.2);
\draw[-,thick,darkred] (2,0.2) to [out=0,in=0] (2,-.2);
\draw[-,thick,darkred] (2,-.2) to [out=180,in=-60] (1.6,1);
    \node at (1.48,-.75) {$\scriptstyle{i}$};
    \node at (2.1,-.7) {$\scriptstyle{k}$};
    \node at (2.44,.6) {$\scriptstyle{i}$};
    \node at (.1,0) {$\scriptstyle{j}$};
\end{tikzpicture}}
$$
A matching is {\em reduced} if each strand has 
at most one 
critical point which should either be a minimum or a maximum, 
there are no self-intersections of strands, and
distinct strands intersect at most once.

Any $\a\b$-matching defines a pairing between the letters of the words
$\a$ and $\b$, two letters being paired if they are endpoints of
the same strand.
We say that two matchings
are {\em equivalent}
if they define the same pairing.
Every matching is equivalent to at least one reduced matching.
For example, here
is a
reduced matching equivalent to the matching displayed above:
$$
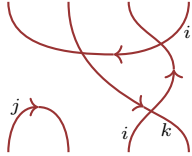
\mathord{\begin{tikzpicture}[baseline = 0]
\draw[->,thick,darkred] (0,-1) to [out=90,in=180] (.4,-.4);
\draw[-,thick,darkred] (.4,-.4) to [out=0,in=90] (0.8,-1);
\draw[-,thick,darkred] (0,1) to [out=-90,in=180] (1.4,0.3);
\draw[<-,thick,darkred] (1.4,0.3) to [out=0,in=-90] (2.4,1);
\draw[->,thick,darkred] (0.8,1) to [out=-90,in=150] (1.8,-0.4);
\draw[-,thick,darkred] (1.8,-0.4) to [out=-30,in=90] (2.4,-1);
\draw[->,thick,darkred] (1.6,-1) to [out=90,in=-90] (2.2,0.1);
\draw[-,thick,darkred] (2.2,0.1) to [out=90,in=-90] (1.6,1);
    \node at (1.55,-.75) {$\scriptstyle{i}$};
    \node at (2.1,-.7) {$\scriptstyle{k}$};
    \node at (2.38,.6) {$\scriptstyle{i}$};
    \node at (.1,-0.4) {$\scriptstyle{j}$};
\end{tikzpicture}}
$$

A {\em decorated $\a\b$-matching}
is an $\a\b$-matching
whose strands have been decorated by finitely many 
dots
located away from intersections and critical points, each of which is
labeled by a non-negative integer.
Given any decorated $\a\b$-matching $\sigma$ and $\lambda \in P$,
there is a unique way to label the regions of $\sigma$ by elements of $P$
so that it becomes
the diagrammatic representation of a 
2-morphism $f(\sigma,\lambda)
\in\Hom_{\UU(\g)}(E_\a 1_\lambda, E_\b 1_\lambda)$
as above.

For each $\a,\b \in \SSeq$, we
choose a set $M(\a,\b)$ of representatives 
for the equivalence classes
of reduced $\a\b$-matchings.
For each element of
$M(\a,\b)$, we also choose
a distinguished point on each of its strands located away from
intersections and critical points.
Then let $\widehat{M}(\a,\b)$ be the set of decorated $\a\b$-matchings
obtained by taking each of the matchings in $M(\a,\b)$ and 
putting a dot labeled with a non-negative integer at each of its distinguished points.
Finally recall the homorphism
$\beta_\lambda:\SYM
\rightarrow 
\End_{\UU(\g)}(1_\lambda)$ from (\ref{beta}).

\begin{theorem}\label{tripleofdata}
Take $\a,\b \in \SSeq$ with $\wt(\a) = \wt(\b)$ and any $\lambda \in P$.
Viewing $\Hom_{\UU(\g)}(E_\a 1_\lambda,
E_\b 1_\lambda)$ as a right $\SYM$-module
so that $p \in \SYM$ acts by horizontally composing on the right with
$\beta_\lambda(p)$, 
the 2-morphisms
$\big\{f(\sigma,\lambda)\:\big|\:\sigma \in \widehat{M}(\a,\b)\big\}$
generate $\Hom_{\UU(\g)}(E_\a 1_\lambda, E_\b
1_\lambda)$
as a right $\SYM$-module.
\label{spanning}
\end{theorem}

\begin{proof}
By the definitions, any 2-morphism in
$\Hom_{\UU(\g)}(E_\a 1_\lambda, E_\b 1_\lambda)$
is a linear combination of diagrams obtained by horizontally and
vertically composing the generators $x,\tau,\eta,\eps,\eta'$ and $\eps'$.
Now the point is that 
we have derived enough relations above to be able to algorithmically rewrite any 2-morphism 
represented by such a 
diagram as a linear combination of the 2-morphisms
$f(\sigma,\lambda) \beta_\lambda(p)$
for $\sigma \in \widehat{M}(\a,\b)$ and $p \in \SYM$.
This proceeds by induction on the total number of crossings in the diagram.
We omit the details since it is essentially the same argument as used to prove
\cite[Proposition 3.11]{KL3}.
\end{proof}

Now we can properly state the Nondegeneracy Conjecture from the introduction:

\vspace{2mm}
\noindent
{\bf Nondegeneracy Conjecture.}
{\em 
For all $\a,\b \in \SSeq$
with $\wt(\a)=\wt(\b)$ and
any $\lambda \in P$, the superspace
$\Hom_{\UU(\g)}(E_\a 1_\lambda, E_\b 1_\lambda)$ is a free
right $\SYM$-module
with basis
$\big\{f(\sigma,\lambda)\:\big|\:\sigma \in \widehat{M}(\a,\b)\big\}$.}

\section{The covering quantum group}\label{sqg}

Henceforth, we assume that the Cartan matrix is symmetrized by
positive integers
$(d_i)_{i \in I}$, and that the
parameters are chosen to satisfy the
homogeneity condition (\ref{hc}).
Let $(-,-)$ be the
symmetric bilinear form on the root lattice 
$Q$ defined from $(\alpha_i,\alpha_j) :=
-d_i d_{ij}$.
In this section, we recall the definition of the covering quantum
group
$\dot U_{q,\pi}(\g)$ of Clark, Hill and Wang \cite{CHW1, CHW2}. Our exposition is based
mostly on \cite{CFLW} and \cite{Clark}.
Note that our $q$ is the parameter denoted $q^{-1}$ in \cite{CHW1,
  CHW2, 
CFLW}, which is $v^{-1}$ in \cite{Clark}. We write $e_i 1_\lambda$
and $f_i 1_\lambda$ in place of $E_i 1_\lambda$ and $F_i 1_\lambda$; we
would also write $k_i$ for the generator $K_i^{-1}$
although we won't actually need this here.
In \cite{CHW2, CFLW, Clark}, an additional assumption 
of ``bar-consistency'' is made on the super Cartan
datum; we do not insist on this until later.

Let $\mathbb{L}$ be the ring $\Q(q)[\pi] / (\pi^2-1)$,
and $\mathcal{L} := \Z[q,q^{-1},\pi] / (\pi^2-1)$ as in the introduction.
For $n \in \Z$, we let
$$
[n]_{q,\pi} := \frac{q^n - (\pi q)^{-n}}{q - (\pi q)^{-1}}
= \left\{
\begin{array}{ll}
q^{n-1} + \pi q^{n-3}+\cdots + \pi^{n-1} q^{1-n}&\text{if $n \geq 0$,}\\
-\pi^n (q^{-n-1}+\pi q^{-n-3}+\cdots+\pi^{-n-1} q^{1+n})&\text{if $n \leq 0$.}
\end{array}\right.
$$
There are corresponding quantum factorials and binomial coefficients:
$$
[n]_{q,\pi}^! := 
[n]_{q,\pi} [n-1]_{q,\pi}\cdots [1]_{q,\pi},
\qquad
\sqbinom{n}{r}_{q,\pi} := 
\frac{[n]_{q,\pi}^!}{[r]_{q,\pi}^! [n-r]_{q,\pi}^!}.
$$
We let $-$ be the involution of $\mathbb{L}$ (or $\mathcal{L}$) with
$\overline{q} = q^{-1}$ and $\overline{\pi} = \pi$. Note this is
different from the bar involution used in \cite{CFLW, Clark}; in particular,
our quantum integers are {\em not} bar invariant, but satisfy
\begin{equation}
\overline{[n]}_{q,\pi}
= \pi^{n-1} [n]_{q,\pi} = -\pi [-n]_{q,\pi}.
\end{equation}
We have that
$\overline{\sqbinom{n}{r}}_{q,\pi} =
\pi^{r(n-r)}\sqbinom{n}{r}_{q,\pi}$,
so that the quantum binomial coefficient is bar invariant if $n$
is odd.
 For $i \in I$, we set
$$
q_i := q^{d_i},
\qquad
\pi_i := \pi^{|i|}.
$$

Let $\dot U_{q,\pi}(\g)$ be the locally unital $\mathbb{L}$-algebra with
mutually orthogonal
idempotents $\{1_\lambda\:|\:\lambda \in P\}$,
and generators $e_i 1_\lambda = 1_{\lambda+\alpha_i} e_i$ and
$f_i 1_\lambda = 1_{\lambda-\alpha_i} f_i$ for all $i \in I$ and
$\lambda \in P$,
subject to the following relations:
\begin{align}\label{cwr1}
(e_i f_j - \pi^{|i||j|} f_j e_i) 1_\lambda &= \delta_{i,j} [\langle
h_i,\lambda\rangle]_{q_i,\pi_i} 1_\lambda,\!\!\!\!\!\!\!\!\!\!\!\\
\sum_{r=0}^{d_{ij}+1} (-1)^r \pi_i^{r |j| + r(r-1)/2}\sqbinom{d_{ij}+1}{r}_{q_i,\pi_i} e_i^{d_{ij}+1-r} e_j
e_i^r 1_\lambda &= 0 &(i \neq j),\label{cwr2}\\
\sum_{r=0}^{d_{ij}+1} (-1)^r\pi_i^{r|j|+r(r-1)/2}\sqbinom{d_{ij}+1}{r}_{q_i,\pi_i} f_i^{d_{ij}+1-r} f_j
f_i^r 1_\lambda &= 0&(i \neq j).\label{cwr3}
\end{align}
Also let $\dot U_{q,\pi}(\g)_{\LC}$ be the $\LC$-subalgebra of $\dot
U_{q,\pi}(\g)$ generated by the divided powers
\begin{equation}
e_i^{(n)} 1_\lambda := e_i^n 1_\lambda / [n]_{q_i,\pi_i}^!,
\qquad
f_i^{(n)} 1_\lambda := f_i^n 1_\lambda / [n]_{q_i,\pi_i}^!
\end{equation}
for all $i \in I, \lambda \in P$ and $n \geq 1$; see also \cite[Lemma 3.5]{Clark}.

We also need the antilinear
(with respect to the bar involution of the ground
ring)
algebra automorphisms
$\psi,\omega:\dot U_{q,\pi}(\g) 
\rightarrow\dot U_{q,\pi}(\g)$
and the linear algebra antiautomorphism
$\rho:\dot U_{q,\pi}(\g)\rightarrow 
\dot U_{q,\pi}(\g)$, which are defined on generators by
\begin{align}
\omega(1_\lambda) &= 1_{-\lambda},
&
\omega(e_i 1_\lambda) &= f_i 1_{-\lambda},
&
\omega(f_i 1_\lambda) &= e_i 1_{-\lambda},\label{omega}\\\label{psi}
\psi(1_\lambda)&=1_\lambda,
&\psi(e_i 1_\lambda)&=e_i 1_\lambda,
&\psi(f_i 1_\lambda)&=\pi^{|i,\lambda|} f_i 1_\lambda,\\
\label{rho}
 \rho(1_\lambda)&=1_\lambda,
&\rho(e_i 1_\lambda)&= q_i^{-\langle h_i,\lambda\rangle-1} 1_\lambda f_i,
&\rho(f_i 1_\lambda)&=q_i^{\langle h_i,\lambda\rangle-1} 1_\lambda e_i.
\end{align}
Note all of these are involutions.
Let $* := \rho \circ \psi$ and $! := \psi \circ \rho$. These are
mutually inverse antilinear
antiautomorphisms
with
\begin{align}\label{marry1}
 1_\lambda^*&=1_\lambda,
&(e_i 1_\lambda)^*&= q_i^{-\langle h_i,\lambda\rangle-1} 1_\lambda f_i,
&(f_i 1_\lambda)^*&=q_i^{\langle h_i,\lambda\rangle-1}
\pi^{|i,\lambda|}1_\lambda e_i,\\
1_\lambda^! &= 1_\lambda,
&(e_i 1_\lambda)^!&= \pi^{|i,\lambda|}q_i^{1+\langle h_i,\lambda\rangle} 1_\lambda f_i,
&(f_i 1_\lambda)^!&=q_i^{1-\langle h_i,\lambda\rangle}1_\lambda e_i.\label{marry2}
\end{align}
The notation here varies somewhat across the literature, e.g. the
counterparts of our $\omega, \psi$ and $\rho$ in the purely even setting
are denoted by $\omega\circ\psi, \psi$ and $\bar\rho$ in
\cite{KL3}.
In the remainder of the section, we are going to explain how to lift
$\omega, \psi$ and $\rho$ to
the Kac-Moody 2-supercategory.

First, we must explain how to deal with antilinearity at the level of 2-categories.
Let $\A$ be a graded supercategory. The supercategory $\A^{\operatorname{sop}}$
from Definition~\ref{curry} is actually a graded supercategory with
the same grading as $\A$,
i.e. $\deg(f^{\operatorname{sop}}) = \deg(f)$.
Similarly, if $\AA$ is a graded 2-supercategory
then $\AA^{\operatorname{sop}}$ is a graded 2-supercategory.
If $\AA$ is a graded $(Q,\Pi)$-2-supercategory in the sense of
\cite[Definition 6.5]{BE},
with
structure maps
$\sigma_\lambda:q_\lambda \stackrel{\sim}{\rightarrow}
1_\lambda,
\bar\sigma_\lambda:q_\lambda^{-1}\stackrel{\sim}{\rightarrow}
1_\lambda$ and $\zeta_\lambda:\pi_\lambda \stackrel{\sim}{\rightarrow}
1_\lambda$, we can regard
$\AA^{\operatorname{sop}}$ as a graded $(Q,\Pi)$-2-supercategory 
by declaring that its structure maps are
$(\bar\sigma_\lambda^{-1})^{\operatorname{sop}}:q_\lambda^{-1}
\stackrel{\sim}{\rightarrow} 1_\lambda$,
$(\sigma_\lambda^{-1})^{\operatorname{sop}}:q_\lambda \stackrel{\sim}{\rightarrow}
1_\lambda$ and $(\zeta_\lambda^{-1})^{\operatorname{sop}}:\pi_\lambda
\stackrel{\sim}{\rightarrow} 1_\lambda$. The key point here is that we
have interchanged the roles of $q$ and $q^{-1}$.

\begin{lemma}\label{classesstarttomorrow}
Suppose that $\AA$ and $\BB$ are graded 2-supercategories, and recall the
definition of their $(Q,\Pi)$-envelopes $\AA_{q,\pi}$ 
and $\BB_{q,\pi}$
from Definition~\ref{ernie}.
Given a graded 2-superfunctor
 $\phi:\AA \rightarrow 
(\BB_{q,\pi})^{\operatorname{sop}}$,
there is a canonical induced graded 2-superfunctor
$\tilde\phi:\AA_{q,\pi}\rightarrow
(\BB_{q,\pi})^{\operatorname{sop}}$.
\end{lemma}

\begin{proof}
View
$(\BB_{q,\pi})^{\operatorname{sop}}$
as a graded $(Q,\Pi)$-2-supercategory as explained above.
Then apply the universal property of $(Q,\Pi)$-envelopes from
\cite[Lemma~6.11(i)]{BE}.
\end{proof}

\begin{remark}
In the setup of Lemma~\ref{classesstarttomorrow}, 
the construction from the proof of \cite[Lemma 6.11(i)]{BE}
implies the following explicit description for
$\tilde \phi$.
It is equal to 
$\phi$ on objects. 
On a 1-morphism $F$ in $\AA$
with $\phi(F) = Q^{m'} \Pi^{a'} F'$ for a 1-morphism $F'$ in $\BB$, 
we have that
$\tilde\phi(Q^m \Pi^a F) = Q^{m'-m} \Pi^{a+a'} F'.$
Given another 1-morphism $G$ in $\AA$
with $\phi(G) = Q^{n'} \Pi^{b'} G'$
and $x \in \Hom_{\AA}(F,G)$ with
$\phi(x) = \left((x')_{n',b'}^{m',a'}\right)^{\operatorname{sop}}$
for $x' \in \Hom_{\BB}(G',F')$, we
have that
$$
\tilde\phi\left(x_{m,a}^{n,b}\right)= (-1)^{a|x|+b|x|+ab+b}
\left((x')_{n'-n,b+b'}^{m'-m,a+a'}\right)^{\operatorname{sop}}.
$$
Note also that $\tilde \phi$ is not strict (even if $\phi$ itself is
strict). 
Its coherence map $$
\tilde c_{Q^n \Pi^b G, Q^m \Pi^a F}:
\tilde \phi(Q^n \Pi^b G) \tilde\phi(Q^m \Pi^a F)
\stackrel{\sim}{\rightarrow} \tilde\phi(Q^{m+n} \Pi^{a+b} GF)
$$
is
$(-1)^{ab} \left(f^{m'+n'-m-n,a+b+a'+b'}_{k'-m-n,a+b+c'}\right)^{\operatorname{sop}}$,
where $\left(f^{m'+n',a'+b'}_{k',c'}\right)^{\operatorname{sop}}$
denotes
the coherence map 
$c_{G,F}:\phi(G) \phi(F) \stackrel{\sim}{\rightarrow} \phi(GF)$
for $\phi$, for 
$H'$ defined so that
$\phi(GF) = Q^{k'} \Pi^{c'} H'$ and
$f \in \Hom_{\BB}(H', G'F')$.
\end{remark}

Since we are assuming now that the parameters satisfy (\ref{hc}),
the Kac-Moody 2-supercategory
$\UU(\g)$ is a graded 2-supercategory with $\Z$-grading defined
as in the introduction. Let $\UU_{q,\pi}(\g)$ be its $(Q,\Pi)$-envelope from
Definition~\ref{ernie}.
We now proceed to define the categorical counterparts of the
antilinear automorphisms (\ref{omega})--(\ref{psi}).
Actually, the first was already defined in Proposition~\ref{opiso}, but we
need to extend this to the envelope.

\begin{proposition}\label{prop1}
There is an isomorphism of graded 2-supercategories
$$\tilde\omega:\UU_{q,\pi}(\g) \stackrel{\sim}{\rightarrow}
\UU_{q,\pi}(\g)^{\operatorname{sop}}
$$
defined on objects by $\lambda \mapsto -\lambda$ and 1-morphisms by
$Q^m \Pi^a E_i 1_\lambda \mapsto Q^{-m} \Pi^a F_i
1_{-\lambda}$,
$Q^m \Pi^a F_i 1_\lambda \mapsto Q^{-m} \Pi^a E_i 1_{-\lambda}$.
\end{proposition}

\begin{proof}
If we compose the strict 2-superfunctor from 
Proposition~\ref{opiso} with the canonical inclusion 
$\UU(\g)^{\operatorname{sop}} 
\rightarrow \UU_{q,\pi}(\g)^{\operatorname{sop}}$, we obtain
 a strict graded 2-superfunctor
$\omega:\UU(\g) \rightarrow
\UU_{q,\pi}(\g)^{\operatorname{sop}}$.
This is defined on objects by $\lambda \mapsto \lambda$, on
1-morphisms
 by
$E_i 1_\lambda \mapsto Q^0 \Pi^\0 F_i 1_{-\lambda},
F_i 1_\lambda \mapsto Q^0 \Pi^\0 E_i 1_{-\lambda}$,
and on 2-morphisms by the following:
\begin{align*}
\mathord{
\begin{tikzpicture}[baseline = 0]
	\draw[->,thick,darkred] (0.08,-.3) to (0.08,.4);
      \node at (0.08,0.05) {$\color{darkred}\bullet$};
   \node at (0.08,-.4) {$\scriptstyle{i}$};
\end{tikzpicture}
}
{\scriptstyle\lambda}
&\mapsto
\mathord{
\begin{tikzpicture}[baseline = 0]
	\draw[<-,thick,darkred] (0.08,-.3) to (0.08,.4);
      \node at (0.08,0.05) {$\color{darkred}\bullet$};
   \node at (.4,.05) {$\scriptstyle{-\lambda}$};
\draw[-,thin,red](.4,-.3) to (-.22,-.3);
\draw[-,thin,red](.4,.4) to (-.22,.4);
\node at (.55,-.3) {$\color{red}\scriptstyle \0$};
\node at (.55,.4) {$\color{red}\scriptstyle \0$};
\node at (-.37,.4) {$\color{red}\scriptstyle 0$};
\node at (-.37,-.3) {$\color{red}\scriptstyle 0$};
   \node at (0.08,.52) {$\scriptstyle{i}$};
\node at (.95,.45) {$\scriptstyle\operatorname{sop}$};
\end{tikzpicture}
},
&\mathord{
\begin{tikzpicture}[baseline = 0]
	\draw[->,thick,darkred] (0.28,-.3) to (-0.28,.4);
	\draw[->,thick,darkred] (-0.28,-.3) to (0.28,.4);
   \node at (-0.28,-.4) {$\scriptstyle{i}$};
   \node at (0.28,-.4) {$\scriptstyle{j}$};
   \node at (.4,.05) {$\scriptstyle{\lambda}$};
\end{tikzpicture}
}
&\mapsto
-(-1)^{|i||j|}\mathord{
\begin{tikzpicture}[baseline = 0]
	\draw[<-,thick,darkred] (0.28,-.3) to [out=90,in=-90] (-0.28,.4);
	\draw[<-,thick,darkred] (-0.28,-.3) to 
[out=90,in=-90] (0.28,.4);
   \node at (.42,.05) {$\scriptstyle{-\lambda}$};
\draw[-,thin,red](.4,-.3) to (-.4,-.3);
\draw[-,thin,red](.4,.4) to (-.4,.4);
\node at (.55,-.3) {$\color{red}\scriptstyle \0$};
\node at (.55,.4) {$\color{red}\scriptstyle \0$};
\node at (-.55,.4) {$\color{red}\scriptstyle 0$};
\node at (-.55,-.3) {$\color{red}\scriptstyle 0$};
   \node at (-0.28,.52) {$\scriptstyle{i}$};
   \node at (0.28,.52) {$\scriptstyle{j}$};
\node at (.95,.45) {$\scriptstyle\operatorname{sop}$};
\end{tikzpicture}
},\\
\mathord{
\begin{tikzpicture}[baseline = 0]
	\draw[<-,thick,darkred] (0.4,0.3) to[out=-90, in=0] (0.1,-0.1);
	\draw[-,thick,darkred] (0.1,-0.1) to[out = 180, in = -90] (-0.2,0.3);
    \node at (-0.2,.4) {$\scriptstyle{i}$};
  \node at (0.3,-0.15) {$\scriptstyle{\lambda}$};
\end{tikzpicture}
}
&\mapsto
\mathord{
\begin{tikzpicture}[baseline = 0]
	\draw[<-,thick,darkred] (0.4,-0.3) to[out=90, in=0] (0.1,0.1);
	\draw[-,thick,darkred] (0.1,0.1) to[out = 180, in = 90] (-0.2,-0.3);
 \node at (0.54,0.05) {$\scriptstyle{-\lambda}$};
\draw[-,thin,red](.55,-.3) to (-.4,-.3);
\draw[-,thin,red](.55,.3) to (-.4,.3);
\node at (.7,-.3) {$\color{red}\scriptstyle \0$};
\node at (.7,.3) {$\color{red}\scriptstyle \0$};
\node at (-.55,.3) {$\color{red}\scriptstyle 0$};
\node at (-.55,-.3) {$\color{red}\scriptstyle 0$};
    \node at (-0.2,-.44) {$\scriptstyle{i}$};
\node at (1.05,.42) {$\scriptstyle\operatorname{sop}$};
\end{tikzpicture}
},
&
\mathord{
\begin{tikzpicture}[baseline = 0]
	\draw[<-,thick,darkred] (0.4,-0.1) to[out=90, in=0] (0.1,0.3);
	\draw[-,thick,darkred] (0.1,0.3) to[out = 180, in = 90] (-0.2,-0.1);
    \node at (-0.2,-.2) {$\scriptstyle{i}$};
  \node at (0.3,0.4) {$\scriptstyle{\lambda}$};
\end{tikzpicture}
}&\mapsto
\mathord{
\begin{tikzpicture}[baseline = 0]
	\draw[<-,thick,darkred] (0.4,0.3) to[out=-90, in=0] (0.1,-0.1);
	\draw[-,thick,darkred] (0.1,-0.1) to[out = 180, in = -90] (-0.2,0.3);
 \node at (0.55,0) {$\scriptstyle{-\lambda}$};
\draw[-,thin,red](.55,-.3) to (-.4,-.3);
\draw[-,thin,red](.55,.3) to (-.4,.3);
\node at (.7,.3) {$\color{red}\scriptstyle \0$};
\node at (.7,-.3) {$\color{red}\scriptstyle \0$};
\node at (-.55,.3) {$\color{red}\scriptstyle 0$};
\node at (-.55,-.3) {$\color{red}\scriptstyle 0$};
    \node at (-0.2,.44) {$\scriptstyle{i}$};
\node at (1.05,.45) {$\scriptstyle\operatorname{sop}$};
\end{tikzpicture}
}.
\end{align*}
It remains to apply Lemma~\ref{classesstarttomorrow} to get the desired graded
2-superfunctor
$\tilde\omega$ (which is no longer strict).
\end{proof}

\begin{proposition}\label{prop2}
Assume that there is a given element $\sqrt{-1} \in \k_{\0}$ which
squares to $-1$. Then there is an isomorphism of graded
2-supercategories $$\tilde\psi:\UU_{q,\pi}(\g) \stackrel{\sim}{\rightarrow}
\UU_{q,\pi}(\g)^{\operatorname{sop}}
$$
defined on objects by $\lambda \mapsto \lambda$ and 1-morphisms by
$Q^m \Pi^a E_i 1_\lambda \mapsto Q^{-m} \Pi^a E_i
1_{\lambda}$,
$Q^m \Pi^a F_i 1_\lambda \mapsto Q^{-m} \Pi^{a+|i,\lambda|} F_i 1_{\lambda}$.
\end{proposition}

\begin{proof}
We claim that there is a strict graded 2-superfunctor
$\psi:\UU(\g) \rightarrow \UU_{q,\pi}(\g)^{\operatorname{sop}}$ which is
defined on objects by $\lambda \mapsto \lambda$, 
1-morphisms by $E_i 1_\lambda \mapsto Q^0 \Pi^\0 E_i 1_\lambda$,
$F_i 1_\lambda \mapsto Q^0 \Pi^{|i,\lambda|} F_i 1_\lambda$, and
2-morphisms by the following:
\begin{align*}
\mathord{
\begin{tikzpicture}[baseline = 0]
	\draw[->,thick,darkred] (0.08,-.3) to (0.08,.4);
      \node at (0.08,0.05) {$\color{darkred}\bullet$};
   \node at (0.08,-.4) {$\scriptstyle{i}$};
\end{tikzpicture}
}
{\scriptstyle\lambda}
&\mapsto
\left\{
\begin{array}{rl}
\mathord{
\begin{tikzpicture}[baseline = 0]
	\draw[->,thick,darkred] (0.08,-.3) to (0.08,.4);
      \node at (0.08,0.05) {$\color{darkred}\bullet$};
   \node at (.35,.05) {$\scriptstyle{\lambda}$};
\draw[-,thin,red](.4,-.3) to (-.22,-.3);
\draw[-,thin,red](.4,.4) to (-.22,.4);
\node at (.55,-.3) {$\color{red}\scriptstyle \0$};
\node at (.55,.4) {$\color{red}\scriptstyle \0$};
\node at (-.37,.4) {$\color{red}\scriptstyle 0$};
\node at (-.37,-.3) {$\color{red}\scriptstyle 0$};
   \node at (0.08,-.42) {$\scriptstyle{i}$};
\node at (.95,.45) {$\scriptstyle\operatorname{sop}$};
\end{tikzpicture}
}&\text{if $|i|=\0$,}\\
\sqrt{-1}\mathord{
\begin{tikzpicture}[baseline = 0]
	\draw[->,thick,darkred] (0.08,-.3) to (0.08,.4);
      \node at (0.08,0.05) {$\color{darkred}\bullet$};
   \node at (.35,.05) {$\scriptstyle{\lambda}$};
\draw[-,thin,red](.4,-.3) to (-.22,-.3);
\draw[-,thin,red](.4,.4) to (-.22,.4);
\node at (.55,-.3) {$\color{red}\scriptstyle \0$};
\node at (.55,.4) {$\color{red}\scriptstyle \0$};
\node at (-.37,.4) {$\color{red}\scriptstyle 0$};
\node at (-.37,-.3) {$\color{red}\scriptstyle 0$};
   \node at (0.08,-.42) {$\scriptstyle{i}$};
\node at (.95,.45) {$\scriptstyle\operatorname{sop}$};
\end{tikzpicture}
}&\text{if $|i|=\1$},
\end{array}\right.\\
\mathord{
\begin{tikzpicture}[baseline = 0]
	\draw[->,thick,darkred] (0.28,-.3) to (-0.28,.4);
	\draw[->,thick,darkred] (-0.28,-.3) to (0.28,.4);
   \node at (-0.28,-.4) {$\scriptstyle{i}$};
   \node at (0.28,-.4) {$\scriptstyle{j}$};
   \node at (.4,.05) {$\scriptstyle{\lambda}$};
\end{tikzpicture}
}
&\mapsto
\left\{
\begin{array}{rl}
\mathord{
\begin{tikzpicture}[baseline = 0]
	\draw[->,thick,darkred] (0.28,-.3) to [out=90,in=-90] (-0.28,.4);
	\draw[->,thick,darkred] (-0.28,-.3) to 
[out=90,in=-90] (0.28,.4);
   \node at (.4,.05) {$\scriptstyle{\lambda}$};
\draw[-,thin,red](.4,-.3) to (-.4,-.3);
\draw[-,thin,red](.4,.4) to (-.4,.4);
\node at (.55,-.3) {$\color{red}\scriptstyle \0$};
\node at (.55,.4) {$\color{red}\scriptstyle \0$};
\node at (-.55,.4) {$\color{red}\scriptstyle 0$};
\node at (-.55,-.3) {$\color{red}\scriptstyle 0$};
   \node at (-0.28,-.42) {$\scriptstyle{j}$};
   \node at (0.28,-.42) {$\scriptstyle{i}$};
\node at (.95,.45) {$\scriptstyle\operatorname{sop}$};
\end{tikzpicture}
}&\text{if $|i||j|=\0$},\\
\sqrt{-1}\mathord{
\begin{tikzpicture}[baseline = 0]
	\draw[->,thick,darkred] (0.28,-.3) to [out=90,in=-90] (-0.28,.4);
	\draw[->,thick,darkred] (-0.28,-.3) to 
[out=90,in=-90] (0.28,.4);
   \node at (.4,.05) {$\scriptstyle{\lambda}$};
\draw[-,thin,red](.4,-.3) to (-.4,-.3);
\draw[-,thin,red](.4,.4) to (-.4,.4);
\node at (.55,-.3) {$\color{red}\scriptstyle \0$};
\node at (.55,.4) {$\color{red}\scriptstyle \0$};
\node at (-.55,.4) {$\color{red}\scriptstyle 0$};
\node at (-.55,-.3) {$\color{red}\scriptstyle 0$};
   \node at (-0.28,-.42) {$\scriptstyle{j}$};
   \node at (0.28,-.42) {$\scriptstyle{i}$};
\node at (.95,.45) {$\scriptstyle\operatorname{sop}$};
\end{tikzpicture}
}&\text{if $|i||j|=\1$},
\end{array}\right.
\\
\mathord{
\begin{tikzpicture}[baseline = 0]
	\draw[<-,thick,darkred] (0.4,0.3) to[out=-90, in=0] (0.1,-0.1);
	\draw[-,thick,darkred] (0.1,-0.1) to[out = 180, in = -90] (-0.2,0.3);
    \node at (-0.2,.4) {$\scriptstyle{i}$};
  \node at (0.3,-0.15) {$\scriptstyle{\lambda}$};
\end{tikzpicture}
}
&\mapsto
\mathord{
\begin{tikzpicture}[baseline = 0]
	\draw[-,thick,darkred] (0.4,-0.3) to[out=90, in=0] (0.1,0.1);
	\draw[->,thick,darkred] (0.1,0.1) to[out = 180, in = 90] (-0.2,-0.3);
 \node at (0.5,0.05) {$\scriptstyle{\lambda}$};
\draw[-,thin,red](.55,-.3) to (-.4,-.3);
\draw[-,thin,red](.55,.3) to (-.4,.3);
\node at (.9,-.3) {$\color{red}\scriptstyle |i,\lambda|$};
\node at (.7,.3) {$\color{red}\scriptstyle \0$};
\node at (-.55,.3) {$\color{red}\scriptstyle 0$};
\node at (-.55,-.3) {$\color{red}\scriptstyle 0$};
    \node at (0.4,-.44) {$\scriptstyle{i}$};
\node at (1.05,.42) {$\scriptstyle\operatorname{sop}$};
\end{tikzpicture}
},\\
\mathord{
\begin{tikzpicture}[baseline = 0]
	\draw[<-,thick,darkred] (0.4,-0.1) to[out=90, in=0] (0.1,0.3);
	\draw[-,thick,darkred] (0.1,0.3) to[out = 180, in = 90] (-0.2,-0.1);
    \node at (-0.2,-.2) {$\scriptstyle{i}$};
  \node at (0.3,0.4) {$\scriptstyle{\lambda}$};
\end{tikzpicture}
}&\mapsto
(-1)^{|i,\lambda|}\mathord{
\begin{tikzpicture}[baseline = 0]
	\draw[-,thick,darkred] (0.4,0.3) to[out=-90, in=0] (0.1,-0.1);
	\draw[->,thick,darkred] (0.1,-0.1) to[out = 180, in = -90] (-0.2,0.3);
 \node at (0.5,0) {$\scriptstyle{\lambda}$};
\draw[-,thin,red](.55,-.3) to (-.4,-.3);
\draw[-,thin,red](.55,.3) to (-.4,.3);
\node at (.9,.3) {$\color{red}\scriptstyle |i,\lambda|$};
\node at (.7,-.3) {$\color{red}\scriptstyle \0$};
\node at (-.55,.3) {$\color{red}\scriptstyle 0$};
\node at (-.55,-.3) {$\color{red}\scriptstyle 0$};
    \node at (0.4,.44) {$\scriptstyle{i}$};
\node at (1.4,.45) {$\scriptstyle\operatorname{sop}$};
\end{tikzpicture}
}.
\end{align*}
To prove the claim, one needs to verify the relations. 
Note to start with that
$$
\mathord{
\begin{tikzpicture}[baseline = 0]
	\draw[->,thick,darkred] (0.08,-.3) to (0.08,.4);
      \node at (0.08,0.05) {$\color{darkred}\bullet$};
   \node at (0.08,-.4) {$\scriptstyle{i}$};
   \node at (-0.15,.05) {$\color{darkred}\scriptstyle{n}$};
\end{tikzpicture}
}
{\scriptstyle\lambda}
\mapsto
\left\{
\begin{array}{rl}
\mathord{
\begin{tikzpicture}[baseline = 0]
	\draw[->,thick,darkred] (0.08,-.3) to (0.08,.4);
      \node at (0.08,0.05) {$\color{darkred}\bullet$};
   \node at (-0.15,.05) {$\color{darkred}\scriptstyle{n}$};
   \node at (.35,.05) {$\scriptstyle{\lambda}$};
\draw[-,thin,red](.4,-.3) to (-.22,-.3);
\draw[-,thin,red](.4,.4) to (-.22,.4);
\node at (.55,-.3) {$\color{red}\scriptstyle \0$};
\node at (.55,.4) {$\color{red}\scriptstyle \0$};
\node at (-.37,.4) {$\color{red}\scriptstyle 0$};
\node at (-.37,-.3) {$\color{red}\scriptstyle 0$};
   \node at (0.08,-.42) {$\scriptstyle{i}$};
\node at (.95,.45) {$\scriptstyle\operatorname{sop}$};
\end{tikzpicture}
}&\text{if $n|i|=\0$,}\\
\sqrt{-1}\mathord{
\begin{tikzpicture}[baseline = 0]
	\draw[->,thick,darkred] (0.08,-.3) to (0.08,.4);
      \node at (0.08,0.05) {$\color{darkred}\bullet$};
   \node at (-0.15,.05) {$\color{darkred}\scriptstyle{n}$};
   \node at (.35,.05) {$\scriptstyle{\lambda}$};
\draw[-,thin,red](.4,-.3) to (-.22,-.3);
\draw[-,thin,red](.4,.4) to (-.22,.4);
\node at (.55,-.3) {$\color{red}\scriptstyle \0$};
\node at (.55,.4) {$\color{red}\scriptstyle \0$};
\node at (-.37,.4) {$\color{red}\scriptstyle 0$};
\node at (-.37,-.3) {$\color{red}\scriptstyle 0$};
   \node at (0.08,-.42) {$\scriptstyle{i}$};
\node at (.95,.45) {$\scriptstyle\operatorname{sop}$};
\end{tikzpicture}
}&\text{if $n|i|=\1$}.
\end{array}\right.$$
Using this, the quiver Hecke superalgebra relations (\ref{now})--(\ref{qhalast})
are straightforward. 
The inversion relations (\ref{inv1})--(\ref{inv3}) are also
fine. 
The adjunction relations (\ref{rightadj}) need a little more
care since the signs coming from (\ref{web}) play a role.
Then apply Lemma~\ref{classesstarttomorrow} to get the desired graded
2-superfunctor
$\tilde\psi$ (which is no longer strict).
\end{proof}

\begin{definition}
Let $\AA$ be a graded 2-supercategory.
Define $\AA^{\operatorname{srev}}$ to be the graded
2-supercategory
with the same objects as $\AA$, and morphism categories
$$
\Hom_{\AA^{\operatorname{srev}}}(\mu,\lambda) :=
\Hom_{\AA}(\lambda,\mu).
$$
We write
$F^{\operatorname{srev}}:\mu\rightarrow\lambda$ (resp.\
  $x^{\operatorname{srev}}:F^{\operatorname{srev}}\rightarrow
  G^{\operatorname{srev}}$) 
for the
1-morphism (resp. 2-morphism) 
in $\AA^{\operatorname{srev}}$ defined by the
1-morphism $F:\lambda\rightarrow\mu$ (resp. $x:F \rightarrow G$) in $\AA$.
Then, horizontal composition in $\AA^{\operatorname{srev}}$
is defined on 1-morphisms by $(F^{\operatorname{srev}})(G^{\operatorname{srev}})
:= (GF)^{\operatorname{srev}}$
and on homogeneous 2-morphisms by
$(x^{\operatorname{srev}})(y^{\operatorname{srev}}) :=
(-1)^{|x||y|}(yx)^{\operatorname{srev}}$.
Vertical composition of 2-morphisms in $\AA^{\operatorname{srev}}$ is the same as
in $\AA$.
Here is the check of the super interchange law in
$\AA^{\operatorname{srev}}$:
\begin{align*}
(x^{\operatorname{srev}} y^{\operatorname{srev}}) \circ
(u^{\operatorname{srev}}v^{\operatorname{srev}})
&=
(-1)^{|x||y|+|u||v|}
 (yx)^{\operatorname{srev}} \circ (vu)^{\operatorname{srev}}\\
&=
(-1)^{|x||y|+|u||v|}
 ((yx) \circ (vu))^{\operatorname{srev}}\\
&=
(-1)^{|x||y|+|u||v|+|x||v|}
((y \circ v) (x \circ u))^{\operatorname{srev}}\\
&= 
(-1)^{|y||u|} (x \circ u)^{\operatorname{srev}} (y \circ
v)^{\operatorname{srev}}\\
&=
(-1)^{|y||u|} (x^{\operatorname{srev}} \circ u^{\operatorname{srev}}) (y^{\operatorname{srev}} \circ
v^{\operatorname{srev}}).
\end{align*}
\end{definition}
 
If $\AA$ is a graded $(Q,\Pi)$-2-supercategory
with
structure maps
$\sigma_\lambda:q_\lambda \stackrel{\sim}{\rightarrow}
1_\lambda,
\bar\sigma_\lambda:q_\lambda^{-1}\stackrel{\sim}{\rightarrow}
1_\lambda$ and $\zeta_\lambda:\pi_\lambda \stackrel{\sim}{\rightarrow}
1_\lambda$, we can regard
$\AA^{\operatorname{srev}}$ as a graded $(Q,\Pi)$-2-supercategory 
by declaring that its structure maps are
$(\sigma_\lambda)^{\operatorname{srev}}:(q_\lambda)^{\operatorname{srev}}
\stackrel{\sim}{\rightarrow} (1_\lambda)^{\operatorname{srev}}$,
$(\bar\sigma_\lambda)^{\operatorname{srev}}:(q_\lambda^{-1})^{\operatorname{srev}} \stackrel{\sim}{\rightarrow}
(1_\lambda)^{\operatorname{srev}}$ and $(\zeta_\lambda)^{\operatorname{srev}}:(\pi_\lambda)^{\operatorname{srev}}
\stackrel{\sim}{\rightarrow} (1_\lambda)^{\operatorname{srev}}$.

\begin{lemma}\label{classesshouldhavestarted}
Suppose that $\AA$ and $\BB$ are graded 2-supercategories.
Given a graded 2-superfunctor
 $\phi:\AA \rightarrow 
(\BB_{q,\pi})^{\operatorname{srev}}$,
there is a canonical induced graded 2-superfunctor
$\tilde\phi:\AA_{q,\pi}\rightarrow
(\BB_{q,\pi})^{\operatorname{srev}}$.
\end{lemma}

\begin{proof}
View
$(\BB_{q,\pi})^{\operatorname{srev}}$
as a graded $(Q,\Pi)$-2-supercategory as explained above.
Then apply  
\cite[Lemma~6.11(i)]{BE}.
\end{proof}

\begin{remark}
In the setup of Lemma~\ref{classesshouldhavestarted},
the construction from the proof of \cite[Lemma 6.11(i)]{BE}
implies the following explicit description for
$\tilde \phi$.
It is equal to 
$\phi$ on objects. 
On a 1-morphism $F$ in $\AA$
with $\phi(F) = (Q^{m'} \Pi^{a'} F')^{\operatorname{srev}}$ for a 1-morphism $F'$ in $\BB$, 
we have that
$
\tilde\phi(Q^m \Pi^a F) = (Q^{m+m'} \Pi^{a+a'} F')^{\operatorname{srev}}.$
Given another 1-morphism $G$ in $\AA$
with $\phi(G) = (Q^{n'} \Pi^{b'} G')^{\operatorname{srev}}$
and $x \in \Hom_{\AA}(F,G)$ with
$\phi(x) = \left((x')_{m',a'}^{n',b'}\right)^{\operatorname{srev}}$
for $x' \in \Hom_{\BB}(F',G')$, we
have that
$$
\tilde\phi\left(x_{m,a}^{n,b}\right)= (-1)^{aa'+bb'}
\left((x')_{m+m',a+a'}^{n+n',b+b'}\right)^{\operatorname{srev}}.
$$
The coherence map $$
\tilde c_{Q^n \Pi^b G, Q^m \Pi^a F}:
\tilde \phi(Q^n \Pi^b G) \tilde\phi(Q^m \Pi^a F)
\stackrel{\sim}{\rightarrow} \tilde\phi(Q^{m+n} \Pi^{a+b} GF)
$$
is
$(-1)^{a(b+b')+(a+b)(a'+b'+c')}\left(f_{m+n+m'+n',a+b+a'+b'}^{m+n+k',a+b+c'}\right)^{\operatorname{srev}}$,
where $\left(f_{m'+n',a'+b'}^{k',c'}\right)^{\operatorname{srev}}$
denotes
the coherence map 
$c_{G,F}:\phi(G) \phi(F) \stackrel{\sim}{\rightarrow} \phi(GF)$
for $\phi$, for 
$H'$ defined so that
$\phi(GF) = Q^{k'} \Pi^{c'} H'$ and
$f \in \Hom_{\BB}(F'G', H')$.
\end{remark}

\begin{proposition}\label{prop3}
Assume that there is a given element $\sqrt{-1} \in \k_{\0}$ which
squares to $-1$. Then there is an isomorphism of graded
2-supercategories
$$
\tilde\rho:\UU_{q,\pi}(\g) \stackrel{\sim}{\rightarrow} \UU_{q,\pi}(\g)^{\operatorname{srev}}
$$
such that $\lambda \mapsto \lambda$ and
$Q^m \Pi^a E_i 1_\lambda \mapsto (Q^{m-d_i(1+\langle h_i,\lambda\rangle)}
\Pi^a 1_\lambda F_i)^{\operatorname{srev}},$
$Q^m \Pi^a F_i 1_\lambda \mapsto (Q^{m-d_i(1-\langle h_i,\lambda\rangle)}
\Pi^a 1_\lambda E_i)^{\operatorname{srev}}.$
\end{proposition}

\begin{proof}
We claim that there is a strict graded 2-superfunctor
$\rho:\UU(\g) \rightarrow \UU_{q,\pi}(\g)^{\operatorname{srev}}$ 
defined on objects by $\lambda \mapsto \lambda$, 
1-morphisms by $E_i 1_\lambda \mapsto (Q^{-d_i(1+\langle
  h_i,\lambda\rangle)} \Pi^\0 1_\lambda F_i)^{\operatorname{srev}}$,
$F_i 1_\lambda \mapsto (Q^{-d_i(1-\langle h_i,\lambda\rangle)}
\Pi^{|i,\lambda|} 1_\lambda E_i)^{\operatorname{srev}}$, and 
2-morphisms by the following:
\begin{align*}
\mathord{
\begin{tikzpicture}[baseline = 0]
	\draw[->,thick,darkred] (0.08,-.3) to (0.08,.4);
      \node at (0.08,0.05) {$\color{darkred}\bullet$};
   \node at (0.08,-.4) {$\scriptstyle{i}$};
\end{tikzpicture}
}
{\scriptstyle\lambda}
&\mapsto
\left\{
\begin{array}{rl}
\mathord{
\begin{tikzpicture}[baseline = 0]
	\draw[<-,thick,darkred] (0.08,-.3) to (0.08,.4);
      \node at (0.08,0.05) {$\color{darkred}\bullet$};
   \node at (-.35,.05) {$\scriptstyle{\lambda}$};
\draw[-,thin,red](.4,-.3) to (-.22,-.3);
\draw[-,thin,red](.4,.4) to (-.22,.4);
\node at (.55,-.3) {$\color{red}\scriptstyle \0$};
\node at (.55,.4) {$\color{red}\scriptstyle \0$};
\node at (-1.2,.4) {$\color{red}\scriptstyle -d_i(1+\langle h_i,\lambda\rangle)$};
\node at (-1.2,-.3) {$\color{red}\scriptstyle -d_i(1+\langle h_i,\lambda\rangle)$};
   \node at (0.08,.52) {$\scriptstyle{i}$};
\node at (.95,.45) {$\scriptstyle\operatorname{srev}$};
\end{tikzpicture}
}&\text{if $|i|=\0$,}\\
\sqrt{-1}\mathord{
\begin{tikzpicture}[baseline = 0]
	\draw[<-,thick,darkred] (0.08,-.3) to (0.08,.4);
      \node at (0.08,0.05) {$\color{darkred}\bullet$};
   \node at (-.35,.05) {$\scriptstyle{\lambda}$};
\draw[-,thin,red](.4,-.3) to (-.22,-.3);
\draw[-,thin,red](.4,.4) to (-.22,.4);
\node at (.55,-.3) {$\color{red}\scriptstyle \0$};
\node at (.55,.4) {$\color{red}\scriptstyle \0$};
\node at (-1.2,.4) {$\color{red}\scriptstyle -d_i(1+\langle h_i,\lambda\rangle)$};
\node at (-1.2,-.3) {$\color{red}\scriptstyle -d_i(1+\langle h_i,\lambda\rangle)$};
   \node at (0.08,.52) {$\scriptstyle{i}$};
\node at (.95,.45) {$\scriptstyle\operatorname{srev}$};
\end{tikzpicture}
}&\text{if $|i|=\1$},
\end{array}\right.\\
\mathord{
\begin{tikzpicture}[baseline = 0]
	\draw[->,thick,darkred] (0.28,-.3) to (-0.28,.4);
	\draw[->,thick,darkred] (-0.28,-.3) to (0.28,.4);
   \node at (-0.28,-.4) {$\scriptstyle{i}$};
   \node at (0.28,-.4) {$\scriptstyle{j}$};
   \node at (.4,.05) {$\scriptstyle{\lambda}$};
\end{tikzpicture}
}
&\mapsto
\left\{
\begin{array}{rl}
\mathord{
\begin{tikzpicture}[baseline = 0]
	\draw[<-,thick,darkred] (0.28,-.3) to [out=90,in=-90] (-0.28,.4);
	\draw[<-,thick,darkred] (-0.28,-.3) to 
[out=90,in=-90] (0.28,.4);
   \node at (-.4,.05) {$\scriptstyle{\lambda}$};
\draw[-,thin,red](.4,-.3) to (-.4,-.3);
\draw[-,thin,red](.4,.4) to (-.4,.4);
\node at (.55,-.3) {$\color{red}\scriptstyle \0$};
\node at (.55,.4) {$\color{red}\scriptstyle \0$};
\node at (-2.8,.4) {$\color{red}\scriptstyle -d_i(1+\langle
  h_i,\lambda\rangle)-d_j(1+\langle h_j,\lambda\rangle)-(\alpha_i,\alpha_j)$};
\node at (-2.8,-.3) {$\color{red}\scriptstyle -d_i(1+\langle
  h_i,\lambda\rangle)-d_j(1+\langle h_j,\lambda\rangle)-(\alpha_i,\alpha_j)$};
   \node at (-0.28,.52) {$\scriptstyle{i}$};
   \node at (0.28,.52) {$\scriptstyle{j}$};
\node at (.95,.45) {$\scriptstyle\operatorname{srev}$};
\end{tikzpicture}
}&\text{if $|i||j|=\0$},\\
-\sqrt{-1}
\mathord{
\begin{tikzpicture}[baseline = 0]
	\draw[<-,thick,darkred] (0.28,-.3) to [out=90,in=-90] (-0.28,.4);
	\draw[<-,thick,darkred] (-0.28,-.3) to 
[out=90,in=-90] (0.28,.4);
   \node at (-.4,.05) {$\scriptstyle{\lambda}$};
\draw[-,thin,red](.4,-.3) to (-.4,-.3);
\draw[-,thin,red](.4,.4) to (-.4,.4);
\node at (.55,-.3) {$\color{red}\scriptstyle \0$};
\node at (.55,.4) {$\color{red}\scriptstyle \0$};
\node at (-2.8,.4) {$\color{red}\scriptstyle -d_i(1+\langle
  h_i,\lambda\rangle)-d_j(1+\langle h_j,\lambda\rangle)-(\alpha_i,\alpha_j)$};
\node at (-2.8,-.3) {$\color{red}\scriptstyle -d_i(1+\langle
  h_i,\lambda\rangle)-d_j(1+\langle h_j,\lambda\rangle)-(\alpha_i,\alpha_j)$};
   \node at (-0.28,.52) {$\scriptstyle{i}$};
   \node at (0.28,.52) {$\scriptstyle{j}$};
\node at (.95,.45) {$\scriptstyle\operatorname{srev}$};
\end{tikzpicture}
}&\text{if $|i||j|=\1$},
\end{array}\right.
\\
\mathord{
\begin{tikzpicture}[baseline = 0]
	\draw[<-,thick,darkred] (0.4,0.3) to[out=-90, in=0] (0.1,-0.1);
	\draw[-,thick,darkred] (0.1,-0.1) to[out = 180, in = -90] (-0.2,0.3);
    \node at (-0.2,.4) {$\scriptstyle{i}$};
  \node at (0.3,-0.15) {$\scriptstyle{\lambda}$};
\end{tikzpicture}
}
&\mapsto
\mathord{
\begin{tikzpicture}[baseline = 0]
	\draw[<-,thick,darkred] (0.4,0.3) to[out=-90, in=0] (0.1,-0.1);
	\draw[-,thick,darkred] (0.1,-0.1) to[out = 180, in = -90] (-0.2,0.3);
 \node at (0.5,0) {$\scriptstyle{\lambda}$};
\draw[-,thin,red](.55,-.3) to (-.4,-.3);
\draw[-,thin,red](.55,.3) to (-.4,.3);
\node at (.7,.3) {$\color{red}\scriptstyle \0$};
\node at (.7,-.3) {$\color{red}\scriptstyle \0$};
\node at (-.55,.3) {$\color{red}\scriptstyle 0$};
\node at (-.55,-.3) {$\color{red}\scriptstyle 0$};
    \node at (-0.3,.44) {$\scriptstyle{i}$};
\node at (1.05,.45) {$\scriptstyle\operatorname{srev}$};
\end{tikzpicture}
},\qquad
\mathord{
\begin{tikzpicture}[baseline = 0]
	\draw[<-,thick,darkred] (0.4,-0.1) to[out=90, in=0] (0.1,0.3);
	\draw[-,thick,darkred] (0.1,0.3) to[out = 180, in = 90] (-0.2,-0.1);
    \node at (-0.2,-.2) {$\scriptstyle{i}$};
  \node at (0.3,0.4) {$\scriptstyle{\lambda}$};
\end{tikzpicture}
}\mapsto
\mathord{
\begin{tikzpicture}[baseline = 0]
	\draw[<-,thick,darkred] (0.4,-0.3) to[out=90, in=0] (0.1,0.1);
	\draw[-,thick,darkred] (0.1,0.1) to[out = 180, in = 90] (-0.2,-0.3);
 \node at (0.5,0.05) {$\scriptstyle{\lambda}$};
\draw[-,thin,red](.55,-.3) to (-.4,-.3);
\draw[-,thin,red](.55,.3) to (-.4,.3);
\node at (.7,-.3) {$\color{red}\scriptstyle \0$};
\node at (.7,.3) {$\color{red}\scriptstyle \0$};
\node at (-.55,.3) {$\color{red}\scriptstyle 0$};
\node at (-.55,-.3) {$\color{red}\scriptstyle 0$};
    \node at (-0.3,-.44) {$\scriptstyle{i}$};
\node at (1.05,.42) {$\scriptstyle\operatorname{srev}$};
\end{tikzpicture}
}.
\end{align*}
To prove the claim, one needs to verify the relations. The quiver
Hecke relations are the most complicated; for these, use (\ref{downcross1})--(\ref{downcross2}),
(\ref{noisy}) and (\ref{lastone}). (Note the degree shifts actually
play no role in this argument; they are included to match (\ref{rho}).)
Finally, apply Lemma~\ref{classesshouldhavestarted} to get $\tilde\rho$.
\end{proof}

Suppose in this paragraph that $\k = \k_\0$ is a field. Then 
the underlying 2-category $\underline{\UU}_{q,\pi}(\g)$ is a
$(Q,\Pi)$-2-category in the sense of \cite[Definition 6.14]{BE}, as is its
additive Karoubi envelope 
$\dot{\underline{\UU}}_{q,\pi}(\g)$. 
The Grothendieck ring $K_0(\dot{\underline{\UU}}_{q,\pi}(\g))$ 
is a locally unital
$\mathcal{L}$-algebra with distinguished idempotents
$\{1_\lambda\:|\:\lambda \in P\}$.
The analogous Grothendieck ring arising from
$\UU_{q,\pi}(\g)^{\operatorname{sop}}$ may be identified with
$K_0(\dot{\underline{\UU}}_{q,\pi}(\g))$ as a ring, but now $q$ acts
as $q^{-1}$. This means that the isomorphisms $\tilde\omega$ and
$\tilde\psi$
from Propositions~\ref{prop1}--\ref{prop2} induce {\em antilinear}
locally unital algebra automorphisms
$$
[\tilde\omega], [\tilde\psi]:
K_0(\dot{\underline{\UU}}_{q,\pi}(\g))\rightarrow
K_0(\dot{\underline{\UU}}_{q,\pi}(\g)).
$$
Also, the Grothendieck ring arising from
$\UU_{q,\pi}(\g)^{\operatorname{srev}}$ may be identified with the
opposite 
$K_0(\dot{\underline{\UU}}_{q,\pi}(\g))^{\operatorname{op}}$,
so that the isomorphism $\tilde\rho$ from Proposition~\ref{prop3}
induces a linear algebra antiautomorphism
$$
[\tilde\rho]:
K_0(\dot{\underline{\UU}}_{q,\pi}(\g))\rightarrow
K_0(\dot{\underline{\UU}}_{q,\pi}(\g)).
$$
The epimorphism $\gamma:
\dot U_{q,\pi}(\g)_{\LC} \twoheadrightarrow
K_0(\underline{\dot\UU}_{q,\pi}(\g))$ to be constructed in Theorem~\ref{bits}
below
intertwines 
the maps $\omega, \psi$ 
and $\rho$ from (\ref{omega})--(\ref{rho}) with
$[\tilde\omega], [\tilde\psi]$ and $[\tilde\rho]$.

\begin{remark}
One can also consider the compositions 
$\tilde\rho\circ \tilde \psi$ and $\tilde\psi \circ \tilde \rho$.
Both of these maps can be defined directly on generators,
revealing that they actually do not require the existence of
$\sqrt{-1} \in \k$, unlike $\tilde\rho$ and $\tilde\psi$ themselves.
Just as discussed in \cite[(3.46)--(3.47]{KL3}, these maps may also 
be interpreted as taking right
duals/mates and left duals/mates, respectively.
They decategorify to
the maps $*$ and $!$ from (\ref{marry1})--(\ref{marry2}).
\end{remark}

\section{The sesquilinear form}\label{sform}

Continue with the assumptions from section \ref{sqg}.
Let $\mathbf{f}$ be the $\mathbb{L}$-superalgebra on generators
$\{\theta_i\:|\:i \in I\}$ with $|\theta_i| := |i|$, 
subject to relations
\begin{equation}\label{serres}
\sum_{r=0}^{d_{ij}+1} (-1)^r \pi_i^{r|j|+r(r-1)/2}
\sqbinom{d_{ij}+1}{r}_{q_i,\pi_i} \theta_i^{d_{ij}+1-r} \theta_j \theta_i^r =
0
\end{equation}
for all $i \neq j$.
There is a $Q$-grading 
$\mathbf{f} = \bigoplus_{\alpha \in Q} \mathbf{f}_\alpha$
on $\mathbf{f}$ compatible with the $\Z/2$-grading
defined by declaring that each 
$\theta_i$ is of degree $\alpha_i$.
Viewing $\mathbf{f} \otimes \mathbf{f}$ as an algebra with the twisted
multiplication
$(x \otimes y) (x' \otimes y') := \pi^{|y||x'|} q^{-(\beta,\alpha')}
xx' \otimes yy'$
for homogeneous $x \in \mathbf{f}_\alpha,y\in \mathbf{f}_\beta,x' \in
\mathbf{f}_{\alpha'},y' \in \mathbf{f}_{\beta'}$, we let
$r:\mathbf{f} \rightarrow \mathbf{f} \otimes \mathbf{f}$ be the 
superalgebra homomorphism defined from 
$r(\theta_i) = \theta_i \otimes 1 + 1 \otimes \theta_i$ for each $i
\in I$.
By \cite[Proposition 3.4.1]{CHW1}, there is a (non-degenerate) symmetric
bilinear form $(-,-)$ on $\mathbf{f}$ defined by the following properties:
\begin{itemize}
\item
$(\theta_i,\theta_j) = \delta_{i,j} / (1-\pi_i q_i^2)$;
\item
$(xy,z) = (x \otimes y, r(z))$;
\item
$(x,yz)= (r(x), y \otimes z)$.
\end{itemize}
Here, the form on $\mathbf{f}\otimes\mathbf{f}$ is defined from 
$(x \otimes y, x' \otimes y') := (x,x') (y,y')$.
Note that $\mathbf{f}_\alpha$ and $\mathbf{f}_\beta$ are orthogonal
for $\alpha \neq \beta$.

\begin{theorem}[Lusztig, Clark]\label{clark}
There is a unique sesquilinear form
($=$ antilinear in the first argument, linear in the second)
$\langle-,-\rangle:\dot U_{q,\pi}(\g)\times \dot
U_{q,\pi}(\g)\rightarrow \mathbb{L}$ 
such that the following hold:
\begin{enumerate}
\item
$\langle 1_\mu x 1_\lambda, 1_{\mu'} x' 1_{\lambda'}\rangle = 0$
if $\lambda \neq \lambda'$ or $\mu \neq \mu'$;
\item
$\langle xy,z\rangle = \langle y, x^* z\rangle$;
\item 
$
\langle e_{i_d} \cdots e_{i_1} 1_\lambda, e_{j_d} \cdots e_{j_1} 1_\lambda\rangle
=
(\theta_{i_1} \cdots \theta_{i_d}, \theta_{j_1} \cdots
\theta_{j_d})$.
\end{enumerate}
Moreover:
\begin{enumerate}
\item[(4)]
$\langle x,y\rangle = \langle \psi(y), \psi(x)\rangle$;
\item[(5)]
$\langle x,yz\rangle = \langle y^!x, z\rangle$.
\end{enumerate}
Assuming in addition that the bar-consistency assumption of
\cite[Definition 2.1(d)]{Clark}
holds, i.e. 
\begin{equation}\label{barconsistency}
d_i \equiv |i|\pmod{2}\quad\text{ for each $i \in I$},
\end{equation} 
the form
$\langle-,-\rangle$ is non-degenerate.
\end{theorem}

\begin{proof}
There is clearly at most one sesquilinear form on $\dot U_{q,\pi}(\g)$
satisfying properties (1)--(3).
To see that there is indeed such a form, we appeal to
\cite[Proposition 5.8]{Clark}, which defines a bilinear form
$(-,-)'$ on $\dot U_{q,\pi}(\g)$ satisfying four
properties. 
Our form $\langle-,-\rangle$ is obtained from Clark's form by setting
\begin{equation}
\langle x,y\rangle := (\sigma(\psi(u)),
\sigma(v))',
\end{equation}
where $\psi$ is the antilinear automorphism from (\ref{psi})
and $\sigma$ is the linear antiautomorphism defined by
declaring that $\sigma(1_\lambda) = 1_\lambda$, $\sigma(e_i 1_\lambda)
= 1_\lambda f_i$ and
$\sigma(f_i 1_\lambda) = 1_\lambda e_i$.
We leave it as an exercise to the reader to check that Clark's four
properties translate into our properties (1)--(4); actually, one needs
the opposite formulation of Clark's second property which may be
derived from
\cite[Proposition 5.3]{Clark},
noting that Clark's $\tau_1$ is our $\sigma \circ * \circ \psi\circ \sigma$.
The property (5) is immediate from (2) and (4) plus the definition of
(\ref{marry2}).
Finally, assuming bar-consistency, the non-degeneracy follows from
\cite[Theorem 5.12]{Clark}.
\end{proof}

\begin{remark}
One could also define a bilinear (rather than sesquilinear) form $(-,-)$ on $\dot
U_{q,\pi}(\g)$
by setting
$(x,y) := \langle\psi(x), y \rangle$. This is a generalization of
Lusztig's form from \cite[Theorem 26.1.2]{Lubook} which is slightly
different from the one introduced in \cite{Clark}.
Theorem~\ref{clark} implies that $(-,-)$
is symmetric and it satisfies $(xy,z) = (y,\rho(x) z)$.
\end{remark}

The next theorem gives a graphical description of the form
$\langle-,-\rangle$ in the spirit of \cite[Theorem 2.7]{KL3}.
Recall the notation $\SSeq$ from section \ref{s6}.
For $\a,\b \in \SSeq$,
let $\widehat{M}(\a,\b)$ be chosen as in
Theorem~\ref{tripleofdata}.
For $\sigma \in \widehat{M}(\a,\b)$ and $\lambda \in P$, define the
{\em degree} $\deg(\sigma,\lambda)$ and the {\em parity} $|\sigma,\lambda|$
to be the degree and parity of the homogeneous 2-morphism
$f(\sigma,\lambda)$, i.e. we sum
the degrees and parities of all of the generating dots, cups, caps and crossings 
in the diagram for $f(\sigma,\lambda)$ as listed in the following table:
$$
\begin{array}{|c|c|c||c|c|c|}
\hline
\text{Generator}&\text{Degree}&\text{Parity}&\text{Generator}&\text{Degree}&\text{Parity}\\\hline
\mathord{
\begin{tikzpicture}[baseline = 0]
	\draw[->,thick,darkred] (0.08,-.3) to (0.08,.4);
      \node at (0.08,0.05) {$\color{darkred}\bullet$};
   \node at (0.08,-.4) {$\scriptstyle{i}$};
\end{tikzpicture}
}
{\scriptstyle\lambda}
&2d_i&|i|&
\mathord{
\begin{tikzpicture}[baseline = 0]
	\draw[<-,thick,darkred] (0.08,-.3) to (0.08,.4);
      \node at (0.08,0.1) {$\color{darkred}\bullet$};
     \node at (0.08,.5) {$\scriptstyle{i}$};
\end{tikzpicture}
}
{\scriptstyle\lambda}
&2d_i&|i|\\
\mathord{
\begin{tikzpicture}[baseline = 0]
	\draw[->,thick,darkred] (0.28,-.3) to (-0.28,.4);
	\draw[->,thick,darkred] (-0.28,-.3) to (0.28,.4);
   \node at (-0.28,-.4) {$\scriptstyle{i}$};
   \node at (0.28,-.4) {$\scriptstyle{j}$};
   \node at (.4,.05) {$\scriptstyle{\lambda}$};
\end{tikzpicture}
}
&-(\alpha_i,\alpha_j)&|i||j|&
\mathord{
\begin{tikzpicture}[baseline = 0]
	\draw[<-,thick,darkred] (0.28,-.3) to (-0.28,.4);
	\draw[->,thick,darkred] (-0.28,-.3) to (0.28,.4);
   \node at (-0.28,-.4) {$\scriptstyle{i}$};
   \node at (-0.28,.5) {$\scriptstyle{j}$};
   \node at (.4,.05) {$\scriptstyle{\lambda}$};
\end{tikzpicture}
}&0&|i||j|\\
\mathord{
\begin{tikzpicture}[baseline = 0]
	\draw[<-,thick,darkred] (0.28,-.3) to (-0.28,.4);
	\draw[<-,thick,darkred] (-0.28,-.3) to (0.28,.4);
   \node at (-0.28,.5) {$\scriptstyle{i}$};
   \node at (0.28,.5) {$\scriptstyle{j}$};
   \node at (.4,.05) {$\scriptstyle{\lambda}$};
\end{tikzpicture}
}&-(\alpha_i,\alpha_j)&|i||j|&
\mathord{
\begin{tikzpicture}[baseline = 0]
	\draw[->,thick,darkred] (0.28,-.3) to (-0.28,.4);
	\draw[<-,thick,darkred] (-0.28,-.3) to (0.28,.4);
   \node at (0.28,-.4) {$\scriptstyle{i}$};
   \node at (0.28,.5) {$\scriptstyle{j}$};
   \node at (.4,.05) {$\scriptstyle{\lambda}$};
\end{tikzpicture}
}&0&|i||j|\\
\mathord{
\begin{tikzpicture}[baseline = 0]
	\draw[<-,thick,darkred] (0.4,0.3) to[out=-90, in=0] (0.1,-0.1);
	\draw[-,thick,darkred] (0.1,-0.1) to[out = 180, in = -90] (-0.2,0.3);
    \node at (-0.2,.4) {$\scriptstyle{i}$};
  \node at (0.3,-0.15) {$\scriptstyle{\lambda}$};
\end{tikzpicture}
}
&d_i(1+\langle h_i,\lambda\rangle)&\0&
\mathord{
\begin{tikzpicture}[baseline = 0]
	\draw[-,thick,darkred] (0.4,0.3) to[out=-90, in=0] (0.1,-0.1);
	\draw[->,thick,darkred] (0.1,-0.1) to[out = 180, in = -90] (-0.2,0.3);
    \node at (0.4,.4) {$\scriptstyle{i}$};
  \node at (0.3,-0.15) {$\scriptstyle{\lambda}$};
\end{tikzpicture}
}
&d_i(1-\langle h_i,\lambda\rangle)&|i,\lambda|\\
\mathord{
\begin{tikzpicture}[baseline = 0]
	\draw[<-,thick,darkred] (0.4,-0.1) to[out=90, in=0] (0.1,0.3);
	\draw[-,thick,darkred] (0.1,0.3) to[out = 180, in = 90] (-0.2,-0.1);
    \node at (-0.2,-.2) {$\scriptstyle{i}$};
  \node at (0.3,0.4) {$\scriptstyle{\lambda}$};
\end{tikzpicture}
}&d_i(1-\langle h_i,\lambda\rangle)&\0&
\mathord{
\begin{tikzpicture}[baseline = 0]
	\draw[-,thick,darkred] (0.4,-0.1) to[out=90, in=0] (0.1,0.3);
	\draw[->,thick,darkred] (0.1,0.3) to[out = 180, in = 90] (-0.2,-0.1);
    \node at (0.4,-.2) {$\scriptstyle{i}$};
  \node at (0.3,0.4) {$\scriptstyle{\lambda}$};
\end{tikzpicture}
}&d_i(1+\langle h_i,\lambda\rangle)&|i,\lambda|\\
\hline
\end{array}
$$
Just as we did in (\ref{safetypin}), a word $\a = \a_m\cdots\a_1 \in
\SSeq$ defines a monomial
\begin{equation}
e_\a 1_\lambda
:= e_{\a_m} \cdots e_{\a_1} 1_\lambda \in \dot U_{q,\pi}(\g),
\end{equation}
where 
$e_{{\color{darkred}{\uparrow}}_i}:=e_i$
and $e_{{\color{darkred}{\downarrow}}_i}:=f_i$.
Clearly, these monomials taken over all $\a \in \SSeq$ and all $\lambda
\in P$
span $\dot U_{q,\pi}(\g)$.

\begin{theorem}\label{sesqui}
The sesquilinear form $\langle-,-\rangle$ from Theorem~\ref{clark}
satisfies 
\begin{equation}\label{bagels}
\langle e_{\a} 1_\lambda, e_{\b} 1_\mu \rangle
= \delta_{\lambda,\mu} 
\sum_{\sigma \in \widehat{M}(\a,\b)}
q^{\deg(\sigma,\lambda)} \pi^{|\sigma,\lambda|} 
\end{equation}
for each $\a,\b \in \SSeq$ and $\lambda,\mu
\in P$.
\end{theorem}

\begin{proof}
This argument parallels the proof of \cite[Theorem 2.7]{KL3} closely.
We can clearly assume $\mu=\lambda$.
Let $\langle \a,\b\rangle_\lambda$ 
denote the expression on the right hand side of (\ref{bagels}). Note to start with that $\langle\a,\b\rangle_\lambda$
does not depend on the particular choice made for
$\widehat{M}(\a,\b)$. This follows because one can pass between any
two choices of decorated reduced matchings by a sequence of isotopies
which do not change degrees or
parities of diagrams. (This is similar to the proof of
Theorem~\ref{tripleofdata}, which applied more complicated relations which are the
same as these isotopies plus terms with fewer crossings.)
To complete the proof of the theorem, we must show:
\begin{equation}\label{justin}
\langle e_{\a} 1_\lambda, e_{\b} 1_\lambda \rangle
= \langle \a,\b\rangle_\lambda.
\end{equation}
We proceed with a series of claims, which mimic \cite[Lemmas
2.8--2.12]{KL3}.

\vspace{1mm}

\vspace{1mm}
\noindent
{\em \underline{Claim 1}.}
The identity (\ref{justin}) is true in case $\a$ and $\b$ are
positive, i.e. they only involve upward arrows. 

\vspace{1mm}
\noindent
To see this, if
$\a = {\color{darkred}\uparrow}_{i_c} \cdots {\color{darkred}\uparrow}_{i_1}$ and
$\b = {\color{darkred}\uparrow}_{j_d} \cdots {\color{darkred}\uparrow}_{j_1}$, then
$M(\a,\b)$ is empty unless $c=d$, in which case its elements are in
bijection with permutations $w \in S_d$
such that $i_{w(r)} = j_r$ for each $r=1,\dots,d$, and we have that
$$
\langle \a,\b\rangle_\lambda
=
\delta_{c,d}\sum_{w \in S_d}
\bigg(\prod_{r=1}^d\frac{ \delta_{i_{w(r)}, j_{r}} }{1-\pi_{i_r} q_{i_r}^2}\bigg)\bigg(
\prod_{\substack{1 \leq r < s \leq d\\w(r) > w(s)}}
\!\!\!\pi^{|i_r||i_s|} q^{-(\alpha_{i_r},\alpha_{i_s})}\bigg).
$$
Using Theorem~\ref{clark}(iv), it remains to check that this
equals
$(\theta_{i_1}\cdots \theta_{i_c}, \theta_{j_1},\dots,\theta_{j_d})$.
This follows by the explicit definition of the latter form on $\mathbf{f}$.

\vspace{1mm}

\noindent
{\em \underline{Claim 2}.} $\langle e_i e_{\a} 1_\lambda, e_{\b} 1_\lambda \rangle = \langle
{\color{darkred}\uparrow}_i \a, \b\rangle_\lambda\Leftrightarrow
\langle e_{\a} 1_\lambda, f_i e_{\b} 1_\lambda \rangle = \langle
\a, {\color{darkred}\downarrow}_i \b \rangle_\lambda$.

\noindent
{\em \underline{Claim 3}.}
$\langle f_i e_{\a} 1_\lambda, e_{\b} 1_\lambda \rangle = \langle
{\color{darkred}\downarrow}_i\a, \b \rangle_\lambda\Leftrightarrow
\langle e_{\a} 1_\lambda, e_i e_{\b} 1_\lambda \rangle = \langle
\a, {\color{darkred}\uparrow}_i\b \rangle_\lambda$.

\vspace{1mm}
\noindent
The proofs of these are the same as for \cite[Lemma 2.9]{KL3}.
For example, for Claim 2, one considers the bijection between
$\widehat{M}({\color{darkred}\uparrow}_i\a, \b)$
and $\widehat{M}(\a,{\color{darkred}\downarrow}_i\b)$ 
obtained by attaching a cup on the bottom left.
On the algebraic side, one uses Theorem~\ref{clark}(2) and (\ref{marry1}).

\vspace{1mm}

\noindent
{\em {\underline{Claim 4}}}.
$\langle e_{\a} e_i f_j e_\b 1_\lambda, 1_\lambda\rangle =
\langle \a {\color{darkred}\uparrow}_i {\color{darkred}\downarrow}_j \b, \varnothing\rangle_\lambda
\Leftrightarrow
\langle e_{\a} f_j e_i e_\b 1_\lambda, 1_\lambda\rangle =
\langle \a {\color{darkred}\downarrow}_j {\color{darkred}\uparrow}_i \b, \varnothing\rangle_\lambda$,
assuming $i \neq j$.

\vspace{1mm}
\noindent
Since 
$\langle e_{\a} e_i f_j e_\b 1_\lambda, 1_\lambda\rangle
=
\pi^{|i||j|}\langle e_{\a} f_j e_i e_\b 1_\lambda, 1_\lambda\rangle$
by (\ref{cwr1}), we must show that
$$
\langle \a {\color{darkred}\uparrow}_i {\color{darkred}\downarrow}_j \b, \varnothing\rangle_\lambda
= \pi^{|i||j|}
\langle \a {\color{darkred}\downarrow}_j {\color{darkred}\uparrow}_i \b, \varnothing\rangle_\lambda.
$$
This follows by considering the bijection between
$\widehat{M}(\a {\color{darkred}\uparrow}_i {\color{darkred}\downarrow}_j \b, \varnothing)$ and $\widehat{M}( \a {\color{darkred}\downarrow}_j {\color{darkred}\uparrow}_i \b, \varnothing)$ obtained 
attaching a rightward crossing under the
${\color{darkred}\uparrow}_i {\color{darkred}\downarrow}_j$ to convert it to ${\color{darkred}\downarrow}_j
{\color{darkred}\uparrow}_i$; see the proof of \cite[Lemma 2.11]{KL3} for further
explanations. The only difference for us is that the crossing is odd in case $|i||j| = \1$.

\vspace{1mm}
\noindent
{\em {\underline{Claim 5}}}.
Assuming that $\langle e_\a e_\b 1_\lambda, 1_\lambda \rangle = \langle \a\b,\varnothing\rangle_\lambda$,
we have that 
$\langle e_{\a} e_i f_i e_\b 1_\lambda, 1_\lambda\rangle =
\langle \a {\color{darkred}\uparrow}_i {\color{darkred}\downarrow}_i \b, \varnothing\rangle_\lambda
\Leftrightarrow
\langle e_{\a} f_i e_i e_\b 1_\lambda, 1_\lambda\rangle =
\langle \a {\color{darkred}\downarrow}_i {\color{darkred}\uparrow}_i \b, \varnothing\rangle_\lambda$.

\vspace{1mm}
\noindent
Define $\mu$ so that $e_\b 1_\lambda = 1_\mu e_\b$.
In view of (\ref{cwr1}), we must show that
\begin{equation}\label{classes}
\langle \a {\color{darkred}\uparrow}_i {\color{darkred}\downarrow}_i \b, \varnothing\rangle_\lambda
- \pi^{|i|}
\langle \a {\color{darkred}\downarrow}_i {\color{darkred}\uparrow}_i \b, \varnothing\rangle_\lambda
= [\langle h_i,\mu \rangle]_{q_i,\pi_i} \langle
\a\b,\varnothing\rangle_\lambda.
\end{equation}
To see this, we divide the 
decorated matchings in $\widehat{M}(\a {\color{darkred}\uparrow}_i {\color{darkred}\downarrow}_i \b, \varnothing)$ and
$\widehat{M}(\a {\color{darkred}\downarrow}_i {\color{darkred}\uparrow}_i \b, \varnothing)$ into three classes exactly as explained in the proof
of \cite[Lemma 2.12]{KL3}. It is then easy to see that the
contributions to the left hand side of (\ref{classes}) from the first two
classes cancel. The third classes arise from decorated matchings in
$\widehat{M}(\a\b)$
by inserting a cap (clockwise or counterclockwise in the two cases)
between $\a$ and $\b$.
Hence, like in the proof of \cite[Lemma 2.12]{KL3} remembering also the sesquilinearity of $\langle-,-\rangle$, we see that the left hand side of
(\ref{classes}) expands to
$$
\overline{\left[q_i^{1-\langle h_i,\mu\rangle} / (1-\pi_i q_i^2) - \pi_i \pi^{|i,\mu|}
    q_i^{1+\langle h_i,\mu\rangle} / (1-\pi_i q_i^2)\right]}
\langle \a\b,\varnothing\rangle_\lambda.
$$
This simplifies to the right hand side of (\ref{classes}).

\vspace{2mm}
Now we can complete the proof of (\ref{justin}) in general.
Using Claims 2 and 3, we reduce to checking (\ref{justin}) in the
special case that $\b = \varnothing$. Under this assumption, we then
proceed by induction on the length of $\a$. Using Claims 4 and 5 plus
the induction hypothesis, we can rearrange $\a$ to assume that all
$\color{darkred}\downarrow$'s appear to the left of all $\color{darkred}\uparrow$'s. Then we use
Claims 3 and 1 to finish the proof.
\end{proof}

\begin{example}[{cf. \cite[Example 5.7]{Clark}}]
\begin{align*}
\langle e_i^{(r)} 1_\lambda, e_i^{(r)} 1_\lambda \rangle =
\langle f_i^{(r)}
1_\lambda, f_i^{(r)} 1_\lambda \rangle &=\prod_{s=1}^r
\frac{1}{1-(\pi_i q_i^2)^s},\\
\langle e_i f_i 1_\lambda, 1_\lambda \rangle 
= \pi^{|i,\lambda|}\langle 1_\lambda, e_i f_i 1_\lambda \rangle &=
\frac{q_i^{1-\langle h_i,\lambda\rangle}}{1-\pi_i q_i^2},\\
\langle e_i f_i 1_\lambda, f_i e_i 1_\lambda \rangle
=
\langle f_i e_i 1_\lambda, e_i f_i 1_\lambda \rangle&=
\frac{\pi_i+q_i^2}{(1-\pi_i q_i^2)^2}.
\end{align*}
\end{example}

\section{Surjectivity of $\gamma$}\label{sonto}

In this section, we continue with the assumptions of ~\S\ref{sqg},
and also assume that $\k = \k_\0$ is a field.
For a graded superalgebra $A$, we write $A\GSMod$ for the Abelian category of
graded left $A$-supermodules with morphisms that preserve degree and
parity.
Let $Q$ and $\Pi$ denote the grading and parity shift functors on $A\GSMod$, so that
$(Q V)_n = V_{n-1}$ and $(\Pi V)_a = V_{a+\1}$.
Let $A\GSProj$ be the full subcategory of $A\GSMod$ consisting of the
finitely generated projective supermodules. 
Let $K_0(A)$ denote the
split Grothendieck group of 
$A\GSProj$. It is naturally an $\LC$-module with $q$ and $\pi$ acting
by $[Q]$ and $[\Pi]$, respectively.
For a detailed discussion of the following basic facts, we refer
the reader to
\cite[$\S\S$3.8.1--3.8.2]{KL3}, all of which is
easily extended to the case of supermodules.
\begin{itemize}
\item
Assume the graded superalgebra $A$ is {\em Laurentian}, i.e. its graded pieces
are finite-dimensional and are zero in sufficiently negative degree.
Then, the Krull-Schmidt property holds in $A\GSProj$. Moreover,
$K_0(A)$ is free as an $\LC$-module, with basis as a free $\Z$-module
given by
the isomorphism classes of indecomposable projectives in $A\GSProj$.
\item
If $\alpha:A \rightarrow B$ is a homomorphism of graded superalgebras,
there is an induced $\LC$-module homomorphism
$[\alpha]:K_0(A) \rightarrow K_0(B)$. 
If $A$ and $B$ are
finite-dimensional and $\alpha$ is surjective, then $[\alpha]$ is
surjective.
\item 
Assume $A$ is Laurentian,
and let
$I$ be a two-sided
homogeneous ideal that is non-zero only in strictly positive degree.
Then, 
the canonical quotient map $A
\twoheadrightarrow A / I$ induces an isomorphism $K_0(A)
\stackrel{\sim}{\rightarrow} K_0(A / I)$.
\item
If $A$ and $B$ are finite-dimensional graded superalgebras
all of whose irreducible graded 
supermodules are absolutely irreducible of type $\mathtt{M}$,
then 
there
is an isomorphism
$K_0(A) \otimes_{\LC} K_0(B) \stackrel{\sim}{\rightarrow} K_0(A \otimes
B),
[P] \otimes [Q] \mapsto [P \otimes Q]$.
\end{itemize}
For more background about $K_0$ for supercategories, see \cite[$\S$1.5]{BE}.

We also need to review some basic facts about quiver Hecke
superalgebras established in \cite{KL1, KL2} in the even case, and in
\cite{HW} in general. Note in \cite{HW} that
the additional assumption (\ref{barconsistency}) of bar-consistency is made
throughout, but it is not needed for the proofs of the particular
results from \cite{HW} cited below.

The {\em quiver Hecke supercategory} $\H$ is the (strict) monoidal
supercategory generated by objects $I$ and morphisms
$\mathord{
\begin{tikzpicture}[baseline = -2]
	\draw[->,thick,darkred] (0.08,-.15) to (0.08,.3);
      \node at (0.08,0.05) {$\color{darkred}\bullet$};
   \node at (0.08,-.25) {$\scriptstyle{i}$};
\end{tikzpicture}
}:i \rightarrow i$ and
$\mathord{
\begin{tikzpicture}[baseline = -2]
	\draw[->,thick,darkred] (0.18,-.15) to (-0.18,.3);
	\draw[->,thick,darkred] (-0.18,-.15) to (0.18,.3);
   \node at (-0.18,-.25) {$\scriptstyle{i}$};
   \node at (0.18,-.25) {$\scriptstyle{j}$};
\end{tikzpicture}
}:i \otimes j \rightarrow j \otimes i$
of parities $|i|$ and $|i||j|$, respectively, 
subject to the relations (\ref{now})--(\ref{qhalast}) (omitting the label
$\lambda$ from these diagrams).
For objects $\bi = i_n \otimes\cdots \otimes i_1 \in I^{\otimes n}$ and $\bj = j_m \otimes\cdots
\otimes j_1 \in I^{\otimes m}$, there are no non-zero morphisms
$\bi \rightarrow \bj$ in $\H$ 
unless $m=n$.
The graded endomorphism superalgebra
\begin{equation}
H_n := \bigoplus_{\bi, \bj \in I^{\otimes n}} \Hom_{\H}(\bi, \bj)
\end{equation}
is the {\em quiver Hecke superalgebra} from \cite{KKT}.
Let $\H_{q,\pi}$ be the $(Q,\Pi)$-envelope of the monoidal
supercategory $\H$, which is defined like in Definition~\ref{ernie}
remembering that monoidal supercategories are 2-supercategories with
one object; see also \cite[Definition
1.16]{BE}.
Let $\underline{\H}_{q,\pi}$ be the underlying monoidal category (same objects,
even morphisms of degree zero).
The idempotent completion 
of the additive envelope of $\underline{\H}_{q,\pi}$ is denoted 
$\dot{\underline{\H}}_{q,\pi}$ as usual.
It is equivalent to the category $\bigoplus_{n \geq 0} H_n\GSProj$,
hence,
we may identify
\begin{equation}\label{crap}
K_0(\dot{\underline{\H}}_{q,\pi}) =\bigoplus_{n \geq 0} K_0(H_n).
\end{equation}
In particular, this means that the $\LC$-module on the right hand side
of (\ref{crap})
is actually an $\LC$-algebra; its multiplication comes from the usual
induction product $-\circ-$ on graded $H_n$-supermodules.

Fix $i \in I$ and consider the idempotent $1_{i^n} := 1_{i\otimes
  i\otimes \cdots \otimes i}
\in H_n$. The graded subalgebra $1_{i^n} H_n 1_{i^n}$ is a copy of the {\em nil-Hecke algebra} in case $|i| = \0$, or the
{\em odd nil-Hecke algebra} in case $|i|=\1$.
In either case, we write simply $X_r$ for the dot on the $r$th strand
and $T_r$ for the crossing of the $r$th and $(r+1)$th strands
(numbering strands by $1,\dots,n$ from right to left).
The elements $D_r := - T_r X_r$ from \cite[(5.20)]{HW} are
homogeneous idempotents
which 
satisfy the braid relations of the
symmetric group $S_n$. Hence, for each $w \in S_n$ there is an element
$D_w$ defined as usual from a reduced expression for $w$.
Letting $w_0$ be the longest element of $S_n$, we define
\begin{equation}
1_{i^{(n)}} := D_{w_0} \in 1_{i^n} H_n 1_{i^n}.
\end{equation}
This is known to be a primitive homogeneous idempotent, hence, 
\begin{equation}\label{crappymovie}
P(i^{(n)}) := Q^{-d_i n(n-1)/2}\:H_n 1_{i^{(n)}}
\end{equation}
is an indecomposable projective graded $H_n$-supermodule.

\begin{lemma}\label{necks}
There is a graded supermodule isomorphism
$H_n 1_{i^n} \cong P(i^{(n)})^{\oplus [n]^!_{q_i, \pi_i}}$
(meaning the obvious direct sum of copies of $
P(i^{(n)})$ with parity and degree shifts matching the expansion of
$[n]^!_{q_i,\pi_i}$).
\end{lemma}

\begin{proof}
This is well known in the even case, and is noted after
\cite[(5.28)]{HW} in the odd case.
A different convention for $(q,\pi)$-integers is adopted in \cite{HW},
which we have taken into account by changing the parity shift in
(\ref{crappymovie}) 
compared to \cite[(5.28)]{HW}.
\end{proof}

Next suppose that we are given two different elements $i,j \in I$.
For $r, s \geq 0$, the tensor product in $\H$ 
gives a superalgebra embedding
$H_r \otimes H_1 \otimes H_s \hookrightarrow H_{r+s+1}$.
Let $1_{i^{(r)} j i^{(s)}}$ denote the image of $1_{i^{(r)}} \otimes
1_j \otimes 1_{i^{(s)}}$ under this map, then set
\begin{equation}
P(i^{(r)} j i^{(s)}) := Q^{-d_i r(r-1)/2-d_i s(s-1)/2} \: H_{r+s+1} 1_{i^{(r)} j i^{(s)}}.
\end{equation}
In other words, $P(i^{(r)} j i^{(s)}) = P(i^{(r)}) \circ P(j) \circ P(i^{(s)})$.
This is a graded projective $H_{r+s+1}$-supermodule.

\begin{proposition}[Khovanov-Lauda, Rouquier, Hill-Wang]\label{qserre}
For $i \neq j \in I$, let $n := d_{ij}+1$. Then there exists a split
exact sequence of graded $H_{r+s+1}$-supermodules
\begin{multline*}
0 
\longrightarrow P(i^{(n)} j)  \longrightarrow
\cdots \longrightarrow
\Pi^{\frac{r(r-1)}{2}|i| + r |i||j|} P(i^{(n-r)} j i^{(r)})
\longrightarrow\cdots\\
\longrightarrow \Pi^{\frac{n(n-1)}{2}|i| + n |i||j|} P(j i^{(n)}) \longrightarrow 0.
\end{multline*}
In particular, there is an isomorphism
$$
\bigoplus_{k=0}^{\lfloor \frac{n+1}{2}\rfloor}
\Pi^{k|i|} P(i^{(n-2k)} j i^{(2k)})
\cong
\bigoplus_{k=0}^{\lfloor \frac{n}{2}\rfloor}
\Pi^{k|i| + |i||j|} P(i^{(n-2k-1)} j i^{(2k+1)}).
$$
\end{proposition}

\begin{proof}
See \cite[Theorem 5.9]{HW}.
\end{proof}

Recall the $\LL$-algebra $\mathbf{f}$ defined at the
beginning of section \ref{sform}.
Let $\mathbf{f}_{\LC}$
be the $\LC$-subalgebra generated by the divided powers $\theta_i^{(n)}
:= \theta_i^n / [n]_{q_i, \pi_i}^!$ for all $i \in I$ and $n \geq 1$.
Using Lemma~\ref{necks} and Proposition~\ref{qserre}, it follows that
there is a unique $\LC$-algebra homomorphism
\begin{equation}\label{hwm}
\bar\gamma:\mathbf{f}_{\LC} \rightarrow 
\bigoplus_{n \geq 0} K_0(H_n),
\qquad \theta_i^{(n)} \mapsto [P(i^{(n)})].
\end{equation}

\begin{theorem}[Khovanov-Lauda, Hill-Wang]\label{betacat}
The homomorphism $\bar\gamma$ from (\ref{hwm}) is an isomorphism.
\end{theorem}

\begin{proof}
See \cite[Theorem 6.14]{HW}.
\end{proof}

\begin{corollary}\label{jor}
Every irreducible graded $H_n$-supermodule is absolutely irreducible
of type $\mathtt{M}$.
\end{corollary}

\begin{proof}
The absolute irreducibility follows from Theorem~\ref{betacat}; see the proof
of \cite[Corollary 3.19]{KL1}. They are all
of type $\mathtt{M}$ by \cite[Proposition 6.15]{HW}.
\end{proof}

Now we are going upgrade some of these results to $\UU(\g)$.
For each $\lambda \in P$, there is a graded superalgebra homomorphism
\begin{equation}\label{titan1}
\alpha_{n,\lambda}:H_n \rightarrow \bigoplus_{\bi,\bj \in I^{\otimes n}}
\Hom_{\UU(\g)}(E_\bi 1_\lambda, E_\bj 1_\lambda),
\end{equation}
where for $\bi = i_n \otimes \cdots \otimes i_1$ we write 
$E_\bi 1_\lambda$ for $E_{i_n} \cdots E_{i_1} 1_\lambda$.
In diagrammatic terms, $\alpha_{n,\lambda}$ takes the string diagram for
an element of $H_n$ to the 2-morphism whose diagram is obtained by adding the label
$\lambda$ on the right hand edge.
Applying this to $1_{i^{(n)}}$, we obtain the homogeneous idempotent
$\alpha_{n,\lambda}\left(1_{i^{(n)}}\right) \in \End_{\UU(\g)}(E_i^n 1_\lambda)$.
Then define the {\em divided power}
$E_i^{(n)} 1_\lambda$ to be the 1-morphism in the idempotent
completion
$\dot{\underline{\UU}}_{q,\pi}(\g)$
associated to the idempotent
$\big(\alpha_{n,\lambda}\left(1_{i^{(n)}}\right)\big)_{0,\0}^{0,\0}$
in the $(Q,\Pi)$-envelope.
Composing with the isomorphism $\omega$ from
Proposition~\ref{opiso}, we get also a graded superalgebra homomorphism
\begin{equation}\label{titan2}
\alpha'_{n,\lambda}:= \omega\circ\alpha_{n,\lambda}:H^{\operatorname{sop}}_n \rightarrow \bigoplus_{\bi,\bj \in I^{\otimes n}}
\Hom_{\U(\g)}(F_\bi 1_\lambda, F_\bj 1_\lambda),
\end{equation}
where $F_\bi 1_\lambda:= F_{i_n} \cdots F_{i_1} 1_\lambda$.
Let $F_i^{(n)} 1_\lambda$ be the 1-morphism in
$\dot{\underline{\UU}}_{q,\pi}(\g)$
associated to the idempotent
$\big(\alpha'_{n,\lambda}\left(1_{i^{(n)}}\right)\big)_{0,\0}^{0,\0}$.

\begin{lemma}\label{was}
In $K_0(\dot{\underline{\UU}}_{q,\pi}(\g))$, we have that
$[Q^0 \Pi^\0 E_i^{n} 1_\lambda] = [n]_{q_i, \pi_i}^! [E_i^{(n)} 1_\lambda]$ and
$[Q^0 \Pi^\0 F_i^{n} 1_\lambda] = [n]_{q_i, \pi_i}^! [F_i^{(n)} 1_\lambda]$.
\end{lemma}

\begin{proof}
This follows from the definitions and Lemma~\ref{necks}.
To give some more detail, Lemma~\ref{necks} means that the idempotent $1_{i^n}
\in H_n$ splits as a sum of $n!$ idempotents, each of which
is conjugate via some unit in $H_n$ to $1_{(i^n)}$. These units are
homogeneous of various degrees and parities encoded in the
$(q,\pi)$-factorial
$[n]_{q_i,\pi_i}^!$.
When we apply the homomorphism $\alpha_{n,\lambda}$ to this
decomposition, 
we deduce that the 2-morphism 
$1_{E_i^n\lambda}$ splits as a sum of $n!$ idempotents, each of which
is conjugate by some homogeneous unit in $\End_{\U(\g)}(E_i^n
1_\lambda)$ to
$\alpha_{n,\lambda}\left(1_{i^{(n)}}\right)$.
Passing to $\dot{\underline{\UU}}_{q,\pi}(\g)$,
we get from this an
isomorphism $Q^0 \Pi^\0 E_i^n 1_\lambda \stackrel{\sim}{\rightarrow}
E_i^{(n)} 1_\lambda^{\oplus [n]_{q_i, \pi_i}}$ by
taking the direct sum of these units appropriately shifted so that
they become even of degree zero. 
\end{proof}

\begin{lemma}\label{relcheck}
In $K_0(\dot{\underline{\UU}}_{q,\pi}(\g))$, we have that
\begin{align*}
[Q^0 \Pi^\0 E_i F_j 1_\lambda]  -
 [Q^0 \Pi^{|i||j|} F_j E_i 1_\lambda]&=
\delta_{i,j} [\langle
h_i,\lambda\rangle]_{q_i,\pi_i} [1_\lambda],\!\!\!\!\!\!\!\!\!\!\!\!\!\!\!\!\!\\
\sum_{r=0}^{d_{ij}+1} (-1)^r \pi_i^{r|j| + r(r-1)/2} [E_i^{(d_{ij}+1-r)} E_j^{(1)}
E_i^{(r)} 1_\lambda] &= 0&(i \neq j),\\
\sum_{r=0}^{d_{ij}+1} (-1)^r \pi_i^{r|j| + r(r-1)/2} [F_i^{(d_{ij}+1-r)} F_j^{(1)}
F_i^{(r)} 1_\lambda] &= 0&(i \neq j).
\end{align*}
\end{lemma}

\begin{proof}
The first identity follows from the inversion relations
(\ref{inv1})--(\ref{inv3}).
For example, to prove it in the case $i=j$ and $\langle h_i,\lambda
\rangle \leq 0$, we use (\ref{inv3}) to see that there is an
isomorphism
in $\dot{\underline{\U}}_{q,\pi}(\g)$
$$
Q^0 \Pi^{|i|} E_i F_i 1_\lambda
\oplus \bigoplus_{n=0}^{-\langle h_i, \lambda \rangle-1}
Q^{d_i(-\langle h_i,\lambda\rangle-1-2n)} \Pi^{n|i|} 1_\lambda
\stackrel{\sim}{\rightarrow}
Q^0 \Pi^\0 F_i E_i 1_\lambda.
$$
Since 
$[\langle h_i,\lambda\rangle]_{q_i,\pi_i} = 
-\pi_i \sum_{n=0}^{-\langle h_i,\lambda\rangle-1} q_i^{-\langle
  h_i,\lambda\rangle-1-2n} \pi_i^n$,
this gives what we need on
 passing to the Grothendieck group.

The second two identities 
are consequences of Proposition~\ref{qserre}.
One needs to interpret the isomorphism there first in terms of
idempotents, then apply the homomorphisms 
$\alpha_{n+1,\lambda}$
and
$\alpha_{n+1,\lambda}'$.
\end{proof}

\begin{theorem}\label{bits}
There is a unique surjective $\LC$-algebra homomorphism
$$
\gamma:\dot U_{q,\pi}(\g)_{\mathcal L} \twoheadrightarrow
K_0(\underline{\dot\UU}_{q,\pi}(\g))
$$
sending $1_\lambda, e_i^{(n)}1_\lambda$
and $f_i^{(n)} 1_\lambda$ to $[1_\lambda]$,
$[E_i^{(n)} 1_\lambda]$ and $[F_i^{(n)} 1_\lambda]$, respectively.
\end{theorem}

\begin{proof}
To establish the existence of the homomorphism $\gamma$,
note to start with that
there 
is an $\LL$-algebra homomorphism 
$\dot U_{q,\pi}(\g) \rightarrow \LL \otimes_{\LC} 
K_0(\underline{\dot\UU}_{q,\pi}(\g))$ 
sending $1_\lambda, e_i^{(r)}1_\lambda$
and $f_i^{(r)} 1_\lambda$ to $[1_\lambda]$,
$[E_i^{(r)} 1_\lambda]$ and $[F_i^{(r)} 1_\lambda]$, respectively.
To see this, we just have to check the defining relations of $\dot
U_{q,\pi}(\g)$ from (\ref{cwr1})--(\ref{cwr3}), which follow by Lemma~\ref{relcheck}.
Then we restrict this homomorphism to $\dot U_{q,\pi}(\g)_{\LC}$,
observing that the image of the restriction lies in 
$K_0(\underline{\dot\UU}_{q,\pi}(\g))$ thanks to Lemma~\ref{was}.

It remains to prove that $\gamma$ is surjective.
The proof of this is essentially the same as the proof in the purely
even case given in \cite[$\S$3.8]{KL3}, so we will try to be brief.
For $n,n' \geq 0$ and $\lambda \in P$, we let
$$
H_{n,n',\lambda} := \bigoplus_{\substack{\bi,\bj \in I^{\otimes n} \\ \bi',\bj'
  \in I^{\otimes n'}}} \Hom_{\UU(\g)}(E_\bi F_{\bi'}
1_\lambda, E_\bj F_{\bj'} 1_\lambda).
$$
Idempotents in this algebra are idempotent 2-morphisms in $\UU(\g)$,
hence, there is a canonical homomorphism $$
\delta_{n,n',\lambda}:K_0(H_{n,n',\lambda})
\rightarrow
K_0(\dot{\underline{\UU}}_{q,\pi}(\g)).
$$
Moreover, there is an $\LC$-algebra homomorphism
$$
\alpha_{n,n',\lambda}:H_n \otimes H_{n'}^{\operatorname{sop}} \otimes \SYM
\rightarrow H_{n,n',\lambda}
$$
sending $a \otimes a' \otimes p$ to $\alpha_{n,\mu}(a)
\alpha_{n',\lambda}'(a') \beta_\lambda(p)$, where 
$\mu$ is
the weight labeling the left hand edge of  the diagram 
$\alpha_{n',\lambda}'(a') 1_\lambda$.
Let $I_{n,n',\lambda}$ be the two-sided ideal of
$H_{n,n',\lambda}$ spanned by all string diagrams which involve a
U-turn, i.e. they involve at least one arc whose endpoints are both on
the top edge; cf. \cite[Proposition 3.17]{KL3}.
Let $$
\beta_{n,n',\lambda}:H_{n,n',\lambda}\twoheadrightarrow
H_{n,n',\lambda} / I_{n,n',\lambda}
$$ 
be the canonical quotient map.
The composition $\gamma_{n,n',\lambda} := \beta_{n,n',\lambda}\circ\alpha_{n,n',\lambda}$ is
surjective.
We get induced a commutative diagram at the level of Grothendieck groups:
$$
\begin{tikzcd}
K_0(H_n \otimes H_{n'}^{\operatorname{sop}} \otimes \SYM)
\arrow[two heads]{rr}{[\gamma_{n,n',\lambda}]}
\arrow[swap]{rd}{[\alpha_{n,n',\lambda}]}
&&
K_0\left(H_{n,n',\lambda} / I_{n,n',\lambda}\right) \\
&
K_0(H_{n,n',\lambda}) 
\arrow[two heads,swap]{ur}{[\beta_{n,n',\lambda}]}
\end{tikzcd}.
$$
Following the proof of \cite[Proposition 3.36]{KL3}, using the facts
summarized at the start of this section plus the fact that $H_n$ is
finite as a module over its center, one shows that
$[\gamma_{n,n',\lambda}]$ is onto, hence, so too is $[\beta_{n,n',\lambda}]$.

Now let $X$ be an indecomposable object in
$\dot{\underline{\UU}}_{q,\pi}(\g)$.
Define its {\em width} to be the smallest $N \geq 0$ such that
$X$ is isomorphic 
to a summand of $Q^m \Pi^b E_\a 1_\lambda$ for some $\a \in \SSeq$
of length $N$ and some $m \in \Z, b \in \Z/2$ and $\lambda \in P$.
We are going to show by induction on width that
each $[X]$ is in the image of $\gamma$.
For the base case, if $X$ is of width zero, 
we claim that 
it is 
isomorphic to some $Q^m \Pi^b 1_\lambda$. 
To see this, recall that
$\End_{\UU(\g)}(1_\lambda)$ is a quotient of $\SYM$, which is strictly
postively graded with $\k$ in degree zero. Hence, $1_\lambda$ is
either indecomposable or zero, which implies our claim.
Since $[1_\lambda]$ is in the image of $\gamma$, the base of the
induction is now established.

For the induction step, take $X$ of width $N > 0$.
We can find some $n,n' \geq 0$ with $n+n' = N$
and $\bi \in I^{\otimes n}, \bi' \in I^{\otimes n'}$
such that $X$ is isomorphic 
to a summand of $Q^m \Pi^b E_\bi F_{\bi'} 1_\lambda$.
This is a consequence of the relations (\ref{inv1})--(\ref{inv3}); cf. the
proof of \cite[Lemma 3.38]{KL3}.
It follows that $[X]$ is in the image of $\delta_{n,n',\lambda}$,
i.e. there is some 
$Y \in H_{n,n',\lambda}\GSProj$
such that $\delta_{n,n',\lambda}([Y]) = [X]$.
The minimality in the definition of width ensures that
$\beta_{n,n',\lambda}([Y]) \neq 0$.
Pick $Z \in H_n \otimes H_{n'}^{\operatorname{sop}} \otimes
\SYM\GSProj$ such that $[\gamma_{n,n',\lambda}]([Z]) = [\beta_{n,n',\lambda}]([Y])$.
Then one argues explicitly with idempotents as in
\cite[$\S$3.8.4]{KL3}
to see that 
$$
[\alpha_{n,n',\lambda}]([Z])= [Y] + [Y']
$$
for $Y' \in H_{n,n',\lambda}\GSProj$ with
$[\beta_{n,n',\lambda}]([Y']) = 0$.
By induction, $\delta_{n,n',\lambda}([Y'])$ is in the image of
$\gamma$. Hence, to show that $[X] = \delta_{n,n',\lambda}([Y])$ is
so, we are reduced to showing that
$\delta_{n,n',\lambda}([\alpha_{n,n',\lambda}]([Z]))$ is in the image of $\gamma$.
This follows using the following commutative diagram:
$$
\begin{tikzcd}
&\mathbf{f}_\LC \otimes_\LC \mathbf{f}_\LC
\arrow[swap]{dl}{i_\lambda}
\\
\dot U_{q,\pi}(\g)_\LC
\arrow[swap]{dd}{\gamma}
&
&K_0(H_n) \otimes_\LC K_0(H_{n'}) 
\arrow[swap]{ul}{\bar\gamma^{-1}\otimes\bar\gamma^{-1}}
\\\\
K_0(\dot{\underline{\UU}}_{q,\pi}(\g))
&&K_0(H_n \otimes H_{n'}^{\operatorname{sop}}\otimes\SYM)
\arrow{dl}{[\alpha_{n,n',\lambda}]}
\arrow[swap]{uu}{j_{n,n'}}
\\
&K_0(H_{n,n',\lambda})
\arrow{ul}{\delta_{n,n',\lambda}}
\end{tikzcd}.
$$
 Here, $\bar\gamma$ is the isomorphism from Theorem~\ref{betacat}.
the isomorphism $j_{n,n'}$ exists because of Corollary~\ref{jor}, and
$i_\lambda$ sends 
$\theta_{i_1} \cdots \theta_{i_n} \otimes
\theta_{j_1}\cdots \theta_{j_m} \mapsto 
e_{i_1}\cdots e_{i_n} f_{j_1}\cdots f_{j_m} 1_\lambda$
\end{proof}

\section{The decategorification conjecture}\label{s4}

We continue to assume the homogeneity condition (\ref{hc}) holds and
that $\k = \k_\0$ is a field.
Let us restate the Decategorification Conjecture from the introduction:

\vspace{2mm}
\noindent
{\bf Decategorification Conjecture.}
{\em 
The surjective homomorphism
$\gamma$
from Theorem~\ref{bits} is an isomorphism.}

\vspace{2mm}

The proof of the following theorem mimics \cite[$\S$3.9]{KL3}.

\begin{theorem}\label{injectivity}
Assume that the Nondegeneracy Conjecture holds and moreover that the
Cartan datum is bar-consistent, i.e. (\ref{barconsistency}) holds. Then the
Decategorification Conjecture holds as well.
\end{theorem}

\begin{proof}
For a graded superspace $V$, we let $\dim_{q,\pi} V := \sum_{n \in \Z}
\sum_{a \in \Z/2} (\dim V_{n,a}) q^n \pi^a$. For example,
viewing the algebra $\SYM$ from (\ref{SYM}) as a graded superalgebra
so that the isomorphism 
(\ref{beta}) preserves degrees and parities, we have that
$$
S := \dim_{q,\pi} \SYM = \prod_{i \in I} \prod_{r \geq 1}
\frac{1}{1-(\pi_i q_i^2)^r}
\in \Z[[q]][\pi] / (\pi^2-1).
$$
The Nondegeneracy Conjecture implies (indeed, is equivalent to) the
assertion that
\begin{equation}\label{dum}
\langle e_\a 1_\lambda, e_\b 1_\lambda \rangle 
=
S^{-1} \dim_{q,\pi} \Hom_{\UU(\g)}(E_\a 1_\lambda, E_\b 1_\lambda)
\end{equation}
for $\a, \b \in \SSeq$ 
with $\wt(\a) = \wt(\b)$ and
$\lambda \in P$.

Now consider the sesquilinear form on $K_0(\underline{\dot
  \UU}_{q,\pi}(\g))$ defined 
by letting $\langle [X],[Y] \rangle$ be zero if $X, Y$ are 1-morphisms
in $\underline{\dot \UU}_{q,\pi}(\g)$
whose domains or codomains are different, and setting
$$
\langle [X], [Y] \rangle := 
S^{-1} \sum_{n \in \Z}\sum_{a \in \Z/2} \dim \Hom_{\underline{\dot
    \UU}_{q,\pi}(\g)}(Q^n \Pi^a X,Y) q^n \pi^a
$$
if $X$ and $Y$ have the same domain and codomain.
Equivalently,
for 1-morphisms $X, Y:\lambda\rightarrow \mu$
in $\UU_{q,\pi}(\g)$, we have that 
$$
\langle [X], [Y] \rangle =
S^{-1} \dim_{q,\pi} \Hom_{\UU_{q,\pi}(\g)}(X,Y).
$$
Comparing with (\ref{dum}), using also Theorem~\ref{sesqui}, 
we deduce that the forms $\langle-,-\rangle$
on $\dot U_{q,\pi}(\g)_{\LC}$ and
$K_0(\underline{\dot
  \UU}_{q,\pi}(\g))$ 
are intertwined by the homomorphism $\gamma$ in the sense that
$\langle x,y \rangle = \langle \gamma(x), \gamma(y)\rangle$.

Finally, suppose that $x \in \dot U_{q,\pi}(\g)_{\LC}$
is in the kernel of $\gamma$. By the previous paragraph, we 
have that $\langle x, y \rangle = 0$ for all $y \in \dot U_{q,\pi}(\g)_{\LC}$.
In view of the non-degeneracy of the form
$\langle-,-\rangle$ from Theorem~\ref{clark},
this implies that $x = 0$.
\end{proof}

\begin{remark}
The assumption of bar-consistency made in both of 
Theorems~\ref{clark} and \ref{injectivity} is probably
unnecessary. We have included it because we have appealed to \cite[Theorem 5.12]{Clark}, where
it is assumed from the outset. Providing one allows that the canonical
basis should be bar-invariant only up to multiplication by
$\pi$, we expect that the arguments of \cite{Clark} should still be
valid without bar-consistency, but we have not checked this assertion
in detail.
\end{remark}

\begin{example}
Take $\g$ to be odd $\mathfrak{b}_1$ and identify
$P$ with $\Z$ as in the introduction.
Then,
\cite[Proposition 8.3]{EL} implies
that the indecomposable 1-morphisms in $\underline{\dot\UU}_{q,\pi}(\g)$
(up to degree and parity
shift) are 
$$
\{ E^{(a)} F^{(b)} 1_\lambda\:|\:a,b\geq 0,\lambda \in \Z, \lambda \leq b-a\}
\cup
\{F^{(b)} E^{(a)} 1_\lambda\:|\:a,b\geq 0,\lambda \in \Z, \lambda \geq b-a\}.
$$
Also by \cite[Theorem 8.4]{EL}, the Decategorification Conjecture
holds in this case, i.e. $\gamma$ is an isomorphism.
As has already been noted in \cite[Example 4.16]{Clark}, $\gamma$ maps
the classes of the indecomposable 1-morphisms listed above to
the canonical basis for $\dot U_{q,\pi}(\g)$
from \cite[Theorem 6.2]{CW}
(up to multiplication by $\pi$).
\end{example}

\end{document}